\documentclass[11pt,twoside,reqno,a4paper]{amsart}
\usepackage[pdftex]{hyperref}
\usepackage{amsmath,amsthm,amsopn,amssymb,a4wide,times,varioref}
\usepackage{enumitem}
\usepackage[mathscr]{eucal}
\usepackage[utopia]{mathdesign}
\usepackage{tikz}
\usepackage[numbers,sort&compress]{natbib}
\newtheorem{theorem}{Theorem}[section]
\newtheorem{lemma}[theorem]{Lemma}
\newtheorem{corollary}[theorem]{Corollary}
\newtheorem{proposition}[theorem]{Proposition}
\theoremstyle{definition}

\newtheorem{definition}[theorem]{Definition}

\numberwithin{equation}{section}
\setlist{label={$($\roman{enumi}\kern1pt$)$}}
\newcommand{\ip}[2]{{
    \left<
      #1,#2
    \right>}}

\newcommand{\HL}{{H^\infty(X,{\mathcal K}_\Lambda)}}

\newcommand{\HLH}{{H^\infty(X,{\mathcal K}_{\Lambda,\mathcal H})}}

\newcommand{\HLb}{{H_1^\infty(X,{\mathcal K}_\Lambda)}}

\newcommand{\HLHb}{{H_1^\infty(X,{\mathcal K}_{\Lambda,\mathcal H})}}

\newcommand{\AL}{{A(X,{\mathcal K}_\Lambda)}}

\newcommand{\ALH}{{A(X,{\mathcal K}_{\Lambda,\mathcal H})}}

\newcommand{\ran}{{\mathrm{ran}\,}}

\newcommand{\clran}{{\overline{\mathrm{ran}}\,}}

\title{Realizations via Preorderings with Application to the Schur Class}

\dedicatory{In loving memory of Deborah Godsey}

\author[M.~A.~Dritschel]{Michael A. Dritschel}

\thanks{Part of this work was carried out during several visits to the
  Indian Institute of Science in Bangalore.  The author gives
  special thanks to Tirtha Bhattacharyya and Gadadhar Misra for their
  warm hospitality and stimulating conversations.}

\address{School of Mathematics and Statistics\\
  Herschel Building\\ Newcastle University\\
  Newcastle upon Tyne\\
  NE1 7RU\\
  UK}

\email{michael.dritschel@newcastle.ac.uk}



\begin{document}

\begin{abstract}
  We extend Agler's notion of a function algebra defined in terms of
  test functions to include products, in analogy with the practice in
  real algebraic geometry, and hence the term preordering in the
  title.  This is done over abstract sets and no additional property,
  such as analyticity, is assumed.  Realization theorems give several
  equivalent ways of characterizing the unit ball (referred to as the
  Schur-Agler class) of the function algebras.  These typically
  include, in Agler's terminology, a model (here called an Agler
  decomposition), a transfer function representation, and an analogue
  of the von~Neumann inequality.  The new ingredient is a certain set
  of matrix valued functions termed ``auxiliary test functions'' used
  in constructing transfer functions.  In important cases, the
  realization theorems can be strengthened so as to allow applications
  to Pick type interpolation problems, among other things.  Principle
  examples have as the domain the polydisk $\mathbb D^d$.  The
  algebras then include $H^\infty(\mathbb D^d,\mathcal{L(H)})$ (where
  the unit ball is traditionally called the Schur class) and
  $A(\mathbb D^d,\mathcal{L(H)})$, the multivariable analogue of the
  disk algebra.  As an application, it is shown that over the polydisk
  $\mathbb D^d$, (weakly continuous) representations which are
  $2^{d-2}$ contractive are completely contractive (hence having a
  commuting unitary dilation), offering fresh insight into such
  examples as Parrott's of contractive representations of $A(\mathbb
  D^3)$ which are not completely contractive.
\end{abstract}

\date{\today}

\maketitle

\section{Introduction}
\label{sec:introduction}

The classical realization theorem gives a variety of characterizations
of those functions which are in the Schur class over the unit disk
$\mathbb D$ in the complex plane $\mathbb C$; that is, those functions
in the closed unit ball of $H^\infty(\mathbb D)$.

Jim Agler found a method for extending this result to the polydisk
$\mathbb D^d$ \cite{MR1207393}, though for dimension $d$ greater than
$2$, one must use a different norm than the $H^\infty$ norm over an
algebra of functions which potentially may be a proper subalgebra of
$H^\infty(\mathbb D^d)$.  The unit ball for such an algebra is now
commonly known as the Schur-Agler class; the term Schur class usually
being reserved for the unit ball of $H^\infty(X)$ when $X$ is a domain
in $\mathbb C^d$.  Among other things, the realization theorem states
that a complex function $\varphi$ on the polydisk is in the
Schur-Agler class if and only if it has a so-called Agler
decomposition, expressing $1-\varphi\varphi^*$ as an element of a cone
generated by products of certain positive kernels and kernels of the
form $1-\psi\psi^*$, where $\psi$ is a coordinate function.  Other
equivalent conditions for membership in the Schur-Agler class include
the existence of a transfer function representation and a von~Neumann
type inequality for suitably restricted tuples of commuting
contractions.  The equivalence of all of these conditions makes no a
priori assumptions about the function $\varphi$, and it is this which
enables the use of the realization theorem in such applications as
Pick interpolation.

These results have been vastly generalized, in the spirit of Agler's
work (see, for example,~\cite{MR3057412, MR2742657, MR2069781,
  MR1937424, MR3223895, MR2337640, MR2595740, MR2597682, MR1797710,
  MR2899979, MR2745478}).  In particular, for a planar domain $X$ with
smooth boundary components, the necessary machinery has been developed
to enable Agler's methods to be applied to the study of $H^\infty(X)$,
and as a result, the settling of long-standing problems in operator
theory related to these domains~\cite{MR2375060, MR2163865, MR2389623,
  MR2643788}.  When $X$ is the polydisk $\mathbb D^d$, $d\geq 3$, far
less is known, though important initial steps have been
taken~\cite{MR2502431, MR2861551}.  A central goal of this paper is to
give realization theorems for $H^\infty(\mathbb D^d)$, and to apply it
to elucidating certain phenomena observed in the study of commuting
tuples of contraction operators.  Following~\cite{MR2389623}, this is
done abstractly, with the results for the polydisk constituting a
special example.  Indeed, the results cover a wide range of domains
and function algebras on these domains, while at the same time
demonstrating that assumptions of such properties as analyticity are
in general unnecessary.

The Agler decomposition has its analogues in real algebraic geometry.
For example, if we have a set in $\mathbb R^n$ consisting of those
points at which a finite collection of polynomials is non-negative (a
so-called basic semi-algebraic set), and these polynomials also
include $1-\psi_i^2$ where each $\psi_i$ is a constant multiples of a
coordinate function, then Putinar's theorem~\cite{MR1254128} (see
also~\cite{MR1829790}) states that a strictly positive polynomial is
in the quadratic module generated by these polynomials; that is, it is
in the cone generated by finite sums of squares of polynomials times
the individual polynomials defining the semi-algebraic set.  If
however the polynomials $1-\psi_i^2$ are not necessarily included, the
statement of Putinar's theorem is in general false, even if the
semi-algebraic set is assumed to be compact~\cite{MR1829790}.  However
the situation can be salvaged in the compact setting by replacing the
quadratic module by a preordering; that is, by considering the cone
generated by finite sums of squares of polynomials multiplied by the
various products of the polynomials defining the semi-algebraic set.
This is the content of Schm\"udgen's theorem~\cite{MR1092173}.
Further refinements are possible.  For example, if only two
polynomials define the compact semi-algebraic set then one can get by
with the quadratic module in Schm\"udgen's theorem (\cite[Corollary
6.3.7]{MR1829790}), which because of And\^o's theorem is analogous to
what happens in the complex case with Agler's realization theorem.

Back in the complex function setting, work of Grinshpan,
Kaliuzhnyi-Verbovet\-skyi, Vinnikov and Woerdeman~\cite{MR2502431}
shows that on the polydisk for dimension greater than $2$, one can
recover the entire Schur class by using the appropriate variant of a
preordering (see also Knese~\cite{MR2861551}).  The caveat is that
they find it necessary to assume that the function they are
considering is already known to be in the Schur class, and so there is
no direct application to Pick interpolation in the Schur class.

Another hurdle to using the results in \cite{MR2502431} for
interpolation is that the crucial transfer function representation is
absent, though they do prove that a form of the von~Neumann inequality
is available.  A particularly interesting aspect of~\cite{MR2502431}
is that the tuples of operators the authors are considering have a
unitary dilation, obtained by showing that they induce a completely
contractive representation of $H^\infty(\mathbb D^d)$ and then
applying the standard machinery.  There are many papers which consider
the problem of determining conditions under which a tuple of commuting
contractions has a unitary dilation, including those of Ball, Li,
Timotin and Trent~\cite{MR1722812} and Archer~\cite{MR2308555}, which
prove a multivariable form of the commutant lifting theorem.

This paper has several goals.  The first is to place the work of
Grinshpan, Kaliuzhnyi-Verbovet\-skyi, Vinnikov and Woerdeman in the
context of test functions on a set $X$, in this way allowing for a
much broader class of function algebras.  For example, there will be
analogues, written $\HLH$, of the algebra $H^\infty(\mathbb D)$.  The
set $X$ can be topologized and closed in an appropriate norm, which
allows us to make sense of the analogues $\ALH$ of the disk algebra
$A(\mathbb D)$, in this context; that is, elements $\HLH$ which extend
continuously to the closure of $X$.  It is noteworthy that there are
\textit{a priori} no assumptions made on the set $X$ or on the set of
test functions (such as analyticity).

To begin with, a careful examination of the continuity properties of
elements of $\HLH$ and $\ALH$ is carried out.  Following this, we
introduce the auxiliary test functions.  In contrast to the original
test functions, which are taken to be scalar valued, these are matrix
valued functions.  Moreover, in the setting of the so-called standard
ample preordering, the auxiliary test functions can be taken to be
functions in matrix valued $\ALH$, and the Schur-Agler class
corresponding to these functions is the unit ball of $\HLH$.  As a
consequence, we are then able to give a full version of the
realization theorem, including the transfer function representation
and analogues of the von Neumann inequality.  Importantly, none of the
realization theorems requires the assumption that the function under
consideration is already in $\HLH$.  Thus in principle, in the ample
case such applications as Agler-Pick interpolation are possible.

Even if the preordering is not ample, we show that elements of our
generalized Schur-Agler class have a transfer functions
representation, though it is not clear that everything with a transfer
function representation is in our algebra except in the ample and the
classical settings.  However, the transfer functions with values in
$\mathcal{L(H)}$ do form the unit ball of an algebra having a natural
matrix norm structure, and so form an operator algebra.  This nicely
complements work in~\cite{MR2595740}, where it is shown that a
collection of analytic (potentially matrix valued) test functions over
a domain in $\mathbb C^n$ generate an operator algebra, and that a
transfer function representation exists for the functions in this
algebra; that is, such algebras are examples of transfer function
algebras.

We are able to show that for transfer function algebras, certain types
of representations (the so-called Brehmer representations over the
analogue of the disk algebra and the weakly continuous Brehmer
representations over the analogue of $H^\infty$) are completely
contractive, implying the existence of a dilation of such a
representation to something akin to a boundary representation (though
without the assumption of irreducibility).  This includes those
representations which are contractive on the auxiliary test functions
in the ample setting when we know these functions are in a matrix
version of $\AL$, meaning that such representations are also
completely contractive.  As a consequence, any representation which is
$n$-contractive for appropriate $n$ (depending only on the number of
test functions) will be completely contractive for $\AL$.  Curiously,
the condition of being a Brehmer representation does not obviously
imply that the representation is contractive on auxiliary test
functions, though this is ultimately an outcome of the realization
theorems.

As mentioned above, the case when $X$ is the polydisk is of particular
interest.  Then the ample preordering gives $H^\infty(\mathbb D^d,
\mathcal{L(H)})$.  Since the auxiliary test functions are not given
constructively, determining if a representation is contractive on
these is in general difficult, but as mentioned above, $n$ contractive
representations will be completely contractive if $n$ is sufficiently
large.  In the classical setting of Agler's realization theorem for
the polydisk the auxiliary test functions are simply the test
functions, and by our definition, the Brehmer representations are in
this case just those representations mapping the coordinate functions
to commuting contractions.  Such representations are also completely
contractive on $A(\mathbb D^d)$ with respect to the appropriate matrix
norm structure (something which can also be gleaned from results
in~\cite{MR2595740}).  At first sight, this might seem paradoxical
given Parrott's example of a commuting triple of contractions without
a unitary dilation.  However since the matrix norm structure is not
that of $H^\infty$ with the supremum norm, there is in fact no
problem.  Indeed, we show that there are choices in the Parrott
example which give rise to a boundary representation (in the sense of
Arveson), since it will be irreducible and not only will there be no
commuting unitary dilation of the image of the coordinate functions,
but in fact the only commuting contractive dilation is by means of a
direct sum.  Several other matrix valued boundary representations are
also explicitly given, one arising from and example of Grinshpan,
Kaliuzhnyi-Verbovet\-skyi and Woerdeman~\cite{MR3057417}, and another
constructed from the Kaijser-Varopoulos example.  It is also not
difficult to see that the Crabbe-Davie example also gives a boundary
representation.  All of these send the coordinate functions to
nilpotent matrices, though it can be shown that there must exist
examples which are neither commuting unitaries nor unitarily
equivalent to commuting nilpotents.

\goodbreak

Finally, we demonstrate that in the setting of ample preorderings,
And\^o's theorem allows us to instead consider so-called nearly ample
preorderings.  With this we are able to recover the full extent of the
results of Grinshpan, Kaliuzhnyi-Verbovet\-skyi, Vinnikov and
Woerdeman,and even generalize it, and at the same time improve the
result mentioned in the previous paragraph by proving that when $d\geq
2$, for $n=2^{d-2}$, $n$-contractive weakly continuous representations
of $\HLH$ and $n$-contractive representations of $\ALH$ are completely
contractive.  In particular, over the polydisk the images of the
coordinate functions under such a representation will be commuting
contractions with a commuting unitary dilation.  Viewed another way,
any example such as Parrot's of a representation of $A(\mathbb D^3)$
which is contractive but not completely contractive must fail to be
$2$-contractive.

\goodbreak

\section{Test functions, preorderings, function spaces and topology}
\label{sec:test-fns-preord-top}

\subsection{Test functions and preorderings}
\label{subsec:test-funct-preord}

Let $X$ be a set, $\mathcal H$ a Hilbert space, and $\Psi$ a
collection of $\mathcal{L(H)}$ valued functions on $X$.  We call
$\Psi$ a set of \textbf{test functions} if for $x\in X$,
$\sup_{\psi\in\Psi} \|\psi(x)\| < 1$, and when restricted to any
finite set, $\Psi$ generates all functions on that set (equivalently
in the scalar valued case we are considering, $\Psi$ separates the
points of $X$).  \textit{We assume that the test functions we are
  dealing with are complex valued, though we later construct certain
  matrix valued test functions from these}.  

There are many interesting situations where the collection of test
functions is infinite~\cite{MR2389623}.  \textit{However, for this
  paper our focus will solely be on the situation when $d = |\Psi| :=
  \mathrm{card}\,\Psi$ is finite}, though many of the initial results
are valid in any case.  This assumption has the advantage of allowing
us to, among other things, avoid additional complexities in the proof
of the realization theorems, since when $|\Psi|$ is finite certain
representations in which we will be interested have a particularly
simple form.

We use standard tuple notation on $\bigoplus_1^d \mathbb N$,
$d$-tuples of non-negative integers $(n_i)$, endowed with the partial
ordering $n' \leq n$ if and only if $n'_i \leq n_i$ for all $i$.  If
$n = (n_i) \in \bigoplus_1^d \mathbb N$, we write $|n|$ for the sum of
the entries of $n$.  Also, we denote by $e_i$ the tuple with all
entries except the $i$th equal to zero, while the $i$th is $1$, and
$0$ for the tuple where all entries are zero, and $1$ will stand for
the tuple with all entries equal to $1$.  We use the notation
$\psi^n$ to stand for $\textstyle\prod \psi_i^{n_i}$, where the
product is over the $n_i \in n$ which are nonzero.

By a \textbf{preordering} we mean a finite set $\Lambda \subset
\bigoplus_1^d \mathbb N$ with the property that for all $i$, $e_i <
\lambda$ for some $\lambda\in\Lambda$.  This is at variance with the
usual definition from real algebraic geometry, but happens to be more
convenient in our context.  The connection with the standard form
should become apparent to those familiar with it.

We will see in the next section that for the applications we have in
mind, the preordering is not unique, and in fact there are two rather
special preorderings associated to any given preordering $\Lambda$.
The first is the \textbf{minimal preordering} $\Lambda_m$, which is
constructed from $\Lambda$ as the union of all $\lambda\in\Lambda$
such that if $\lambda'\in\Lambda$ and $\lambda \leq \lambda'$, then
$\lambda' = \lambda$.  In other words, the minimal preordering
consists of the union of the maximal elements $\Lambda$.  The other is
the \textbf{maximal preordering} $\Lambda_M := \{\lambda\in
\bigoplus_{\psi\in\Psi} \mathbb N : \lambda \leq \lambda' \text{ for
  some } \lambda'\in\Lambda\}$.  Hence if $\lambda' \in\Lambda$ and
$\lambda \leq \lambda'$, then $\lambda\in\Lambda_M$.

We find it convenient to decompose any maximal preordering $\Lambda_M$
as a disjoint union $\bigcup_{j=0}^\infty \Lambda_j$, where $\lambda
\in \Lambda_j$ if and only if $|\lambda| = j$.  Thus the only element
of $\Lambda_0$ is $0$, those in $\Lambda_1$ are the $e_i$s, and so on.
Set $\Lambda_+, \Lambda_-$ equal to the union over $\Lambda_j$s where
the index is even and odd, respectively.  Now for any $\lambda \in
\Lambda_M$, there are $2^{|\lambda|}$ elements $\Lambda_M$ which are
less than or equal to $\lambda$, half of which are in $\Lambda_+$ and
half in $\Lambda_-$.  For the purposes of fixing a clear labeling on
certain vectors later on, we use the ordering on $\Psi$ to endow
$\Lambda_M$ with the lexicographic ordering $\leq_\ell$.

Since $d = |\Psi| < \infty$, of particular interest will be the
so-called \textbf{ample preorderings}.  These are preorderings which
have a largest element; that is, a unique maximal element,
$\lambda^{m}$.  When $\lambda^m = 1$, we call the resulting
preordering a \textbf{standard ample preordering}.  Thus if $\Lambda$
is an ample preordering, the corresponding maximal preordering has the
form $\Lambda := \{\lambda\in \mathbb N^d : \lambda \leq \lambda^m
\}$.  A minimal ample preordering thus consists of a single element,
$\Lambda_m = \{\lambda^m\}$.

Let $\Lambda$ be ample with maximal element $\lambda^m$, and
$\lambda^1,\lambda^2 < \lambda^m$ where $\lambda > \lambda^1$ or
$\lambda^2$ implies $\lambda = \lambda^m$.  Then for $j=1,2$, $\lambda
= \lambda^j + e_{\ell_j}$ for some $e_{\ell_j}$, where the addition is
entrywise.  A preordering $\Lambda_s \subset \Lambda$ with the
property that $\lambda^1$ and $\lambda^2$ are maximal elements in
$\Lambda_s$ is termed a \textbf{nearly ample preordering} under
$\lambda^m$, and a \textbf{standard nearly ample preordering} when
$\lambda^m = 1$.

A simple (indeed, unique) example of a standard nearly ample
preordering when $d = |\Psi| =2$ is $\lambda^1 =(1,0)$, $\lambda^2 =
(0,1)$.  For $d=3$, there are three choices of minimal standard nearly
ample preordering, such as for example, $\lambda^1 =(1,1,0)$,
$\lambda^2 = (1,0,1)$.

\subsection{Kernels and function spaces}
\label{subsec:kern-funct-spac}

Write $\mathcal{L(H)}$ for the bounded linear operators on a Hilbert
space $\mathcal H$, $\mathcal{L(H,K)}$ for the bounded linear
operators mapping between Hilbert spaces $\mathcal H$ and $\mathcal
K$.

Let $\{\sigma_\lambda\}_{\lambda\in\Lambda}$ be a collection of
$n_\lambda\times n_\lambda$ matrix valued functions on $X$ such that
for each $x$,\break $\sup_{\lambda\in\Lambda}\|\sigma_\lambda(x)\| < 1$.
(These will later be the auxiliary test functions.)  Define bounded
functions $E_x$ on $\Lambda$ by $E_x(\lambda) = \sigma_\lambda(x)$.
We use the notation $C_b(\Lambda)$ for the unital $C^*$-algebra
generated by these functions.  This is a finite dimensional algebra of
dimension at most $\sum_\lambda n_\lambda$.  As such, it is isomorphic
to a direct sum of matrix algebras, and consequently, any
representation will be (isomorphic to) a direct sum of identity
representations applied to these matrix algebras.  More specifically,
for $\rho : C_b(\Lambda) \to \mathcal{L(E)}$, there will be orthogonal
projections $P_\lambda$ with orthogonal ranges such that $\mathcal E =
\bigoplus_\lambda \ran P_\lambda \otimes \mathbb C^{n_\lambda}$, and
$\rho(E_x) = \bigoplus_\lambda P_\lambda \otimes \sigma_\lambda(x)$.

Let $\mathcal A$ and $\mathcal B$ be $C^*$-algebras.  A kernel
$\Gamma:X\times X \to \mathcal{L}(A,B)$ is called \textbf{completely
  positive} if for all finite sets $\{x_1,\dots,x_n\} \subset X$,
$\{a_1,\dots,a_n\} \subset \mathcal A$ and $\{b_1,\dots,b_n\} \subset
\mathcal B$,
\begin{equation*}
  \sum_{i,j=1}^n \ip{\Gamma(x_i,x_j)(a_ia_j^*)b_i}{b_j} \geq 0.
\end{equation*}
A theorem due to Bhat, Barreto, Liebscher and
Skeide~\cite[Theorem~3.6]{MR2065240} shows that this is equivalent the
the condition that for finite sets $\{x_1,\dots,x_n\} \subset X$, the
matrix $(\Gamma(x_i,x_j))$ is a completely positive map from
$M_n(\mathcal A)$ to $M_n(\mathcal B)$, and that this is further
equivalent to the existence of a Kolmogorov decomposition for
$\Gamma$.  We state a special case of this suited to our purposes.

\begin{proposition}
  \label{prop:factorization}
  Let $\mathcal H$ be a Hilbert space.  The kernel $\Gamma:X\times X
  \to \mathcal{L}(C_b(\Lambda), \mathcal{L(H)})$ is $($completely$)$
  positive if and only if it has a \textbf{Kolmogorov decomposition};
  that is, there exists a Hilbert space $\mathcal E$, a function
  $\gamma: X \to \mathcal{L(E,H)}$ and a unital $*$-representation
  $\rho:C_b(\Lambda) \to \mathcal{L(E)}$ such that
  \begin{equation*}
    \Gamma(x,y)(fg^*) = \gamma(x)\rho(f) \rho(g)^* \gamma(y)^*
  \end{equation*}
  for all $f,g\in C_b(\Lambda)$.
\end{proposition}

In the case of kernels $k:X\times X \to \mathcal{L(H)}$, which
corresponds to replacing $C_b(\Lambda)$ by $\mathbb C$, it follows
from standard results on completely positive maps, that positivity
implies complete positivity.  The existence of a Kolmogorov
decomposition of positive operator valued kernels is originally due to
Mlak~\cite{MR0222712}.  We use the notation $\mathbb
K_X^+(C_b(\Lambda), \mathcal{L(H)})$ for the set of completely
positive kernels on $X\times X$ with values in $\mathcal L
(C_b(\Lambda), \mathcal{L(H)})$.

For a fixed preordering $\Lambda$, the collection of kernels
\begin{equation*}
  \begin{split}
    \mathcal K_{\Lambda,\mathcal H} := &\left\{ k :X\times X \to
      \mathcal{L(H)} : k\geq 0 \text{ and for each }
      \lambda \in \Lambda,
      \vphantom{\textstyle\prod_{\lambda\ni\lambda_i\neq 0}}\right. \\
    &\hphantom{k :X\times X \to \mathbb \mathcal{L(H)} : }
    \left.\textstyle\prod_{\lambda\ni\lambda_i\neq 0}
      ([1_{\mathcal{L(H)}}] - (\psi_i\otimes
      1_{\mathcal{L(H)}})(\psi_i^*\otimes
      1_{\mathcal{L(H)}}))^{\lambda_i}*k \geq 0\right\},
  \end{split}
\end{equation*}
are termed the \textbf{admissible kernels}.  Here the kernel
$[1_{\mathcal{L(H)}}]$ has all entries equal to $1_{\mathcal{L(H)}}$,
the identity operator on $\mathcal H$, ``$*$'' indicates the pointwise
or Schur product of kernels, and $([1_{\mathcal{L(H)}}] -
(\psi_i\otimes 1_{\mathcal{L(H)}})(\psi_i^*\otimes
1_{\mathcal{L(H)}}))^{\lambda_i}$ is the $\lambda_i$-fold Schur
product of $[1_{\mathcal{L(H)}}] - (\psi_i\otimes
1_{\mathcal{L(H)}})(\psi_i\otimes 1_{\mathcal{L(H)}})^*$.  In the
non-scalar case we interpret this Schur product as follows: for a
kernel $F$ over $X\times X$,
\begin{equation*}
  \left(([1_{\mathcal{L(H)}}] - (\psi_i\otimes
    1_{\mathcal{L(H)}})(\psi_i\otimes 1_{\mathcal{L(H)}})^* *
    F\right)(x,y) := F(x,y) - (\psi_i(x)\otimes
  1_{\mathcal{L(H)}})F(x,y)(\psi_i(y)\otimes 1_{\mathcal{L(H)}})^*.
\end{equation*}
More generally, if $F=ff^*$ and $G=gg^*$ are Kolmogorov decompositions
of two positive kernels over Hilbert spaces $\mathcal F$ and $\mathcal
G$, respectively, then
\begin{equation*}
  F*G(x,y) = (f(x)\otimes g(x))(f(y)\otimes g(y))^*.
\end{equation*}
It is clear that the resulting kernel is positive.

The kernels in $\mathcal K_{\Lambda,\mathcal H}$ are then used to
define the Banach algebra $\HLH$ consisting of those functions
$\varphi : X \to \mathcal{L(H)}$ for which there is a finite constant
$c\geq 0$ such that for all $k\in \mathcal K_{\Lambda,\mathcal H}$,
\begin{equation*}
  (c^2[1_{\mathcal{L(H)}}] - \varphi\varphi^*)*k \geq 0,
\end{equation*}
and $\|\varphi\|$ is defined to be the smallest such $c$.  We call the
resulting algebra the \textbf{Agler algebra} and the norm the
\textbf{Schur-Agler norm}.  Denote the unit ball in this norm by
$\HLHb$.  This is referred to as the \textbf{Schur-Agler class}.  In
case the Agler algebra is isometrically isomorphic to $H^\infty(X)$,
the unit ball is usually simply called the \textbf{Schur class}.  It
is not difficult to see that the function $1_X$ equaling
$1_{\mathcal{L(H)}}$ at all $x$ is in $\HLHb$ since
$[1_{\mathcal{L(H)}}] = 1_X1_X^*$.  If $\mathcal{L(H)} = \mathbb C$,
we write $\HL$ and $\HLb$ for $\HLH$ and $\HLHb$, respectively.

There are obvious modifications of these definitions which we will not
explicitly state in the case of matrix valued test functions.  For
this setting, it will still be the case that the Agler algebra norm
dominates the supremum norm.

For $\varphi \in \HLH$, we can also define a norm by
$\|\varphi\|_\infty := \sup_{x\in X} \|\varphi(x)\|$.  This will in
general be different from the norm defined above.  Furthermore, since
the kernel
\begin{equation*}
  k(y,z) =
  \begin{cases}
    1 & y=z; \\
    0 & \text{otherwise,}
  \end{cases}
\end{equation*}
is admissible, it is apparent that $\|\varphi\|_\infty \leq
\|\varphi\|$.

Two preorderings $\Lambda_1$ and $\Lambda_2$ are \textbf{equivalent
  preorderings} if for all Hilbert spaces $\mathcal H$, $\mathcal
K_{\Lambda_1,\mathcal H} = \mathcal K_{\Lambda_2,\mathcal H}$, and
consequently they generate the same Banach algebras.

\begin{lemma}
  \label{lem:preorderings_are_equivalent}
  Any preordering $\Lambda$ is equivalent to both its minimal
  preordering $\Lambda_m$ and its maximal preordering $\Lambda_M$.
\end{lemma}

\begin{proof}
  We prove the lemma when $\mathcal H = \mathbb C$, the other cases
  following in an identical manner.

  It is clear that $\mathcal K_{\Lambda_M} \subseteq \mathcal
  K_\Lambda \subseteq \mathcal K_{\Lambda_m}$, so it suffices to
  ascertain that if $k\in \mathcal K_{\Lambda_m}$ and $\lambda \in
  \Lambda_M$, then
  \begin{equation*}
    \prod_{\lambda\ni\lambda_i\neq 0} (1 -
    \psi_i\psi_i^*)^{\lambda_i}*k \geq 0.
  \end{equation*}

  Choose $\lambda' \in \Lambda_m$ such that $\lambda' \geq \lambda$.
  We may assume that $\lambda' \neq \lambda$, since otherwise there is
  nothing to show.  Hence there is some $i$ such that $p =
  \lambda'(i)- \lambda(i) > 0$.  The kernel $k_{\psi_i}$ with
  \begin{equation*}
    k_{\psi_i}(x,y) = (1-\psi_i(x)\psi_i(y)^*)^{-1} = \sum_n
    \psi_i^n(x)\psi_i^{n*}(y),
  \end{equation*}
  is positive on $X$.  The Schur product of positive kernels is
  positive, so if we set ${\tilde\lambda} = \lambda' - p
  e_{\lambda_i}$ (the arithmetic done in the standard way), we find
  that
  \begin{equation*}
    \prod_{\tilde\lambda\ni\lambda_i\neq 0} (1 -
    \psi_i\psi_i^*)^{\lambda_i}*k 
    = k_{\psi_i}^p * \prod_{{\tilde\lambda} \ni \lambda_i
      \neq 0} (1 - \psi_i\psi_i^*)^{\lambda_i}*k \geq 0.
  \end{equation*}
  Continuing through those $i$ such that $\lambda'(i) > \lambda(i)$,
  after a finite number of steps we achieve the desired result.
\end{proof}

\begin{lemma}
  \label{lem:preorderings_order_adm_kers}
  Given two preorderings $\Lambda_1$ and $\Lambda_2$, if for the
  corresponding maximal preorderings $\Lambda_{1,M} \subseteq
  \Lambda_{2,M}$, then $\mathcal K_{\Lambda_2,\mathcal H} \subseteq
  \mathcal K_{\Lambda_1,\mathcal H}$.  Consequently,
  $H^\infty(X,\mathcal K_{\Lambda_1,\mathcal H}) \subseteq
  H^\infty(X,\mathcal K_{\Lambda_2,\mathcal H})$, with the norm of any
  element in the first algebra greater than or equal to the value of
  the norm of that element in the second.
\end{lemma}

\begin{proof}
  If $\Lambda_{1,M} \subseteq \Lambda_{2,M}$ , then every element
  $\lambda_1 \in \Lambda_{1,M}$ is less than or equal to a maximal
  element $\lambda_2^M \in \Lambda_{2,M}$, and so the first statement
  follows by arguing as in the last lemma.  The other statements are
  immediate from the definitions of the admissible kernels and
  corresponding algebras.
\end{proof}

We say that a kernel $\tilde k$ is \textbf{subordinate to} another
kernel $k$ if there is a positive kernel $F$ such that $\tilde k =
k*F$.  It is clear that if $G$ is a difference of positive kernels
such that $G*k \geq 0$ and $\tilde k$ is subordinate to $k$, then
$G*\tilde k \geq 0$.  Hence if $k$ is an admissible kernel, any kernel
subordinate to $k$ is also admissible.

The admissible kernels are particularly simple when we are dealing
with standard ample preorderings, since they are all subordinate to a
single kernel.

\begin{lemma}
  \label{lem:adm_kernels_for_ample_po}
  Let $\Lambda$ be a standard ample preordering over $\Psi =
  \{\psi_1,\dots,\psi_d\}$.  Then every kernel in $\mathcal
  K_{\Lambda,\mathcal H}$ is subordinate to
  \begin{equation*}
    k_s(x,y) := \left( 1_{\mathcal{L(H)}}\otimes \prod_{j=1}^d
    (1-\psi_j(x)\psi_j(y)^*)^{-1} \right).
  \end{equation*}
\end{lemma}

\begin{proof}
  Obviously $k_s$ is an admissible kernel, since it is the inverse
  with respect to the Schur product of $\left(
    1_{\mathcal{L(H)}}\otimes \prod_{j=1}^d
    (1-\psi_j(x)\psi_j(y)^*)\right)$.  Hence if $k$ is an admissible
  kernel, so that $(\prod_{j=1}^d (1-(1_{\mathcal{L(H)}}\otimes
  \psi_j(x))(1_{\mathcal{L(H)}}\otimes \psi_j(y)^*)) k(x,y)) =
  (F(x,y)) \geq 0$, then $k$ is seen to be subordinate to $k_s$.
\end{proof}

The lemma implies that when $\Lambda$ is a standard ample preordering,
it suffices to check membership in $\HLH$ against the single kernel
$k_s$.  There is an obvious version of this for ample preorderings as
well, but since we will primarily be interested in the standard case,
we do not state it.

\begin{corollary}
  \label{cor:po_for_polydisk}
  For $X = \mathbb D^d$ with $\Psi = \{z_1,\dots,z_d\}$ the coordinate
  functions and $\Lambda$ the standard ample preordering, $\HLH =
  H^\infty(\mathbb D^d, \mathcal{L(H)})$, and all admissible kernels
  are subordinate to $k_s\otimes 1_{\mathcal{L(H)}}$, where $k_s$ is
  the Szeg\H{o} kernel
  \begin{equation*}
    k_s(z,w) = \prod_{i = 1}^d (1 - z_i w_i^*)^{-1}.
  \end{equation*}
\end{corollary}

\begin{proof}
  This follows from the observation that $\varphi$ is in the unit ball
  of $H^\infty(\mathbb D^d, \mathcal{L(H)})$ if and only if
  $([1_{\mathcal{L(H)}}] - \varphi\varphi^*)*(k_s\otimes
  1_{\mathcal{L(H)}}) \geq 0$, where $k_s = \prod_{j=1}^d
  (1-z_jz_j^*)^{-1}$ is the Szeg\H{o} kernel for the polydisk.
\end{proof}

\subsection{The realization theorem for the Schur class of the disk
  and Agler's generalization}
\label{subsec:real-theor-schur}

The now classical realization theorem is an amalgam of various
results, all characterizing the Schur class for the unit disk $\mathbb
D$ (that is the unit ball of $H^\infty(\mathbb D)$).  We state here
the operator valued generalization (see, for
example,~\cite{MR1637941}).

\begin{theorem}[Classical Realization Theorem]
  \label{thm:classical-realization}
  Let $\varphi : \mathbb D \to \mathcal{L(H)}$.  The following are
  equivalent:
  \begin{enumerate}
  \item[$($SC$\,)$] $([1_{\mathcal{L(H)}}]-\varphi \varphi^*)*k_s \geq
    0$, where $k_s(z,w) = (1-zw^*)^{-1}$ is the Szeg\H{o} kernel, or
    equivalently, $\varphi \in H_1^\infty(\mathbb D,\mathcal{L(H)})$;
    that is, $\varphi$ is in the Schur class;
  \item[$($AD$\,)$] There is a positive kernel $\Gamma:\mathbb D \times
    \mathbb D \to \mathcal{L(H)}$ such that $1_{\mathcal{L(H)}} -
    \varphi(z) \varphi(w)^* = \Gamma(z,w)(1-zw^*)$;
  \item[$($TF$\,)$] There is a Hilbert space $\mathcal E$ and a unitary
    operator $U = \begin{pmatrix} A & B \\ C & D \end{pmatrix}$ in
    $\mathcal{L}(\mathcal{E} \oplus \mathcal{H})$ such that
    \begin{equation*}
      \varphi(z) = D+Cz(I-Az)^{-1}B;
    \end{equation*}
  \item[$($vN$\,)$] For every $T\in \mathcal{L(K)}$, $\mathcal K$ a
    Hilbert space, with $\|T\| < 1$, $\|\varphi(T)\| \leq 1$.
  \end{enumerate}
\end{theorem}

The last item is a version of von~Neumann's inequality.  If
$\varphi\in A(\mathbb D,\mathcal{L(H)})$, the operator valued version
of the disk algebra, then we may instead simply assume that $\|T\|
\leq 1$ in von~Neumann's inequality.  We interpret $\varphi(T)$ as
$D+(C\otimes T)(I-(A\otimes T))^{-1}B$.  The third item is referred to
as a transfer function representation, and $(\mathcal H,U)$ is called
a unitary colligation.  The terminology comes from systems theory.
The second item is called the Agler decomposition.  In this case it is
a trivial restatement of the first item.  It becomes less trivial in
the next theorem, which in the scalar version is due to Jim
Agler~\cite{MR1207393} (see~\cite{MR1637941} for the operator valued
case).  We state it in terms of preorderings.

\begin{theorem}[Agler's Realization Theorem for the polydisk]
  \label{thm:Aglers-realization}
  Fix $d\in\mathbb N$, $\Lambda^o = \{e_j\}_{j=1}^d$ and let $\varphi
  : \mathbb D^d \to \mathcal{L(H)}$.  The following are equivalent:
  \begin{enumerate}
  \item[$($SC$\,)$] $([1_{\mathcal{L(H)}}]-\varphi \varphi^*)*k \geq 0$
    for all $k\in \mathcal K_{\Lambda^o,\mathcal H}$, or equivalently
    $\varphi \in \HLHb$; that is, $\varphi$ is in the Schur-Agler
    class;
  \item[$($AD$\,)$] There are positive kernels $\Gamma_j:\mathbb D
    \times \mathbb D \to \mathcal{L(H)}$, $j=1,\dots, d$, such that
    $1_{\mathcal{L(H)}} - \varphi(z) \varphi(w)^* =
    \sum_j\Gamma_j(z,w) (1 - z_ jw_j^*)$;
  \item[$($TF$\,)$] There is a Hilbert space $\mathcal E$ and a unitary
    operator $U = \begin{pmatrix} A & B \\ C & D \end{pmatrix}$ in
    $\mathcal{L}(\mathcal{E} \oplus \mathcal{H})$ such that for $z\in
    \mathbb D^d$,
    \begin{equation*}
      \varphi(z) = D+CZ(z)(I-AZ(z))^{-1}B,
    \end{equation*}
    where $Z(z) = \sum_j z_j P_j$ and $\sum_j P_j =
    1_{\mathcal{L(E)}}$;
  \item[$($vN$\,)$] For every $d$-tuple of commuting contractions $T =
    (T_1,\dots, T_d)$ with $T_j\in \mathcal{L(K)}$, $\mathcal K$ a
    Hilbert space, with $\|T_j\| < 1$, we have $\|\varphi(T)\| \leq
    1$.
  \end{enumerate}
\end{theorem}

We interpret $\varphi(T)$ in a similar manner as in the single
variable case.  As before, when $\varphi \in \ALH$, we may instead
simply assume $\|T_j\| \leq 1$ for all $j$.  Various examples,
including that of Kaijser and Varopoulos~\cite{MR0355642}, show that
when $d>2$, the Schur-Agler class is a strict subset of the unit ball
of $H^\infty(\mathbb D^d,\mathcal{L(H)})$.  On the other hand,
And\^o's theorem implies that when $d=2$, $\HLH$ and $H^\infty(\mathbb
D^d,\mathcal{L(H)})$ coincide.

One of the most useful aspects of the realization theorem is that it
allows us to do interpolation~\cite{MR1637941}.  Since it particularly
suits our needs, we state it in the setting of what is commonly known
as tangential Nevanlinna-Pick interpolation.

\begin{theorem}
  \label{thm:Agler-Pick-interp-polydisk}
  Fix $d\in\mathbb N$, $\Lambda^o = \{e_j\}_{j=1}^d$ and domain
  $\mathbb D^d$.  Let $\Omega$ be a subset of $\,\mathbb D^d$, and
  suppose that for Hilbert spaces $\mathcal H$ and $\tilde{\mathcal
    H}$ there are functions $a,b:\Omega \to
  \mathcal{L}(\tilde{\mathcal{H}} ,\mathcal{H})$ such that for any
  admissible kernel $k$ for $\HLH$ restricted to $\Omega$,
  \begin{equation*}
    (aa^*-bb^*)*k \geq 0.
  \end{equation*}
  Then there is a function $\varphi \in H_1^\infty({\mathcal
    K}_{\Lambda^o,\tilde{\mathcal{H}}})$ such that $b = \varphi a$,
  where the multiplication is pointwise.
\end{theorem}

The proof is essentially a reworking of the proof of the realization
theorem, giving a function $\varphi$ over $\Omega$ with a transfer
function representation such that $b = \varphi a$.  The transfer
function representation immediately extends to all of $\mathbb D^d$,
and so by the realization theorem, $\varphi$ extends to a function in
the unit ball of $H^\infty(\mathbb D^d,{\mathcal K}_{\Lambda^o,
  \tilde{\mathcal{H}}})$.

The identification of $\HLH$ and $H^\infty(\mathbb
D^d,\mathcal{L(H)})$ when $d=2$ in the realization theorem uses a
version of And\^o's theorem~\cite{MR0155193} for $\ALH$ (the standard
version of And\^o's theorem corresponds to the case when $\mathcal H =
\mathbb C$), as well as a theorem of Arveson's~\cite{MR1668582}.

\begin{theorem}[And\^o's theorem]
  \label{thm:Andos-theorem-for-ALH}
  Let $\pi: A(\mathbb D^2,\mathcal{L(H)}) \to \mathcal{L(K)}$ or $\pi:
  H^\infty(\mathbb D^2,\mathcal{L(H)}) \to \mathcal{L(K)}$ be a unital
  representation with the property that
  $\|\pi(1_{\mathcal{L(H)}}\otimes z_j)\| \leq 1$ or
  $\|\pi(1_{\mathcal{L(H)}}\otimes z_j)\| < 1$, respectively, for
  $j=1,2$.  Then $\pi$ dilates to a representation $\tilde\pi$ such
  that $\pi(1_{\mathcal{L(H)}}\otimes z_j)$ is unitary, and
  consequently, $\pi$ is completely contractive.
\end{theorem}

\begin{proof}
  For any $F\in \mathcal{L(H)}$, the constant function $F\otimes 1$ is
  obviously in $A(\mathbb D^2,\mathcal{L(H)})$.  Thus $\pi$ restricted
  to the constant functions induces a unital representation of
  $\mathcal{L(H)}$, and since $\mathcal{L(H)}$ is a $C^*$-algebra,
  the induced representation is contractive.

  Now suppose that $\pi: A(\mathbb D^2,\mathcal{L(H)}) \to
  \mathcal{L(K)}$ with $\|\pi(1_{\mathcal{L(H)}}\otimes z_j)\| \leq
  1$, $j=1,2$.  Let $\{e_\alpha\}$ be an orthonormal basis for
  $\mathcal H$.  Define an operator $\eta_{\alpha\beta} :
  \mathcal{L(H)} \to \mathbb C$ by $\eta_{\alpha\beta}(F) := \ip{F
    e_\alpha}{e_\beta}$.  Note that $\eta(F) :=
  {\left(\eta_{\alpha\beta}(F)\right)}_{\alpha\beta} = F$.

  Let $\varphi \in A(\mathbb D^2,\mathcal{L(H)})$ and define
  $\varphi_{\alpha\beta} \in A(\mathbb D^2)$ (the scalar valued
  bidisk algebra) by $\varphi_{\alpha\beta}(z) =
  \eta_{\alpha\beta}(\varphi(z))$.  By the standard form of And\^o's
  theorem, there is a pair of commuting unitary operators $U =
  (U_1,U_2)$ on $\mathcal L(\tilde{\mathcal K})$ such that for all
  $\alpha, \beta$, $\varphi_{\alpha\beta}(T) = P_\mathcal{K}
  \varphi_{\alpha\beta}(U) |_\mathcal{K}$.  Thus
  \begin{equation*}
    \begin{split}
      \varphi(T) = & \eta\otimes 1_{\mathcal{L(K)}} (\varphi(T)) =
      {\left(\varphi_{\alpha\beta}(T)\right)}_{\alpha\beta} \\
      = & {\left(P_\mathcal{K} \varphi_{\alpha\beta}(U)
          |_\mathcal{K} \right)}_{\alpha\beta} =
      {\left(P_\mathcal{K} (\eta_{\alpha\beta}\otimes 1_{\mathcal
            L(\tilde{\mathcal K})})(\varphi(U)) |_\mathcal{K}
        \right)}_{\alpha\beta} \\
      = & P_{\mathcal{K} \otimes \mathcal{H}} \varphi(U)
    |_{\mathcal{K} \otimes \mathcal{H}}.
    \end{split}
  \end{equation*}
  The easy direction of a result of Arveson's~\cite{MR1668582} or the
  spectral mapping theorem and a bit of work then shows that $\pi$ is
  completely contractive.

  Next, suppose that $\pi: H^\infty(\mathbb D^2,\mathcal{L(H)}) \to
  \mathcal{L(K)}$ and $\|\pi(1_{\mathcal{L(H)}}\otimes z_j)\| < 1$,
  $j=1,2$.  Let $n\in\mathbb N$.  Then $H^\infty(\mathbb
  D^2,\mathcal{L(H)}) \otimes M_n(\mathbb C) = H^\infty(\mathbb
  D^2,\mathcal{L}(\mathcal{H}^n))$ and for $j=1,2$,
  \begin{equation*}
    \left\| \pi^{(n)}(1_{\mathcal{L}(\mathcal{H}^n)} \otimes z_j)
    \right\| =
    \left \| \bigoplus_1^n \pi (1_{\mathcal{L(H)}}) \right\| = 
    \left \| \pi (1_{\mathcal{L(H)}}) \right\| < 1.
  \end{equation*}
  It then follows from Theorem~\ref{thm:Aglers-realization} that
  $\pi^{(n)}$ is contractive, and so $\pi$ is completely contractive.
\end{proof}

Another fundamental theorem is due to Brehmer~\cite{MR0131169}.  It
states that a $d$-tuple of commuting contractions satisfying an extra
positivity condition dilates to commuting unitary operators.

\begin{theorem}[Brehmer's theorem]
  \label{thm:Brehmers-theorem}
  Let $T = (T_1,\dots,T_d)$ be a $d$-tuple of commuting contractions
  on a Hilbert space $\mathcal H$ satisfying
  \begin{equation*}
    \prod_1^d (1-T_jT_j^*) \geq 0,
  \end{equation*}
  where the product is in the hereditary sense $($that is, adjoints on
  the right$)$.  Then there is a Hilbert space $\tilde{\mathcal H}$
  containing $\mathcal H$ and a $d$-tuple of commuting unitaries $U =
  (U_1,\dots,U_d)$ such that for any polynomial $p$ over $\mathbb
  C^d$, $p(T) = P_{\mathcal H} p(U) |_\mathcal H$.
\end{theorem}

Because polynomials in $p[\mathbb C^d]\otimes \mathcal{L(H)}$ are
weakly dense in $H^\infty(\mathbb D^d,\mathcal{L(H)})$, the same
reasoning as in the operator generalization of And\^{o}'s theorem,
coupled with Arveson's result, gives an alternate version of Brehmer's
theorem.

\begin{theorem}
  \label{thm:Brehmers-theorem-II}
  Let $\pi: A(\mathbb D^d,\mathcal{L(H)}) \to \mathcal{L(K)}$ or $\pi:
  H^\infty(\mathbb D^d,\mathcal{L(H)}) \to \mathcal{L(K)}$ be a unital
  representation with the property that
  \begin{equation*}
    \prod_{i=1}^d \left(1 -
      \pi(z_i\otimes 1_{\mathcal{L(H)}})\pi(z_i \otimes
      1_{\mathcal{L(H)}})^*\right) \geq 0,
  \end{equation*}
  where the product is hereditary $($that is, adjoints on the
  right$)$, is either positive or strictly positive, respectively.
  Also assume that either $\|\pi(1_{\mathcal{L(H)}}\otimes z_j)\| \leq
  1$ or $\|\pi(1_{\mathcal{L(H)}}\otimes z_j)\| < 1$, respectively.
  Then $\pi$ dilates to a representation
  $\tilde\pi$ with $\tilde\pi(1_{\mathcal{L(H)}}\otimes z_j)$
  unitary, and as a consequence, $\pi$ is completely contractive.
\end{theorem}

There is a version of Brehmer's theorem for standard nearly ample
preorderings, but this requires further developments.

Notice that Agler's realization theorem
(Theorem~\ref{thm:Aglers-realization}) with $d=2$ and
Corollary~\ref{cor:po_for_polydisk} tell us that, at least over the
bidisk with the coordinate functions as test functions, the ample
preordering and the (in this case, unique) nearly standard ample
preordering are equivalent; that is, they generate the same algebra
and norm.  As it happens, we can extend this idea to the polydisk, and
as we will see later, to more general sets of test functions and
domains.  This will be used to prove a generalization of the main
result of~\cite{MR2502431} in Theorem~\ref{thm:d-var-polydisk_real}.

\begin{theorem}
  \label{thm:ample-near-ample-equiv-polydsk}
  Let $\Psi$ be the set of coordinate functions over the polydisk
  $\mathbb D^d$, $d\geq 2$, $\Lambda^a$ the standard ample preordering
  $($so with largest element $\lambda^m = (1,\dots,1)\,)$, and
  $\Lambda^{na}$ a standard nearly ample preordering under
  $\lambda^m$.  Then $\Lambda^a$ and $\Lambda^{na}$ are equivalent
  preorderings.
\end{theorem}

\begin{proof}
  When $d=2$ this has already been shown.  Hence we assume $d>2$.

  Fix a nearly ample preordering $\Lambda^{na}$ under $\lambda^m =
  (1,\dots,1)$.  Recall that by Corollary~\ref{cor:po_for_polydisk},
  with the standard ample preordering, all admissible kernels are
  subordinate to the Szeg\H{o} kernel $k_s$.  Since this kernel is
  also admissible for the standard nearly ample preordering, any
  $\varphi\in H^\infty(\mathbb D^d,{\mathcal K}_{\Lambda^{na},\mathcal
    H})$ is in $H^\infty(\mathbb D^d,{\mathcal K}_{\Lambda^a,\mathcal
    H})$ and the norm in the first algebra $\|\varphi\|_{na}$ is
  greater than or equal to that in the second, $\|\varphi\|_a$.

  We now show that the two norms are the same on $H^\infty(\mathbb
  D^d,{\mathcal K}_{\Lambda^{na},\mathcal H})$.  First of all, recall
  that $\Lambda^{na}$ has two maximal elements $\lambda^m_1$ and
  $\lambda^m_2$, which are $\lambda^m$ with one of the $1$s changed to
  a zero in distinct places.  By relabeling if necessary, we may
  assume that these are in the first two places.  Let $k_3$ be the
  positive kernel defined by $k_3(z,w) = 1_{\mathcal{L(H)}} \otimes
  \prod_{j=3}^d (1-zw^*)^{-1}$.  This has a Kolmogorov decomposition
  $k_3 = aa^*$, where $a: \mathbb D^d \to
  \mathcal{L}(\tilde{\mathcal{H}} ,\mathcal{H})$.  Any admissible
  kernel in $\mathcal K_{\Lambda^{na}}$ then has the form $k = k_3 *
  k'$, where for fixed $z_3,\dots,z_d$ and $j=1,2$,
  \begin{equation*}
    \left([1_{\mathcal{L(H)}}]\otimes (1-z_jw_j^*)k _3(z,w)
      k'(z,w)\right) \geq 0.
  \end{equation*}
  This inequality is valid in particular for any positive kernel $k'$
  which is a function only of the first two coordinates such that for
  $j=1,2$, $\left([1_{\mathcal{L(H)}}]\otimes
    (1-z_jw_j^*)k'(z,w)\right) \geq 0$; that is, by
  Theorems~\ref{thm:Aglers-realization}
  and~\ref{thm:Andos-theorem-for-ALH}, the kernels for
  $H^\infty(\mathbb D^2, \mathcal{L(H)})$.

  Let $\varphi \in H_1^\infty(\mathbb D^d,{\mathcal
    K}_{\Lambda^{na},\mathcal H})$, and define $b:\mathbb D^d \to
  \mathcal{L}(\tilde{\mathcal{H}} ,\mathcal{H})$ by the pointwise
  product $b = a\varphi$.  Then for any admissible kernel $k = k_3*k'$
  as above for the standard nearly ample preordering,
  $([1_{\mathcal{L(H)}}] - \varphi\varphi^*)*k \geq 0$.  For fixed
  $z_3,\dots,z_d$, $(aa^* - bb^*)*k' \geq 0$.  By
  Theorem~\ref{thm:Agler-Pick-interp-polydisk}, we can also write $b =
  a \tilde\varphi$, where as a function of only the first two
  variables, $\tilde\varphi \in H_1^\infty(\mathbb D^2,\mathcal{L(H)})$;
  that is, is analytic and has supremum norm less than or equal
  to~$1$.

  Fix $z_1,z_2\in\mathbb D$.  Set the kernel $k'$ to be
  \begin{equation*}
    k'(z,w) =
    \begin{cases}
      1_{\mathcal{L(H)}} & w_1 = z_1 \text{ and } w_2 = z_2\,;\\
      0 & \text{otherwise.}
    \end{cases}
  \end{equation*}
  Then $k = k_3 * k'$ is admissible for the standard nearly ample
  preordering, and so $([1_{\mathcal{L(H)}}] - \varphi\varphi^*)*k_3
  \geq 0$.  Hence for fixed $z_1,z_2$, the function $\varphi \in
  H_1^\infty(\mathbb D^{d-2},\mathcal{L(H)})$; that is, as a function
  of $z_3,\dots ,z_d$, $\varphi$ is also analytic and bounded by $1$.
  Being separately analytic in all variables and bounded, we conclude
  from Hartog's theorem that $\varphi \in H^\infty(\mathbb
  D^d,\mathcal{L(H)})$.  Writing $k_s^d$ for the scalar Szeg\H{o}
  kernel on $\mathbb D^d$, it follows that
  \begin{equation*}
    ([1_{\mathcal{L(H)}}] - \varphi\varphi^*)*(1_{\mathcal{L(H)}}
    \otimes k^d_s) 
    = (aa^*-bb^*)*(1_{\mathcal{L(H)}}\otimes k^2_s) 
    = aa^* * ([1_{\mathcal{L}(\tilde{\mathcal{H}})}] -
    \tilde\varphi{\tilde\varphi}^*)*(1_{\mathcal{L(H)}}\otimes k^2_s)
    \geq 0;
  \end{equation*}
  that is, $\varphi \in H_1^\infty(\mathbb D^d,\mathcal{L(H)})$.  This
  shows that $H^\infty(\mathbb D^d,{\mathcal K}_{\Lambda^{na},\mathcal
    H})$ is a norm closed subalgebra of $H^\infty(\mathbb
  D^d,\mathcal{L(H)})$ (with the same norm).

  To finish the proof, we note from the transfer function
  representation for $H^\infty(\mathbb D^d,{\mathcal
    K}_{\Lambda^o,\mathcal H})$ from
  Theorem~\ref{thm:Aglers-realization} that polynomials in the
  coordinate functions are weakly dense in $H^\infty(\mathbb
  D^d,{\mathcal K}_{\Lambda^o,\mathcal H})$.  Since they are also
  weakly dense in $H^\infty(\mathbb D^d,\mathcal{L(H)})$ and by
  Lemma~\ref{lem:preorderings_order_adm_kers},
  \begin{equation*}
    H^\infty(\mathbb D^d,{\mathcal K}_{\Lambda^o,\mathcal H})
    \subseteq H^\infty(\mathbb D^d,{\mathcal K}_{\Lambda^{na},\mathcal
      H}) \subseteq H^\infty(\mathbb D^d,{\mathcal
      K}_{\Lambda^a,\mathcal H}) = H^\infty(\mathbb
    D^d,\mathcal{L(H)}),
  \end{equation*}
  the result follows.
\end{proof}

We finally mention an abstract version of the realization theorem for
general domains and sets of test functions~\cite{MR2389623} which is
also part of the inspiration for the work that follows.  The theorem
was only stated and proved for scalar valued functions, though the
generalization to operator valued functions is straightforward, as we
shall see.

\goodbreak

\begin{theorem}[Classical realization theorem]
  \label{thm:classic-real-thm}
  Let $X$ be a set, $\Psi = \{\psi_j\}$ a (not necessarily finite)
  collection of test functions on $X$, $\Lambda = \{e_j\}$, $\mathcal
  K$ the set of $\mathcal{L(H)}$ valued admissible kernels, and
  $\HLHb$ the unit ball of the algebra generated by the kernels in
  $\mathcal K$ in $\HLH$.  Let $\varphi: X\to \mathcal{L(H)}$.  The
  following are equivalent:
  \begin{enumerate}
  \item[$($SC$\,)$] $\varphi \in \HLHb$;
  \item[$($AD$\,)$] There exists a positive kernel $\Gamma: C_0(\Psi)
    \to \mathcal{L(H)}$ so that $[1_{\mathcal{L(H)}}] -
    \varphi\varphi^* = \Gamma*([1] - EE^*)$, where $E(x)(\psi_j) =
    \psi_j(x)$;
  \item[$($TF$\,)$] There is a unitary colligation $(U,\mathcal E,
    \rho)$, where $\mathcal E$ is a Hilbert space, $U =
    \begin{pmatrix} A & B \\ C & D \end{pmatrix}$ a unitary operator
    on $\mathcal E \oplus \mathcal H$ and $\rho: C_0(\Psi) \to
    \mathcal{L(E)}$ a unital representation such that $\varphi$ has a
    transfer function representation
    \begin{equation*}
      \varphi (x) = D + C Z(x)(1_{\mathcal L(E)} - AZ(x))^{-1} B,
    \end{equation*}
    with $Z(x) = \rho(E)(x)$;
  \item[$($vN1$\,)$] $\varphi \in \HLH$ and for every unital
    representation $\pi$ of $\HLH$ satisfying \break $\|\pi(\psi_j)\|
    < 1$ for all $j$, we have $\|\pi(\varphi)\| \leq 1$.
  \item[$($vN2$\,)$] $\varphi \in \HLH$ and for every weakly
    continuous unital representation $\pi$ of \break $\HLH$ satisfying
    $\|\pi(\psi_j)\| \leq 1$ for all $j$, we have $\|\pi(\varphi)\|
    \leq 1$.
  \end{enumerate}
\end{theorem}

Note that the second item (the so-called Agler decomposition) has a
more familiar form in this setting when the set of test functions is
finite; namely, we can rewrite this as $[1_{\mathcal{L(H)}}] -
\varphi\varphi^* = \sum_j \Gamma_j*([1_{\mathcal{L(H)}}] -
\psi_j\psi_j^*)$, where each $\Gamma_j$ is a positive $\mathcal{L(H)}$
valued kernel, as in Theorem~\ref{thm:Aglers-realization}.  A
representation is \textbf{weakly continuous} if whenever
$(\varphi_\alpha)$ is a bounded net in $\HLH$ converging pointwise in
norm to $\varphi \in \HLH$, $(\pi(\varphi_\alpha))$ converges weakly
to $\pi(\varphi)$.

Recall that in this paper we are taking the set of test functions to
be finite.  Thus $Z$ has the simpler form
\begin{equation*}
  Z(x) = \sum_j \psi_j(x)P_j,
\end{equation*}
where each $P_j$ is an orthogonal projection and $\sum_j P_j =
1_{\mathcal{L(E)}}$.

\subsection{Topologizing $X$}
\label{subsec:topologising-x}

In the construction of kernels and function spaces, we did not assume
that the underlying set $X$ has a topology, though even when it does
have one, it will be convenient to take it to have the weakest
topology making the test functions continuous.  In most cases of
interest the test functions are already continuous when $X$ has its
native topology, so this assumption will make no substantial
difference, at least when $\overline{X}$ is compact in the original
topology and the test functions extend continuously to
$\overline{X}$.

Write ${\HL^*}_0$ for the vector space of continuous linear
functionals on $\HL$.  The set $\mathcal N = \{e\in{\HL^*}_0:|e(\psi)|
=0 \text{ for all } \psi\in\Psi \}$ is a subspace, and we write
$\HL^*$ for the quotient space ${\HL^*}_0/\mathcal N$.  The test
functions induce a topology on $\HL^*$ with a subbase consisting of
sets of the form
\begin{equation*}
  U_{w,c} = \{\eta\in \HL^* : \sup_{\psi\in\Psi} |\eta(\psi)-w| < c \},
  \qquad w\in\mathbb D,\ c\in(0,1).
\end{equation*}
By construction the map $\hat E: X \to \HL^*$ define by $\hat
E[x](\varphi) = \varphi(x)$ is an embedding by the point separation
property of the test functions.  With this topology $\HL^*$ becomes a
locally compact, convex topological vector space, and so is Hausdorff
(in fact it has even stronger separation properties, which we will not
need).  We identify the closure of $X$, $\overline{X}$ with the
closure of $\hat E[X]$ in $\HL^*$.

The space $\HL^*$ induces a weak-$*$ topology on $\HL^{**}$ in which
the norm closed unit ball of $\HL^{**}$ is compact by the
Banach-Alaoglu theorem.  By dint of being finite $\Psi$ in $\HL^{**}$
is also compact.

We also find that the test functions extend continuously to
$\overline{X}$.  However it is not \emph{a priori} evident that the
test functions will separate the points of $\overline{X}$.  However,
this can be achieved simply by identifying those points in
$\overline{X}$ which are not distinguished by the test functions.

We could have carried out the same sort of construction replacing
$\HLH^*$ by the space of bounded linear operators from $\HLH$ to
$\mathcal{L(H)}$, once again modding out by those maps $\eta$ with the
property that $\eta(\Psi) = \{0\}$.  Since any $\eta' \in \HLH^*$
induces $\eta\in \HL^*$ by $\eta(\varphi) = \eta'(\varphi \otimes
1_{\mathcal{L(H)}})$, this essentially adds nothing new.

By construction, $\HLH$ is a norm closed subalgebra of
$C_b(X,\mathcal{L(H)})$, the $C^*$-algebra of bounded continuous
$\mathcal{L(H)}$ valued functions on $X$; that is, a subalgebra of
$C(\beta X, \mathcal{L(H)})$, where $\beta X$ is the Stone-\v{C}ech
compactification of $X$.  However $\overline{X}$ generally tends to be
quite a bit smaller than $\beta X$.  For example, if $X = \mathbb D^d$
and the test functions are the coordinate functions, $\overline{X} =
{\overline{\mathbb D}}^d$, as one would expect.

\subsection{Continuity and convergence in $\HLH$ and $\ALH$}
\label{subsec:cont-conv-hlh}

With $X$ topologized as in the last subsection, we can now address the
continuity of elements of $\HLH$ (we show that they are all
continuous) and connections between various topologies on $\HLH$ and
$\ALH$, the subalgebra of $\HLH$ which in analogy with the disk
algebra, consists of those (continuous) elements of $\HLH$ which
extend continuously to $\overline{X}$.  When dealing with $\ALH$ it is
convenient to assume that the test functions are in this algebra.

\begin{lemma}
  \label{lem:elts_of_HLH_ctnuous}
  Let $\Psi$ be a finite set and $\varphi\in \HLH$.  Then $\varphi$ is
  continuous; that is, if $\,(x_\alpha)$ is a net in $X$ converging to
  $x\in X$, $\|\varphi(x_\alpha) - \varphi(x)\| \to 0$.  Furthermore,
  given $\epsilon >0$ and $x\in X$, there is an open set $U_x\ni x$
  such that for all $\varphi\in \HLH$, $\|\varphi(y)-\varphi(y')\| <
  \epsilon \|\varphi\|$ for all $y,y'\in U_x$.
\end{lemma}

\begin{proof}
  Fix $0 < \delta < 1$.  By definition, for $x\in X$,
  $\sup_{\psi\in\Psi} |\psi(x)| = 1-\tilde\epsilon$ for some
  $\tilde\epsilon > 0$.  In fact, for any $\epsilon >0$, it follows
  that since $\Psi$ is a finite set, $U_{x,\epsilon} := \{y\in X :
  \sup_{\psi\in\Psi} |\psi(x)-\psi(y)| < \epsilon\}$ is a relatively
  open neighborhood of $x$.

  We claim that for $\epsilon$ sufficiently small and $y\in
  U_{x,\epsilon}$, the kernel defined by
  \begin{equation*}
    k_{x,y}(z,w) :=
    \begin{cases}
      1 & z=w=x \text{ or }z=w=y, \\
      1-\delta & (z=x\text{ and }w=y)\text{ or }(z=y\text{ and }
      w=x),\\
      0 &\text{otherwise,}
    \end{cases}
  \end{equation*}
  is an admissible kernel for $\HL$, and hence $k_{x,y}\otimes
  1_{\mathcal{L(H)}}$ is admissible for $\HLH$.  We require that
  \begin{equation*}
    \left(
      \begin{pmatrix}
        1 & 1 \\ 1 & 1
      \end{pmatrix}
      -
      \begin{pmatrix}
        \psi(x)\psi(x)^* & \psi(x)\psi(y)^* \\ \psi(y)\psi(x)^* &
        \psi(y)\psi(y)^*
      \end{pmatrix}
      \right) *
      \begin{pmatrix}
        1 & 1-\delta \\ 1-\delta & 1
      \end{pmatrix}
      \geq 0.
  \end{equation*}
  This will be nonnegative as long as $(1-\delta)^2
  |1-\psi(x)\psi(y)^*|^2 \leq (1-\psi(x)\psi(x)^*)
  (1-\psi(y)\psi(y)^*)$.

  Since $\psi(x)\psi(y)^* = \psi(x)(\psi(x)^* - (\psi(x)-\psi(y))^*)$
  and $|\psi(x)| < 1$,
  \begin{equation*}
    \begin{split}
      |1-\psi(x)\psi(y)^*|^2 
      = & |1-\psi(x)\psi(x)^* + \psi(x)(\psi(x)-\psi(y))^*)|^2 \\
      \leq & (1-\psi(x)\psi(x)^*)^2 + 2\epsilon (1-\psi(x)\psi(x)^*)
      +\epsilon^2 .
    \end{split}
  \end{equation*}
  Also,
  \begin{equation*}
    1-\psi(y)\psi(y)^* \geq (1-\psi(x)\psi(x)^*) -
    2\epsilon - \epsilon^2.
  \end{equation*}
  Hence it suffices to choose $\epsilon$ so that
  \begin{equation*}
    (1-\delta)^2\left(c^2 + 2\epsilon c +\epsilon^2\right) \leq
    c^2 - 2\epsilon c - \epsilon^2 c,
  \end{equation*}
  where $c= 1-\psi(x)\psi(x)^*$, or equivalently, so that
  \begin{equation*}
    -(c+(1-\delta)^2)\epsilon^2 - 2c(1+(1-\delta)^2)\epsilon
    +c^2(1-(1-\delta)^2) > 0.
  \end{equation*}
  This is a polynomial in $\epsilon$ which is positive when $\epsilon
  = 0$, and so by continuity is positive for sufficiently small
  $\epsilon > 0$.  In fact it is positive on the interval $\epsilon\in
  (0,d)$, where
  \begin{equation*}
    \begin{split}
      d &= \frac{-c(1+(1-\delta)^2)+\sqrt{c^2(1+(1-\delta)^2)^2 +
        c^2(c+(1-\delta)^2)(1-(1-\delta)^2)}}{c+(1-\delta)^2} \\
    &\geq c\frac{-(1+(1-\delta)^2)+\sqrt{(1+(1-\delta)^2)^2 +
        (c+(1-\delta)^2)(1-(1-\delta)^2)}}{1+(1-\delta)^2} \\
    & = c\left(\sqrt{1 +
        \frac{(c+(1-\delta)^2)(1-(1-\delta)^2)}{(1+(1-\delta)^2)^2}} - 
        1 \right) \\
    &\geq \frac{c(c+(1-\delta)^2)(1-(1-\delta)^2)}{2(1+(1-\delta)^2)^2} \\
    &\geq c\delta \geq 2\tilde\epsilon\delta.
    \end{split}
  \end{equation*}

  Fix $\varphi\in \HLH$.  We may assume without loss of generality
  that $\|\varphi\| = 1$.  Since $k_{x,y}$ is admissible,
  \begin{equation}
    \label{eq:6}
    \begin{pmatrix}
      1-\varphi(x)\varphi(x)^* & (1-\delta)(1-\varphi(x)\varphi(y)^*)
      \\
      (1-\delta)(1-\varphi(y)\varphi(x)^*) & 1-\varphi(y)\varphi(y)^*
    \end{pmatrix}
    \geq 0.
  \end{equation}

  Let $h \in \mathcal H$ with $\|h\| = 1$, and set
  \begin{equation*}
    \begin{split}
      c_1 &= \ip{(1-\varphi(x)\varphi(x)^*)h}{h} \geq 0, \\
      c_2 &= \ip{\varphi(x)(\varphi(x)-\varphi(y))^*h}{h}, \\
      c_3 &= \|(\varphi(x)-\varphi(y))^*h\|^2.
    \end{split}
  \end{equation*}
  Then positivity of the matrix in \eqref{eq:6} implies the scalar
  matrix
  \begin{equation}
    \label{eq:11}
    \begin{pmatrix}
      c_1 & (1-\delta)(c_1+c_2) \\ (1-\delta)(c_1+c_2^*) &
      c_1+c_2+c_2^*-c_3
    \end{pmatrix}
    \geq 0,
  \end{equation}
  which is equivalent to
  \begin{equation}
    \label{eq:10}
    (1-(1-\delta)^2)(c_1^2+c_1c_2^*+c_2c_1) -(1-\delta)^2|c_2|^2 \geq
    c_1c_3.
  \end{equation}
  Since $1\geq c_1$, $2\geq |c_2|$ and $1-(1-\delta)^2 \geq 2\delta$,
  it follows from \eqref{eq:10} that
  \begin{equation*}
    10\delta - (1-\delta)^2|c_2|^2 \geq c_1c_3,
  \end{equation*}
  and so $c_3 \leq 10\delta/c_1$ when $c_1\neq 0$.  So if $c_1\geq
  \sqrt{\delta}$, then $c_3 \leq 10\sqrt{\delta}$.  On the other hand,
  if $c_1 < \sqrt{\delta}$, then $\|\varphi(x)h\|^2 \geq
  1-\sqrt{\delta}$.  Also, by \eqref{eq:11}
  \begin{equation*}
    0\leq
    \begin{pmatrix}
      \sqrt{\delta} & (1-\delta)(c_1+c_2) \\ (1-\delta)(c_1+c_2^*) &
      \sqrt{\delta}+c_2+c_2^*-c_3
    \end{pmatrix}
    \leq
    \begin{pmatrix}
      \sqrt{\delta} & (1-\delta)(c_1+c_2) \\ (1-\delta)(c_1+c_2^*) & 5
    \end{pmatrix}.
  \end{equation*}
  Consequently, $|c_1+c_2| \leq 5\sqrt{\delta}/(1-\delta)$, and so
  $|c_2| \leq 6\sqrt{\delta}/(1-\delta)$.  Then
  $\sqrt{\delta}+c_2+c_2^*-c_3 \geq 0$ gives $c_3 \leq
  13\sqrt{\delta}/(1-\delta)$.  Thus $\|\varphi(x)-\varphi(y)\| \leq
  13\sqrt{\delta}/(1-\delta)$ whenever $y\in U_{x,\epsilon}$.  Note
  that by construction the set $U_{x,\epsilon}$ is independent of the
  choice of $\varphi\in \HLHb$.

  Suppose $(x_\alpha)$ is a net converging to $x \in X$.  We saw that
  given $\delta > 0$, there is an $\epsilon > 0$ such that $k_{x,y}$
  is an admissible kernel.  Also, there is a $\alpha_\delta$ such that
  for all $\alpha > \alpha_\delta$, $x_\alpha \in U_{x,\epsilon}$.
  Hence by what we have shown, $\|\varphi(x)-\varphi(x_\alpha)\| \leq
  13\sqrt{\delta}/(1-\delta)$.  Since any open neighborhood of $x$
  contains a $U_{x,\epsilon}$ for sufficiently small $\epsilon$, we
  conclude that $\varphi$ is norm continuous.
\end{proof}

Lemma~\ref{lem:elts_of_HLH_ctnuous} ensures that the definition of
$\ALH$ makes sense, though of course at this point we do not know if
it consists of any more than the constant functions on $X$.  In the
concrete examples most commonly considered, it is also the case that
the test functions are in $\ALH$, and we will generally assume this to
be the case, as well as that they separate the points of
$\overline{X}$.

\begin{lemma}
  \label{lem:HLH-topologies}
  The space $\HLH$ is complete in the norm topology, and its norm
  closed unit ball $\HLHb$ is closed in both the topology of pointwise
  convergence and the topology of uniformly convergence on compact
  subsets of $X$.
\end{lemma}

\begin{proof}
  Let $(\varphi_\alpha)$ be a Cauchy net in $\HLH$.  For fixed $x\in
  X$ let $k_x$ be the kernel which equals $1_{\mathcal{L(H)}}$ in the
  $(x,x)$ place and zero elsewhere.  It is clear by definition of the
  test functions that this is an admissible kernel.  By the assumption
  that $(\varphi_\alpha)$ is a Cauchy net, using the kernel $k_x$, we
  see that $(\varphi_\alpha(x))$ is a Cauchy net in $\mathcal{L(H)}$,
  and since this space is complete, $(\varphi_\alpha(x))$ converges in
  norm.  We denote the limit by $\varphi(x)$.

  We show that the function $\varphi: x \mapsto \varphi(x)$ is in
  $\HLH$.  Given $\epsilon > 0$, there is an $\alpha_0$ such that for
  all $\alpha,\beta > \alpha_0$ and any admissible kernel $k$,
  $(\epsilon[1_{\mathcal{L(H)}}] - (\varphi_\alpha - \varphi_\beta)
  (\varphi_\alpha - \varphi_\beta)^*)*k \geq 0$.  From this we see
  that there is a constant $c>0$ such that for all $\alpha >
  \alpha_0$, $\|\varphi_\alpha\| < c$

  Let $F\subset X$ be a finite set.  Given $\epsilon > 0$, choose
  $\alpha_0$ as above, and also so that for all $\alpha > \alpha_0$
  and $x,y\in F$, $2c\|\varphi_\alpha(x) - \varphi(x)\| +
  \|\varphi_\alpha(y) - \varphi(y)\|^2 \leq \epsilon/|F|^2$.
  Letting $I$ denote the kernel which is $1_{\mathcal{L(H)}}$ on the
  main diagonal and zero elsewhere, we have on $F\times F$,
  \begin{equation*}
    \begin{split}
      & (c[1_{\mathcal{L(H)}}] + \epsilon I - \varphi\varphi^*)*k \\
      = \,& (c[1_{\mathcal{L(H)}}] + \epsilon I - (\varphi_\alpha +
      (\varphi - \varphi_\alpha))(\varphi_\alpha^* + (\varphi -
      \varphi_\alpha) ^*))*k \\
      \geq \,& (c[1_{\mathcal{L(H)}}] - \varphi_\alpha
      \varphi_\alpha^*)*k + (\epsilon - (2c \|\varphi -
      \varphi_\alpha\| - \|\varphi - \varphi_\alpha\|^2)I*k \\
      \geq \,& 0.
    \end{split}
  \end{equation*}
  Since $F$ and $\epsilon$ are arbitrary, this shows that
  $(c[1_{\mathcal{L(H)}}] - \varphi\varphi^*)*k \geq 0$, and so
  $\varphi \in\HLH$.

  Suppose that $\varphi_\alpha \to \varphi$ pointwise, where
  $\|\varphi_\alpha\| \leq 1$.  Let $k$ be any admissible kernel and
  $F$ a finite subset of $X$.  Given $\epsilon > 0$, there is some
  $\alpha_0$ such that for all $\alpha > \alpha_0$ and all $x\in F$,
  $\|\varphi_\alpha(x) - \varphi(x)\| < \epsilon$.  Then for $\kappa =
  \max_{x,y\in F} \|k(x,y)\|$,
  \begin{equation*}
    \begin{split}
      & \left((([1_{\mathcal{L(H)}}] +
        \varphi(x)\varphi(y)^*) k(x,y)\right) \\
      = \, & \left(([1_{\mathcal{L(H)}}] +
        \varphi_\alpha(x)\varphi_\alpha(y)^*) k(x,y) \right. \\
      & \left.  - 2 \left(\varphi_\alpha(x) (\varphi(y) -
          \varphi_\alpha(y))^* + (\varphi(x) -
          \varphi_\alpha(x))\varphi_\alpha(y)^* \right) k(x,y) 
        - (\varphi(x) - \varphi_\alpha(x))(\varphi(y) -
        \varphi_\alpha(y))^* k(x,y)\right) \\
      \geq \,& -\epsilon\kappa(1+ \epsilon)I_{F\times F}, \\
    \end{split}
  \end{equation*}
  which goes to zero as we take $\epsilon$ to zero, showing that
  $\varphi \in \HLHb$.

  Finally, if $\varphi_\alpha \to \varphi$ uniformly on compact
  subsets of $\HLHb$, then in particular it converges pointwise to
  $\varphi$, and hence by what we have shown, $\varphi \in \HLHb$.
\end{proof}

\begin{lemma}
  \label{lem:ALH-is-norm-closed}
  The algebra $\ALH$ is closed in the norm topology, the topology of
  uniform convergence, and the topology of pointwise convergence.
\end{lemma}

\begin{proof}
  Let $(\varphi_\alpha)_{\alpha\in A}$ be a net in $\ALH$ converging
  in norm in $\HLH$ to $\varphi$.  We show that $\varphi \in \ALH$.

  Let $(x_\beta)$ be a net in $X$ converging to $x\in\overline{X}$.
  By norm convergence, given any $\epsilon > 0$, there exists
  $\alpha_0$ such that for all $\alpha_1,\alpha_2 > \alpha_0$ and all
  $\beta$, $\|\varphi_{\alpha_1}(x_\beta) -
  \varphi_{\alpha_2}(x_\beta)\| < \epsilon$.  By continuity,
  $\|\varphi_{\alpha_1}(x) - \varphi_{\alpha_2}(x)\| < \epsilon$, and
  so $(\varphi_{\alpha}(x))$ is a Cauchy net and hence has a limit,
  which we denote by $\varphi(x)$.  By continuity, $\varphi(x)$ is
  independent of the choice of net $(x_\beta)$.

  We show that the function $\varphi: x\mapsto \varphi(x)$ is
  continuous on $\overline{X}$.  Given $\epsilon > 0$, let $V_x$ be an
  open ball in $\mathcal{L(H)}$ of radius $\epsilon/2$ about
  $\varphi(x)$, and set ${\tilde U}_x = \varphi^{-1}(V_x)\cap X$, an
  open set in $X$.  Then let $U_x$ be an open set in $\overline{X}$
  such that $U_x\cap X = {\tilde U}_x$ and note that $x\in U_x$.  Let
  $y\in U_x$ and construct $U_y$ in an identical manner.  Obviously,
  $U_x\cap U_y \cap X \neq \emptyset$, and so if we choose $w$ in this
  set, $\epsilon > \|\varphi(x) - \varphi(w)\| + \|\varphi(w) -
  \varphi(y)\| \geq \|\varphi(x) - \varphi(y)\|$.  It follows that
  $\varphi$ is continuous on $\overline{X}$.

  The last two statements follow from the previous lemma.
\end{proof}

\subsection{Connections between $\HLH$, $\ALH$ and algebras over the
  polydisk}
\label{subsec:conn-betw-hlh}

Suppose either that $X$ has a topology in which $\overline{X}$ is
compact (say for example, as a bounded subset of $\mathbb C^d$) or
that $X$ is endowed with a topology as in
Subsection~\ref{subsec:topologising-x} which then ensures the
continuity of the test functions and compactness of $\overline{X}$.
In either case, we also assume that the test functions are in $\AL$
and that they separate the points of $\overline{X}$.  Then there is a
natural identification of $\HLH$ and $\ALH$ with certain subalgebras
of bounded analytic functions over subsets of the polydisk, which we
give below.

Recall that by definition the test functions have the property that
for $x\in X$,
\begin{equation*}
  z = \xi(x) := (\psi_1(x),\dots,\psi_d(x)) \in \mathbb D^d,
\end{equation*}
and that the test functions separate the points of $X$, or
equivalently, that $\xi$ is injective.  By the assumption that the
test functions are in $\ALH$ and that they separate the points of
$\overline{X}$, which is compact, we have that $\xi(X) = \Omega
\subseteq \mathbb D^d$, and $\xi(\overline{X}) = \overline{\Omega}$ is
a compact subset of ${\overline{\mathbb D}}^d$.

Write $\Psi^{\mathrm{pd}} = \{Z_1(z),\ldots,Z_d(z)\}$, $z\in\mathbb
D^d$, where $Z_j(z) =z_j$ are the coordinate functions, and take these
as test functions over $\Omega$.  Let $\mathcal H$ be a Hilbert space
and write ${\mathcal K}^{\mathrm{pd}}_{\Lambda,\mathcal H}$ for the
admissible kernels in this setting, and ${H_1^\infty({\mathcal
    K}^{\mathrm{pd}}_{\Lambda,\mathcal H})}$ for the associated
algebra with unit ball $H_1^\infty({\mathcal
  K}^{\mathrm{pd}}_{\Lambda,\mathcal H})$.  In analogy with the
Serre-Swan theorem, we have the following.

\begin{lemma}
  \label{lem:id_HLH_w_pdisk_subalg}
  The map $\xi:\overline{X} \to \overline{\Omega}\subseteq
  {\overline{\mathbb D}}^d$ defined above is a homeomorphism.
  Consequently, given a preordering $\Lambda$ and Hilbert space
  $\mathcal H$, there is an isometric unital algebra homomorphism from
  $\HLH$ onto $H^\infty(\Omega,{\mathcal
    K}^{\mathrm{pd}}_{\Lambda,\mathcal H})$ and from $\ALH$ onto
  $A(\Omega, {\mathcal K}^{\mathrm{pd}}_{\Lambda,\mathcal H})$.
\end{lemma}

\begin{proof}
  Since $\xi$ is a bijection from $\overline{X}$ to
  $\overline{\Omega}$, $\xi^{-1}$ is well-defined, and so it suffices
  to show that $\xi^{-1}$ is continuous.  Suppose not.  Then there is
  a net $(z_\alpha)_{\alpha\in A}$ converging to $z$ in
  $\overline{\Omega}$ such that $x_\alpha = \xi^{-1}(z_\alpha)$ does
  not converge to $x=\xi^{-1}(z)$.  Hence there is an open set $U$
  containing $x$ with the property that for all $\alpha$ in $A$, there
  exists $\beta \geq \alpha$ such that $x_\beta \notin U$.  The set $B
  = \{\beta\in A: x_\beta \notin U\}$ is thus a directed set.  Since
  $\overline{X}$ is compact, there is a subnet
  $(x_\gamma)_{\gamma\in\Gamma}$ of $(x_\beta)_{\beta\in B}$
  converging to some $\tilde x \neq x$.  But the subnet
  $(z_\gamma)_{\gamma\in\Gamma}$ converges to $z$, and so $\xi(\tilde
  x) = \xi(x)$, contradicting the injectivity of $\xi$.

  It is then clear that for $k^{\mathrm{pd}}\in {\mathcal
    K}^{\mathrm{pd}}_{\Lambda,\mathcal H}$, $k$ defined by $k(x,y) =
  k^{\mathrm{pd}}(\xi(x),\xi(y))$ is in ${\mathcal
    K}_{\Lambda,\mathcal H}$, and similarly, that for $k\in {\mathcal
    K}_{\Lambda,\mathcal H}$, $k^{\mathrm{pd}}$ defined by
  $k^{\mathrm{pd}}(z,w) = k(\xi^{-1}(z),\xi^{-1}(w))$ is in ${\mathcal
    K}^{\mathrm{pd}}_{\Lambda,\mathcal H}$, giving a bijective
  correspondence between the sets of admissible kernels.  It follows
  easily that $\nu: \HLH \to H^\infty(\Omega, {\mathcal
    K}^{\mathrm{pd}}_{\Lambda,\mathcal H})$ given by $\nu(\varphi)(z)
  = \varphi(\xi^{-1}(z))$ is an isometric unital algebra homomorphism
  and that $\nu(\ALH) = A(\Omega, {\mathcal
    K}^{\mathrm{pd}}_{\Lambda,\mathcal H})$.
\end{proof}

\begin{corollary}
  \label{cor:Szego-ker-inv}
  Let $\Lambda$ be an ample preordering and $F$ a finite subset of
  $X$.  Then the Szeg\H{o} kernel restricted to $F\times F$ has closed
  range.
\end{corollary}

\begin{proof}
  Since the statement is true over the polydisk and the above map
  $\xi$ is injective, the result is immediate.
\end{proof}

\subsection{Auxiliary test functions}
\label{subsec:auxil-test-funct}

Let $0 < \lambda \in \Lambda$.  Define two $\mathbb
C^{2^{|\lambda|-1}}$ valued functions by
\begin{equation*}
  \psi_\lambda^+(x) = \mathrm{row}_{\lambda'\in\Lambda_+,\,
    \lambda'\leq_\ell \lambda}\,\left(\psi^{\lambda'}\right)
  \quad\text{and}\quad
  \psi_\lambda^-(x) = \mathrm{row}_{\lambda'\in\Lambda_-,\,
    \lambda'\leq_\ell \lambda}\,\left(\psi^{\lambda'}\right);
\end{equation*}
that is, $\psi_\lambda^+$ has entries consisting of products of even
numbers of $\psi$s (counting multiplicity) taken from $\psi^\lambda$,
while $\psi_\lambda^-$ has entries consisting of products of odd
numbers of $\psi$s (counting multiplicity) taken from $\psi^\lambda$.
We order these in increasing order, so that the first entry of
$\psi_\lambda^+$ is $1$ (corresponding to $0 <_\ell \lambda$).

By construction, for $\lambda \in \Lambda$,
\begin{equation*}
  \prod_{\lambda_i\in\lambda} ([1] -
  \psi_i\psi_i^*)^{\lambda_i}(x,y)
  = \psi_\lambda^+(x)\psi_\lambda^+(y)^* -
  \psi_\lambda^-(x)\psi_\lambda^-(y)^*
\end{equation*}
and for each $x\in X$,
\begin{equation*}
  \prod_{\lambda_i\in\lambda} ([1] -
  \psi_i\psi_i^*)^{\lambda_i}(x,x)
  > 0.
\end{equation*}
From this we see that $|\psi_\lambda^+(x)|^2 =
\psi_\lambda^+(x)\psi_\lambda^+(x)^* > 1$.  For $\psi_\lambda^+(x) =
|\psi_\lambda^+(x)|\nu_\lambda(x)$ the polar decomposition of
$\psi_\lambda^+(x)$, we set
\begin{equation*}
  \omega_\lambda(x) := \nu_\lambda(x)^* |\psi_\lambda^+(x)|^{-1} =
  \psi_\lambda^+(x)^* |\psi_\lambda^+(x)|^{-2}.
\end{equation*}
Then $\psi_\lambda^+(x) \omega_\lambda(x) = 1$.  Note that
$\|\omega_\lambda(x)\| = |\psi_\lambda^+(x)|^{-1} < 1$.  Define
\begin{equation*}
  \sigma_\lambda(x) := \omega_\lambda(x) \psi_\lambda^-(x) =
  \psi_\lambda^+(x)^* |\psi_\lambda^+(x)|^{-2} \psi_\lambda^-(x) \in
  M_{2^{|\lambda|-1}}(\mathbb C).
\end{equation*}
Obviously for all $x\in X$,
\begin{equation*}
  \psi_\lambda^+(x)\sigma_\lambda(x) = \psi_\lambda^-(x)
  \qquad\text{and} \qquad \|\sigma_\lambda(x)\| < 1.
\end{equation*}
As defined, $\sigma_\lambda(x)(\mathrm{ran\,}\psi_\lambda^{-*}(x))
\subseteq \psi_\lambda^{+*}(x)$ and
$\sigma_\lambda(x)(\mathrm{ker\,}\psi_\lambda^{-}(x)) = \{0\}$.
Additionally, if the test functions are in $\AL$, then $\sigma_\lambda
\in C(\overline{X}, M_n(\mathbb C))$ where $n = 2^{|\lambda|-1}$.

Let $n = 2^{|\lambda|-1}$ and $1_n$ be the identity matrix on $\mathbb
C^n$.  Then
\begin{equation}
  \label{eq:5}
  \begin{split}
    &(\psi_\lambda^+\psi_\lambda^{+*}*([1_n] -
    \sigma_\lambda\sigma_\lambda^*)*(k\otimes 1_n))(x,y) \\
    = & \psi_\lambda^+(x)((k(x,y)\otimes 1_n)
    - \sigma_\lambda(x) (k(x,y)\otimes 1_n) \sigma_\lambda(y)^*)
    \psi_\lambda^+(y)^* \\
    = & \psi_\lambda^+(x) (k(x,y)\otimes 1_n) \psi_\lambda^+(y)^*
    - \psi_\lambda^-(x) (k(x,y)\otimes 1_n) \psi_\lambda^-(y)^* \\
    = & \left(\textstyle\prod_{\lambda_i\in\lambda} ([1] -
      \psi_i\psi_i^*)^{\lambda_i}*k\right)(x,y).
  \end{split}
\end{equation}
We call the functions $\sigma_\lambda$, $\lambda\in\Lambda$
\textbf{auxiliary test functions}.  The last calculation shows that we
apparently only have positivity of $([1_n] - \sigma_\lambda
\sigma_\lambda^*)*(k\otimes 1_n)$ after taking the Schur product with
$\psi_\lambda^+\psi_\lambda^{+*}$, though clearly $\sigma_\lambda \in
\HL$ if $\lambda = e_j$ for some $j$, since when $|\lambda| = 1$, the
auxiliary test functions are just the ordinary test functions.  We
examine this point more closely in the next section.

Fixing $x,y \in X$, we use the above to construct certain continuous
functions over $\Lambda$.  In particular, define
\begin{equation*}
  \begin{split}
    E^\pm(x)(\lambda) &= \psi^\pm_\lambda(x) \\
    D(x,y)(\lambda) &= \prod_{j=1}^d (1 - \psi_j(x)
  \psi_j(y)^*)^{\lambda_j}.
  \end{split}
\end{equation*}
Then
\begin{equation}
  \label{eq:7}
  E^+(x)(\lambda)E^+(y)(\lambda)^* - E^-(x)(\lambda)E^-(y)(\lambda)^*
  = D(x,y)(\lambda),
\end{equation}
and $E^+(x)(\lambda)E^+(x)(\lambda)^* \geq 1$.

\subsection{Auxiliary test functions for ample preorderings}
\label{subsec:auxil-test-funct-ample}

We show in this subsection that when $\Lambda$ is an ample
preordering, the auxiliary test function $\sigma_\lambda$ can be
modified so as to obtain a matrix valued $\HL$
function.

\begin{theorem}
  \label{thm:ext-aux-test-fns-ample-case}
  Assume that $\Psi$ is a finite collection of test functions over a
  set $X$, $\Lambda$ an ample preordering with maximal element
  $\lambda^m$.  Then for $\lambda\in\Lambda$, $n = 2^{|\lambda|-1}$,
  the auxiliary test function $\sigma_\lambda$ can be defined so that
  $([1_n] - \sigma_\lambda \sigma_\lambda^*)*(k_s\otimes 1_n) \geq 0$
  and for all $x\in X$, $\|\sigma_\lambda(x)\| < 1$; that is,
  $\sigma_\lambda \in H_1^\infty({\mathcal K}_{\Lambda,\mathbb C^n})$.
  If the test functions are in $A(X, \mathcal K_\Lambda)$, then we
  have $\sigma_\lambda \in A({\mathcal K}_{\Lambda,\mathbb C^n})$.
  Furthermore, $k\in {\mathcal K}_{\Lambda,\mathcal H}$ if and only if
  $\,([1_n] - \sigma_{\lambda^m} \sigma_{\lambda^m}^*)*(k\otimes 1_n)
  \geq 0$, $n = 2^{|\lambda^m|-1}$.  As a consequence, if we define a
  positive kernel $k_\Lambda$ by $k_\Lambda(x,y) =
  (1-\sigma_{\lambda^m}(x) \sigma_{\lambda^m}(y)^*)^{-1}$, then
  $\varphi \in \HLHb$ if and only if $([1_{n}\otimes
  1_{\mathcal{L(H)}}] - (1_{n}\otimes \varphi)(1_{n}\otimes
  \varphi)^*)*(k_\Lambda \otimes 1_{\mathcal{L(H)}}) \geq 0$.
\end{theorem}

\begin{proof}
  Fix $\lambda\in \Lambda$ and let $F$ be a finite subset of $X$.  Let
  $n = 2^{|\lambda|-1}$.  We have a Kolmogorov decomposition of the
  Szeg\H{o} kernel
  \begin{equation*}
    k_s(x,y) = 1_n \otimes \prod_{j=1}^d
    (1-\psi_j(x)\psi_j(y)^*)^{-1} = \kappa_s(x)\kappa_s(y)^*,
  \end{equation*}
  where $\kappa_s(x):\mathcal E \to \mathbb C^n$, and $\mathcal E$ is
  a Hilbert space equal to the closed span over $x\in X$ of the
  functions $\kappa_s(x)^*$.

  Write $k_F$ for the restriction of $k$ to $F\times F$ and let
  ${\mathcal E}_F = \bigvee_{x\in F} \kappa_s(x)^* \subset \mathcal
  E$.  By Corollary~\ref{cor:Szego-ker-inv}, $k_F$ has closed range
  isomorphic to $\mathbb C^{n|F|}$.  Denote by $\kappa_F$ the column
  over $x\in F$ of $k_s(x)$ restricted to ${\mathcal E}_F$.  Then
  $\mathrm{ran}\, \kappa_F^* = {\mathcal E}_F$ and $\kappa_F
  \kappa_F^* = k_F$.  Let $\Psi^+_F$ be a matrix with diagonal
  elements $\psi^+_\lambda(x)$, which maps from the range of
  $\kappa_F$, and let $P^+_F$ be the orthogonal projection onto the
  range of $\Psi^{+*}_F$.  Also write $\sigma_F$ for the direct sum
  over $x\in F$ of $\sigma_\lambda(x)$.  Define
  \begin{equation*}
    Q_F = \kappa_F^{-1} P^+_F \kappa_F \qquad\text{and}\qquad
    G^o_F = P_{\mathrm{ran}\,Q_F^*} \kappa_F^{-1} \sigma_F \kappa_F.
  \end{equation*}
  Here $P_{\mathrm{ran}\,Q_F^*}$ is the orthogonal projection onto the
  range of $Q_F^*$.

  It is clear that $Q_F$ is a (not necessarily orthogonal)
  projection.  Also, $\Psi^+_F \kappa_F = \Psi^+_F P^+_F \kappa_F =
  \Psi^+_F \kappa_F Q_F$ and $\kappa_F G^o_F = \sigma_F \kappa_F =
  P^+_F \sigma_F \kappa_F = \kappa_F Q_F G^o_F$.  Now,
  \begin{equation*}
    \begin{split}
      0 & \leq \Psi^+_F (k_F - \sigma_F k_F \sigma_F^*)\Psi^{+*}_F \\
      & = \Psi^+_F \kappa_F Q_F (1 - G^o_F G^{o*}_F)\kappa_F^* Q_F^*
      \Psi^{+*}_F,
    \end{split}
  \end{equation*}
  and so $P_{\mathrm{ran}\,Q_F^*} - G^o_F G^{o*}_F \geq 0$.
  Consequently, $\|G^o_F\| \leq 1$.

  Let
  \begin{equation*}
    {\mathcal G}_F = \left\{ G_F\in \mathcal{L}({\mathcal E}_F) :
      \|G_F\| \leq 1 \text{ and } P_{\mathrm{ran}\,Q_F^*}G_F = G^o_F
    \right\}.
  \end{equation*}
  This is a compact subset of the unit ball of $\mathcal{L}({\mathcal
    E}_F)$, and for any $G_F \in {\mathcal G}_F$,
  \begin{equation*}
    \begin{split}
      & \Psi^+_F \kappa_F (1 - G_F G^*_F)\kappa_F^* \Psi^{+*}_F 
      = \,\Psi^+_F \kappa_F Q_F (1 - G_F G^*_F)\kappa_F^* Q_F^*
      \Psi^{+*}_F \\
      & = \,\Psi^+_F (k_F - \sigma_F k_F \sigma_F^*)\Psi^{+*}_F 
      \geq 0.
    \end{split}
  \end{equation*}
  If we define
  \begin{equation*}
    {\mathcal S}_F = \left\{ S_F\in \mathcal{L}(\mathrm{ran}\, k_F) :
        S_F = \kappa_F G_F \kappa_F^{-1} \text{ for some } G_F \in
        {\mathcal G}_F \right\},
  \end{equation*}
  then by construction this set is nonempty, $k_F - S_F k_F S_F^* =
  \kappa_F (1 - G_F G^*_F)\kappa_F^* \geq 0$, and
  \begin{equation}
    \label{eq:8}
    \begin{split}
      \Psi^+_F (k_F - S_F k_F S_F^*)\Psi^{+*}_F & \,= \Psi^+_F (k_F -
    \sigma_F k_F \sigma_F^*)\Psi^{+*}_F \\
    & = \,{\left((\psi_\lambda^+\psi_\lambda^{+\,*} * 
      ([1_n] - \sigma_\lambda\sigma_\lambda^*) * k_s)(x,y)
    \right)}_{x,y\in F}.
    \end{split}
  \end{equation}

  The set ${\mathcal S}_F$ is isomorphic to ${\mathcal G}_F$, and so
  is also compact.  It is natural to define the norm of $S_F \in
  {\mathcal S}_F$ to equal the norm of the associated $G_F \in
  {\mathcal G}_F$.  If $F' \supset F$ and $G_{F'} \in {\mathcal
    G}_{F'}$, then since $\kappa_F$ is the restriction of
  $\kappa_{F'}$ to $\mathcal E_F$, $G_F := P_{\mathcal E_F}G_{F'} \in
  {\mathcal G}_F$.  Thus for $S_{F'} = \kappa_{F'} G_{F'}
  \kappa_{F'}^{-1}$ and $S_F = \kappa_F G_F \kappa_F^{-1}$, the map
  $\pi_F^{F'} :\mathcal S_{F'} \to \mathcal S_F$ defined by
  \begin{equation*}
    \pi_F^{F'} (S_{F'})  = S_F
  \end{equation*}
  is contractive and so continuous.  Let $\mathcal F$ be the
  collection of all finite subsets of $X$ partially ordered by
  inclusion.  The triple $(\mathcal S_F,\pi_F^{F'},\mathcal F)$ is an
  inverse limit of nonempty compact spaces, and so by Kurosh's
  Theorem~\cite[p.~30]{MR1882259}, for each $F\in\mathcal F$ there is
  an $S_F \in \mathcal S_F$ so that whenever $F,F'\in\mathcal F$ and
  $F\subset F'$,
  \begin{equation*}
    \pi_F^{F'} ((S_{F'}))  = (S_F).
  \end{equation*}
  We can thus define ${\tilde\sigma}_{\lambda}$ on $X$ by
  \begin{equation*}
    {\tilde\sigma}_{\lambda}(x) = S_{\{x\}}.
  \end{equation*}
  By construction,
  \begin{equation}
    \label{eq:16}
    ([1_n] - {\tilde\sigma}_\lambda {\tilde\sigma}_\lambda^* ) *
      k_s \geq 0;
  \end{equation}
  that is, ${\tilde\sigma}_\lambda \in H^\infty(X,{\mathcal
    K}_{\Lambda,\mathbb C^n})$.

  It follows in particular from \eqref{eq:16} that for any kernel $k$
  subordinate to the Szeg\H{o} kernel $k_s$ and $n =
  2^{|\lambda^m|-1}$,
  \begin{equation}
    \label{eq:15}
    ([1_n] - {\tilde\sigma}_{\lambda^m}
    {\tilde\sigma}_{\lambda^m}^*)*(k\otimes 1_n) \geq 0.
  \end{equation}
  Also, by \eqref{eq:8} any kernel $k$ for which \eqref{eq:15} holds
  is subordinate to the Szeg\H{o} kernel.  Hence the collection of
  auxiliary test functions constructed gives the same set of
  admissible kernels, and so generates $\HLH$ with the same norm.

  Let $\lambda = \lambda^m$, with ${\tilde\sigma}_{\lambda}$
  constructed as above.  For the time being, we assume $\Lambda =
  \{\lambda^m\}$.  Suppose that $x\in X$ has the property that
  $\|{\tilde\sigma}_{\lambda}(x)\| = 1$.  Since
  ${\tilde\sigma}_{\lambda}(x) \in M_n(\mathbb C)$ for $n =
  2^{|\lambda|-1} < \infty$, there is some $f\in {\mathbb C}^n$ such
  that $\|{\tilde\sigma}_{\lambda}(x)f\| = \|f\| = 1$.

  The test functions all have absolute value less than one in $X$, so
  for $y\neq x\in X$, the Szeg\H{o} kernel satisfies $k_s(x,y)\neq 0$
  and $k_s(x,x)>0$, and by Corollary~\ref{cor:Szego-ker-inv} when
  restricted to the two point set $\{x,y\} \subset X$, $k_s$ is
  invertible.  Consequently, $k_s(x,y) = k_s(x,x)^{1/2} g
  k_s(y,y)^{1/2}$, where $|g| < 1$.

  Let $k_s(x,y) = \ip{k_x}{k_y}$ be the Kolmogorov decomposition of
  $k_s$.  Since \eqref{eq:16} holds over the set $\{x,y\}$ and
  \begin{equation*}
    \ip{f\otimes k_x}{f\otimes k_x} -
    \ip{{\tilde\sigma}_{\lambda}(x)f\otimes
      k_x}{{\tilde\sigma}_{\lambda}(x)f\otimes k_x} = 0,
  \end{equation*}
  it follows that 
  \begin{equation*}
    \ip{f\otimes k_x}{f\otimes k_y} -
    \ip{{\tilde\sigma}_{\lambda}(x)f\otimes
      k_x}{{\tilde\sigma}_{\lambda}(y)f\otimes k_y} = 0.
  \end{equation*}
  Hence there is an isometry $V$ such that
  ${\tilde\sigma}_{\lambda}(x)f = {\tilde\sigma}_{\lambda}(y)f = Vf$.

  Define the $M_n(\mathbb C)$ valued kernel $\tilde k(z,w)$ to be
  equal to $P_f$, the projection onto the span of $f$ if $z,w\in
  \{x,y\}$, and $0$ otherwise.  Since $([1_n] -
  {\tilde\sigma}_{\lambda} {\tilde\sigma}_{\lambda}^*)* \tilde k =0$,
  $\tilde k$ is admissible.  On the other hand, to be admissible, it
  must also be the case that there is some positive kernel $F$ such
  that $\tilde k = k_s * F$.  Obviously, $F(z,w)$ must be zero if
  $z,w\notin \{x,y\}$ and $F(z,z) = k_s(z,z)^{-1}\otimes P_f$ for $z
  =x$ or $y$.  Positivity then implies that $F(x,y) = (k_s(x,x)^{-1}
  g'k_s(y,y)^{-1})\otimes P_f$ with $|g'| \leq 1$.  Hence
  $(k_s*F)(x,y) = gg'P_f \neq P_f$, giving a contradiction.  We
  conclude that for all $x$, $\|{\tilde\sigma}_{\lambda^m}(x)\| < 1$.

  Now suppose that $\Lambda$ is any ample preordering.  If $\lambda
  \in\Lambda$, $\lambda \neq \lambda^m$, has the property that at some
  $x$, $\|{\tilde\sigma}_{\lambda}(x)\| = 1$, then an identical
  argument shows that the norm is achieved on a subspace $\mathcal F$,
  and that there is an isometry $V$ such that for all $f\in \mathcal
  F$ and all $y\in X$, ${\tilde\sigma}_{\lambda}(y)f = Vf$.
  Consequently, we can change ${\tilde\sigma}_{\lambda}$ so that
  ${\tilde\sigma}_{\lambda}(y)f = 0$ for all $y\in X$.  Testing
  against the Szeg\H{o} kernel, it is clear that the resulting
  function is still in $H^\infty(X,{\mathcal K}_{\Lambda,\mathbb
    C^n})$ for appropriate $n$ and now satisfies
  $\|{\tilde\sigma}_\lambda (x)\| < 1$.
\end{proof}

\begin{corollary}
  \label{cor:aux-test-fns-for-polydisk}
  Let $d\in\mathbb N$ and $n = 2^{d-1}$.  There is a function $\sigma
  \in H_1^\infty(\mathbb D^d, M_n(\mathbb C))$ such that the set of
  positive kernels $\mathcal K$ with the property that $k\in \mathcal
  K$ if and only if $([1_n] - \sigma\sigma^*)* k \geq 0$ are all
  subordinate to $1_n\otimes k_s$, where $k_s$ is the the Szeg\H{o}
  kernel $\prod_{j=1}^d (1-z_jz_j^*)^{-1}$, and so $H^\infty(\mathbb
  D^d,\mathcal K_{\sigma, \mathcal H}) = H^\infty(\mathbb D^d,
  \mathcal{L(H)})$.
\end{corollary}

\begin{proof}
  This is a consequence of the last theorem
  and~\ref{cor:po_for_polydisk}.
\end{proof}

\subsection{Representations of $C_b(\Lambda)$}
\label{subsec:representations}

As noted previously, since $|\Lambda| < \infty$, a unital
representation $\rho: C_b(\Lambda)\to \mathcal{L(E)}$, $\mathcal E$ a
Hilbert space, will have the form
\begin{equation*}
  \rho(f) = \sum_{\lambda\in\Lambda} P_\lambda \otimes f(\lambda),
\end{equation*}
where the $P_\lambda$s are orthogonal projections with orthogonal
ranges and $\sum_{\lambda\in\Lambda} \ran P_\lambda \otimes \mathbb
C^{2|\lambda|-1} = \mathcal E$.  We then naturally define
\begin{equation*}
  Z^\pm(x) := \sum_{\lambda\in\Lambda} P_\lambda \otimes
  E^\pm(x)(\lambda) = \sum_{\lambda\in\Lambda} P_\lambda \otimes
  \psi^\pm_\lambda(x)
\end{equation*}
and
\begin{equation*}
  R(x,y) = \rho(D(x,y)) := \sum_{\lambda\in\Lambda} P_\lambda
  \otimes D(x,y)(\lambda).
\end{equation*}
It follows that
\begin{equation*}
  Z^+(x)Z^+(y)^* - Z^-(x)Z^-(y)^* = R(x,y),
\end{equation*}
and $Z^+(x)Z^+(x)^* \geq 1$.  In particular, $Z^+(x)Z^+(x)^*$ is
invertible.

The operator $Z^+(x)$ has a right inverse given by
\begin{equation*}
  Y(x) := \sum_{\lambda\in\Lambda} P_\lambda \otimes \omega_\lambda(x)
  = \sum_{\lambda\in\Lambda} P_\lambda \otimes \psi^+_\lambda(x)^*
  |\psi^+_\lambda(x)|^{-2},
\end{equation*}
and so $P(x) = Y(x) Z^+(x)$ is the orthogonal projection onto $\clran
Z^+(x)^*$.

Setting
\begin{equation*}
  S(x) := Y(x) Z^-(x) = \sum_{\lambda\in\Lambda} P_\lambda \otimes
  \sigma_\lambda(x),
\end{equation*}
we have $Z^-(x) = Z^+(x) S(x)$.  Also, since $P(x) Y(x) = Y(x) Z^+(x)
Y(x) = Y(x)$, we have $P(x) S(x) = S(x)$.  Thus
\begin{equation*}
  \begin{split}
    1 - S(x)S(x)^* &= 1 - Y(x)Z^-(x)Z^-(x)^*Y(x)^* \\
    & = 1 - P(x)P(x) + P(x)P(x) - Y(x)Z^-(x)Z^-(x)^*Y(x)^* \\
    & = 1 - P(x)P(x) + Y(x)Z^+(x)Z^+(x)^*Y(x)^* -
    Y(x)Z^-(x)Z^-(x)^*Y(x)^* \\
    & = 1 - P(x)P(x) + Y(x)\left( Z^+(x)Z^+(x)^* -
      Z^-(x)Z^-(x)^*\right) Y(x)^* > 0. \\
  \end{split}
\end{equation*}

In case the preordering is ample, in the definition of $S(x)$ we may
use Theorem~\ref{thm:ext-aux-test-fns-ample-case} to replace
$\sigma_\lambda$ by a corresponding element of $H^\infty(X,{\mathcal
  K}_{\Lambda,\mathbb C^{2^{|\lambda|-1}}})$.

We summarize in the following lemma.

\goodbreak

\begin{lemma}
  \label{lem:factoring_Z-}
  Let $x\in X$.
  \begin{enumerate}
  \item $Z^+(x)Z^+(x)^* - Z^-(x)Z^-(x)^* \geq 0$.
  \item The operator $Y(x)$ is a right inverse of $Z^+(x)$ and
    $P(x) := Y(x) Z^+(x)$ is the orthogonal projection onto $\ran
    Z^+(x)^*$.
  \item The operator $S(x) = Y(x) Z^-(x): \clran Z^-(x)^* \to \ran
    Z^+(x)^*$ $($or a corresponding element of $H^\infty(X,{\mathcal
      K}_{\Lambda,\mathbb C^{2^{|\lambda|-1}}})$ in case of an ample
    preordering$)$ satisfies $Z^-(x) = Z^+(x) S(x)$ and has the
    property that $\|S(x)\| < 1$.
  \end{enumerate}
\end{lemma}

Although it has not been explicitly indicated, it is worth bearing in
mind that $Z^+$, $Z^-$, $S$ and so on, depend both on $\Lambda$ and
the choice of representation, and we will at times make this
dependence explicit by writing $Z^+_{\Lambda,\rho}$,
$Z^-_{\Lambda,\rho}$, $S_{\Lambda,\rho}$, etc.

\section{Transfer functions, Brehmer representations and dilations}
\label{sec:brehm-repr-transf}

\subsection{The transfer function algebra}
\label{subsec:transfer-functions}

In the standard manner, we define a \textbf{$C_b(\Lambda)$-unitary
  colligation} $\Sigma$ as a triple $(U,\mathcal E, \rho)$, where
$\mathcal E$ is a Hilbert space, $U = \begin{pmatrix} A&B \\ C&D
\end{pmatrix} \in \mathcal B(\mathcal E \oplus \mathcal H)$ is a
unitary operator, and $\rho: C_b(\Lambda) \to \mathcal{L}(\mathcal E)$
a unital $*$-representation.

\begin{definition}
  Assume the notation from Lemma~\ref{lem:factoring_Z-}.  Given a
  $C_b(\Lambda)$-unitary colligation $\Sigma$, we define the
  \textbf{transfer function} $W_\Sigma:X \to \mathcal{L(H)}$
  associated to $\Sigma$ as
  \begin{equation*}
    W_\Sigma(x) := D + C S(x) (1 -  A S(x))^{-1} B,
  \end{equation*}
  where $S = S_{\Lambda,\rho}$.  Write
  \begin{equation*}
    \mathcal{T}_1(X,\Lambda, \mathcal H) := \{W_\Sigma : \Sigma \text{
      a unitary colligation }\},
  \end{equation*}
  and $\mathcal{T}(X,\Lambda, \mathcal H)$ for the scalar multiples of
  elements in $\mathcal{T}_1(X,\Lambda, \mathcal H)$.  It is clear
  that $W\in \mathcal{T}(X,\Lambda, \mathcal H)$ will not in general
  be uniquely represented as a multiple of a single element of
  $\mathcal{T}_1(X,\Lambda, \mathcal H)$.  For $W \in
  \mathcal{T}(X,\Lambda, \mathcal H)$, define a norm by
  \begin{equation}
    \label{eq:12}
    \|W\| := \inf\left\{c\geq 0 : W = c W_\Sigma \text{ for some
        unitary colligation }\Sigma \right\}.
  \end{equation}
  (We show in Theorem~\ref{thm:trfr-fns-span-alg} below that
  $\|\cdot\|$ on $\mathcal{T}(X,\Lambda, \mathcal H)$ really is a
  norm.)  Finally, we write\break $\mathcal{T}^A(X,\Lambda, \mathcal
  H)$ for the set of those $W\in\mathcal{T}(X,\Lambda, \mathcal H)$
  which extend continuously to $\overline{X}$
\end{definition}

The formula gives the standard form of the transfer function when
$\Lambda = \Lambda_1 = \{e_\psi\}$.  Again, one should bear in mind
that $S$ depends on $\rho$.

More generally, we might also consider
\textbf{$C_b(\Lambda)$-contractive colligations} by allowing $U$ to be
contractive rather than unitary, and then likewise define a transfer
function.  As it happens, this does not enlarge the collection of
functions we obtain through the apparently more restrictive unitary
colligations, since any any contractive operator has a unitary
dilation.

\begin{lemma}
  \label{lem:contractive-to-unitary-transfer-fn}
  Let $W_\Sigma: X\to \mathcal{L(H)}$ be a transfer function obtained
  via a contractive colligation $\Sigma = (U,\mathcal E, \rho)$.  Then
  there is unitary colligation $\tilde\Sigma = (\tilde U,
  \tilde{\mathcal E}, \tilde\rho)$ such that $W_\Sigma =
  W_{\tilde\Sigma}$.
\end{lemma}

\begin{proof}
  At least one of the projections, say $P_{\lambda_0}$ will be
  nonzero, so we take $g$ to be a unit vector from its range.  Let
  $\{a_j\}$ be an orthonormal basis for $\mathbb C^{n_{\lambda_0}}$,
  where $n_{\lambda_0} = 2^{|\lambda_0|-1}$, and define $\mathcal Q =
  \bigvee_j (g\otimes n_j)$.  Elements of $\mathcal Q$ have the form
  $e = \sum_j \beta_j g\otimes a_j$, where $\beta = (\beta_j)\in
  \mathbb C^{n_{\lambda_0}}$.  In addition, if $e' = \sum_j {\beta'}_j
  g\otimes a_j$, then $\ip{e'}{e} = \sum_j
  {\beta'}_j\overline{\beta_j}$.

  By assumption,
  \begin{equation*}
    U =
    \begin{pmatrix}
      A & B \\ C & D
    \end{pmatrix}
    : \mathcal E \oplus \mathcal H \to \mathcal E \oplus \mathcal H
  \end{equation*}
  is a contraction.  Let
  \begin{equation*}
    {\tilde D}_U =
    \begin{pmatrix}
      {\tilde D}_1 \\ {\tilde D}_2
    \end{pmatrix}
    \qquad\text{and}\qquad
    D_U =
    \begin{pmatrix}
      D_1 \\ D_2
    \end{pmatrix}
  \end{equation*}
  be defect operators for $U$ and $U^*$, respectively (so $1-U^*U =
  {\tilde D}_U{\tilde D}_U^*$ and $1-UU^* = D_UD_U^*$ with defect
  spaces ${\tilde{\mathcal D}}_U = \clran {\tilde D}_U^*$ and
  ${\mathcal D}_U = \clran D_U^*$), and $\begin{pmatrix} L^* & {\tilde
      D}_U^* \\ D_U & U \end{pmatrix}$ the corresponding Julia
  operator, which is unitary from ${\tilde{\mathcal D}}_U \oplus
  (\mathcal E \oplus \mathcal H)$ to ${\mathcal D}_U \oplus (\mathcal
  E \oplus \mathcal H)$.  Then there is a unitary dilation of $U$ of
  the form
  \begin{equation*}
    \begin{split}
      U' =& \overbrace{S_2\oplus\cdots \oplus S_2}^{n_{\lambda_0} - 1}
      \oplus \\
      &\left(\begin{array}{ccccccc|c}
          & \ddots &&&&&& \vdots \\
          && 1 &&&&& \\
          &&& 0 &&&& \\
          \ddots &&&& 1 &&& \\
          & 1 &&&& 0 && \\
          && 0 &&& L^* & {\tilde D}_1^* & {\tilde D}_2^* \\
          &&& 1 &&0&0&0 \\
          &&&& 0 & D_1 & A & B \\\hline
          \cdots &&&&& D_2 & C & D \\
        \end{array}\right),
    \end{split}
  \end{equation*}
  where unspecified entries are $0$ and the blocks act on $(\mathcal
  K^{2n_{\lambda_0} - 1} \oplus \mathcal E)\oplus\mathcal H$,
  $\mathcal K = \cdots \oplus {\tilde{\mathcal D}}_U \oplus {\mathcal
    D}_U \oplus {\tilde{\mathcal D}}_U \oplus {\mathcal D}_U$, a
  direct sum of defect spaces.  The operator $S_2$ is a unitary
  operator on $\mathcal K \oplus \mathcal K$ defined as
  \begin{equation*}
    S_2 =
    \left(\begin{array}{cccccc|cccccc}
        &&& \ddots &&&&&&&& \\
         \ddots &&&& 1 &&&&&&& \\
        & 1 &&&& 0 &&&&&& \\
        && 0 &&&& 1 &&&&& \\
        &&& 1 &&&& 0 &&&& \\\hline
        &&&& 0 &&&& 1 &&& \\
        &&&&& 1 &&&& 0 && \\
        &&&&&& 0 &&&& 1 & \\
        &&&&&&& 1 &&&& \ddots \\
        &&&&&&&& \ddots &&& \\
      \end{array}\right).
  \end{equation*}
  (Here we have made the obvious identification of the direct sum
  defining $\mathcal K$ written in the forward and backward direction
  with ${\tilde{\mathcal D}}_U$ and ${\mathcal D}_U$ reversed.)

  Let $\tilde{\mathcal E} = (\mathcal K \otimes \mathcal E) \oplus
  \mathcal E$.  Define an isometry $Q : \mathcal K \otimes \mathcal Q
  \to \mathcal K^{2n_{\lambda_0}-1}$ by
  \begin{equation*}
    Q(k\otimes e) = (\beta_1 k, 0 , \beta_2 k,0, \dots ,
    \beta_{n_{\lambda_0} - 1}k, 0 ,\beta_{n_{\lambda_0}}k ), \quad
    \text{where }e = \sum_j \beta_j g\otimes a_j,
  \end{equation*}
  extending linearly.  Let $P$ to be the orthogonal projection onto
  $(\ran Q^* \oplus \mathcal E \oplus \mathcal H)^\bot$ in
  $\tilde{\mathcal E}$, and set
  \begin{equation*}
    \tilde U = P \oplus (Q^* \oplus P_{\mathcal E} \oplus P_{\mathcal
      H}) U' (Q\oplus P_{\mathcal E} \oplus P_{\mathcal H}),
  \end{equation*}
  where $P_{\mathcal E}$, $P_{\mathcal H}$ are the orthogonal
  projections from $\tilde{\mathcal E} \oplus \mathcal H$ onto
  $\mathcal E$ and $\mathcal H$.  This is unitary on $\tilde{\mathcal
    E} \oplus \mathcal H$.  We view it as a colligation by setting
  \begin{equation*}
    \begin{split}
      \tilde A &= P \oplus (Q^* \oplus P_{\mathcal E}) U' (Q\oplus
      P_{\mathcal E}) \\
      \tilde B &= (Q^* \oplus P_{\mathcal E}) U' P_{\mathcal H} \\
      \tilde C &= P_{\mathcal H} U' (Q\oplus P_{\mathcal E}) \\
      \tilde D &= P_{\mathcal H} U' P_{\mathcal H} = D.
    \end{split}
  \end{equation*}
  Define a unital representation
  \begin{equation*}
    \tilde\rho = (1_{\mathcal K} \otimes \rho) \oplus \rho :
    C_b(\Lambda) \to \tilde{\mathcal E}.
  \end{equation*}
  Recall that using $\rho(f) = \sum_\lambda P_\lambda \otimes f$, we
  defined $S(x) = \sum_\lambda P_\lambda \otimes \sigma(x)$.  If we
  now set ${\tilde P}_\lambda = (1_{\mathcal K} \otimes P_\lambda)
  \oplus P_\lambda$, we can likewise define
  \begin{equation*}
    \tilde S(x) = \sum_\lambda {\tilde P}_\lambda \otimes \sigma(x)
    \in \mathcal L (\tilde{\mathcal E}),
  \end{equation*}
  and from this, a transfer function
  \begin{equation*}
    \tilde W(x) = \tilde D + \tilde C \tilde S(x) (1_{\tilde{\mathcal
        E}} - \tilde A \tilde S(x))^{-1} \tilde B.
  \end{equation*}

  We verify that $\tilde W(x) = W(x)$ by showing $\tilde C \tilde S(x)
  (\tilde A \tilde S(x))^n \tilde B = CS(x)(AS(x))^n B$ for
  $n=0,1,\dots$.  Fix $h\in \mathcal H$.  Then
  \begin{equation*}
    \tilde B h = (Q^* \oplus P_{\mathcal E})
    \begin{pmatrix}
      \vdots \\ 0 \\ {\tilde D}_2^* h \\ 0 \\ Bh
    \end{pmatrix}
    = (k^0 \otimes e^0) \oplus Bh,
  \end{equation*}
  where $k^0 = {\begin{pmatrix} \cdots & 0 & {\tilde D}_2^* h & 0
  \end{pmatrix}}^t$ and $e^0 = g\otimes a_1$ (since in the column
  vector, $k^0$ occurs in the first copy of $\mathcal K$ in $\mathcal
  K^{2n_{\lambda_0} - 1}$).  Now
  \begin{equation*}
    S(x) e^0 = S(x) (g\otimes a_1) = \sum_j \beta_j^1 h\otimes a_j
  \end{equation*}
  for some $(\beta_1^1,\dots, \beta_{n_{\lambda_0}}) \in \mathbb
  C^{n_{\lambda_0}}$.  Setting $e^1 = \sum_j \beta_j^1 h\otimes a_j
  \in \mathcal Q$,
  \begin{equation*}
    \tilde S(x)\tilde B h = (k^0\otimes e^1)\oplus S(x)Bh.
  \end{equation*}
  Then
  \begin{equation*}
    \tilde C \tilde S(x)\tilde B h =
    \begin{pmatrix}
      0 & \cdots & 0 &
      \begin{pmatrix}
        \cdots & 0 & D_2
      \end{pmatrix}
       & C
    \end{pmatrix}
    \begin{pmatrix}
      \beta^1_{n_{\lambda_0}} k^0 \\ 0 \\ \vdots \\ \beta^1_1 k^0 \\
      S(x)Bh
    \end{pmatrix}
    = CS(x)Bh,
  \end{equation*}
  proving the claim when $n=0$.

  For $n=1$,
  \begin{equation*}
    \begin{split}
      & S_2\oplus\cdots \oplus S_2 \oplus
      \begin{pmatrix}
        & \ddots &&&&& \\
        && 1 &&&& \\
        &&& 0 &&& \\
        \ddots &&&& 1 && \\
        & 1 &&&& 0 & \\
        && 0 &&& L^* & {\tilde D}_1^* \\
        &&& 1 &&0&0 \\
        &&&& 0 & D_1 & A
      \end{pmatrix}
      \begin{pmatrix}
        \beta^1_{n_{\lambda_0}} k^0 \\ 0 \\ \vdots \\ \beta^1_1 k^0 \\
        S(x)Bh
      \end{pmatrix}
      \\
      & = 
      \begin{pmatrix}
        \beta^1_{n_{\lambda_0}} k^1 \\ 0 \\ \vdots \\ \beta^1_1 k^1 \\
        0
      \end{pmatrix}
      +
      \begin{pmatrix}
        0 \\ 0 \\ \vdots \\ {k'}^0 \\
        AS(x)Bh
      \end{pmatrix},
    \end{split}
  \end{equation*}
  where $k^1 = {\begin{pmatrix} \cdots & 0 & {\tilde D}_2^* h & 0 & 0
      & 0 \end{pmatrix}}^t$ (that is, $k^0$ with entries shifted up by
  two positions) and ${k'}^0 = {\begin{pmatrix} \cdots & 0 & {\tilde
        D}_1^*S(x)B h & 0 \end{pmatrix}}^t$.  Notice that in both
  cases, only even numbered entries in odd numbered spaces $\mathcal
  K$ of $\mathcal K^{2n_{\lambda_0} - 1}$ are non-zero.  Also, these
  vectors are in the kernel of $P$.  From this, we conclude that
  \begin{equation*}
    \tilde A \tilde S(x)\tilde B h = (k^1 \otimes e^1 + {k'}^0 \otimes
    {e'}^1) \oplus AS(x)Bh,
  \end{equation*}
  where ${e'}^1$ is likewise a vector in $\mathcal Q$.  Applying
  $\tilde S(x)$, we get
  \begin{equation*}
    (k^1 \otimes e^2 + {k'}^0 \otimes {e'}^2) \oplus S(x)AS(x)Bh
  \end{equation*}
  for some vectors $e^2$ and ${e'}^2$ in $\mathcal Q$.  Because
  $\begin{pmatrix} \cdots & 0 & D_2 \end{pmatrix}$ only acts
  nontrivially on odd labeled entries of $\mathcal K$ in the first
  $\mathcal K$ of $\mathcal K^{2n_{\lambda_0} - 1}$, we conclude that
  $\tilde C \tilde S(x) \tilde A \tilde S(x)\tilde B h =
  CS(x)AS(x)Bh$, proving the case when $n=1$.

  Repeated application of $\tilde A \tilde S(x)$ to vectors in
  $\mathcal Q \subset \mathcal K^{2n_{\lambda_0} - 1}$ where the only
  nonzero entries are in the odd labeled spaces and within those
  spaces, only in the even labeled entries, yields vectors of the
  same sort.  An induction argument then finishes the proof.
\end{proof}

While we only stated and proved the last result in the specific case
we need later in the paper, minor alterations would allow for it to
cover cases where the test functions are operator valued (rather than
simply matrix valued) and where there are infinitely many of them.

With Lemma~\ref{lem:contractive-to-unitary-transfer-fn} in hand, we
can show that $\mathcal{T}(X,\Lambda, \mathcal H)$ is a normed unital
algebra.

\begin{theorem}
  \label{thm:trfr-fns-span-alg}
  With norm $\|\cdot\|$ defined as in \eqref{eq:12}, unit $1_X(x) =
  1_\mathcal H$ and pointwise addition and multiplication, the set
  $\mathcal{T}(X,\Lambda, \mathcal H)$ is a normed unital algebra.
  Furthermore, any $W\in\mathcal{T}(X,\Lambda, \mathcal H)$ can be
  approximated uniformly in norm on compact subsets of $X$ by elements
  of $\mathcal{L(H)} \otimes \mathcal P_\Psi$, the operator valued
  polynomials in the test functions, while elements of
  $\mathcal{T}^A(X,\Lambda, \mathcal H)$ can be approximated uniformly
  on $X$ by such polynomials.  If $\mathcal H$ is finite dimensional,
  $\mathcal{L(H)} \otimes \mathcal P_\Psi$ is dense in
  $\mathcal{T}(X,\Lambda, \mathcal H)$ endowed with the supremum norm.
  Finally, if $W \in \mathcal{T}(X,\Lambda, \mathcal H)$,
  $W=c(D+CS(1-AS)^{-1}B)$, then $W$ can be approximated uniformly in
  norm on compact subsets of $X$ by polynomials in $S$ which are in
  $c\mathcal{T}_1(X,\Lambda, \mathcal H)$.
\end{theorem}

\begin{proof}
  We first show that $\mathcal{T}_1(X,\Lambda, \mathcal H)$ is convex.
  Let $W_{\Sigma_1},W_{\Sigma_2} \in \mathcal{T}_1(X,\Lambda, \mathcal
  H)$ and $t\in[0,1]$.  The operator
  \begin{equation*}
    \begin{split}
      U' = & 
      \begin{pmatrix}
        1&0&0&0 \\ 0&1&0&0 \\ 0&0&t^{1/2}&(1-t)^{1/2}
      \end{pmatrix}
      \begin{pmatrix}
        A_1&0&B_1&0 \\ 0&A_2&0&B_2 \\ C_1&0&D_1&0 \\ 0&C_2&0&D_2
      \end{pmatrix}
      \begin{pmatrix}
        1&0&0 \\ 0&1&0 \\ 0&0&t^{1/2} \\ 0&0&(1-t)^{1/2}
      \end{pmatrix}
      \\ & =
      \begin{pmatrix}
        A_1&0&t^{1/2}B_1 \\ 0&A_2&(1-t)^{1/2}B_2 \\
        t^{1/2}C_1&(1-t)^{1/2}C_2&tD_1+(1-t)D_2
      \end{pmatrix},
    \end{split}
  \end{equation*}
  being the product of contractions is a contraction.  If we set
  ${\mathcal{E}}' = \mathcal E_1 \oplus \mathcal E_2$ and $\rho' =
  \rho_1\oplus\rho_2$, then $tW_{\Sigma_1}+(1-t)W_{\Sigma_2} =
  W_{\Sigma'}$ where $\Sigma'$ is a contractive colligation.  Hence by
  the last lemma equals $tW_{\Sigma_1}+(1-t)W_{\Sigma_2} = W_\Sigma$
  where $\Sigma$ is some unitary colligation.

  Clearly, by taking the contractive colligation with $U = 0$, the
  function which is identically $0$ is in $\mathcal{T}_1(X,\Lambda,
  \mathcal H)$.  Hence by convexity, $tW_\Sigma \in
  \mathcal{T}_1(X,\Lambda, \mathcal H)$ for all $t\in[0,1]$, showing
  that $\mathcal{T}_1(X,\Lambda, \mathcal H)$ is barreled.

  Let $W_{\Sigma_1},W_{\Sigma_2} \in \mathcal{T}_1(X,\Lambda, \mathcal
  H)$ and define the unitary operator
  \begin{equation*}
    U =
    \begin{pmatrix}
      A_1&B_1C_2&B_1D_2 \\ 0&A_2&B_2 \\ C_1&D_1C_2&D_1D_2
    \end{pmatrix}
    =
    \begin{pmatrix}
      A_1&0&B_1 \\ 0&1&0 \\ C_1&0&D_1
    \end{pmatrix}
    \begin{pmatrix}
      1&0&0 \\ 0&A_2&B_2 \\ 0&C_2&D_2
    \end{pmatrix}.
  \end{equation*}
  Let $\mathcal{E} = \mathcal E_1 \oplus \mathcal E_2$ and $\rho =
  \rho_1\oplus\rho_2$, it follows that $W_{\Sigma_1}W_{\Sigma_2} =
  W_\Sigma$.

  To see that what we defined in \eqref{eq:12} is a norm, first of all
  note that if $W\in \mathcal{T}_1(X,\Lambda, \mathcal H)$, then
  $\|W\| \leq 1$.  It is also evident that $\|cW\| = |c|\|W\|$, and
  $\|W\| \geq 0$ with equality if and only if $W=0$.  Since
  $\mathcal{T}_1(X,\Lambda, \mathcal H)$ is convex, if $W_1,W_2\in
  \mathcal{T}(X,\Lambda, \mathcal H)$ and $W_1 = c_1 W_{\Sigma_1}$,
  $W_2 =c_2 W_{\Sigma_2}$, then
  \begin{equation*}
    \frac{c_1}{c_1+c_2}W_{\Sigma_1} + \frac{c_2}{c_1+c_2}W_{\Sigma_2}
    \in \mathcal{T}_1(X,\Lambda, \mathcal H).
  \end{equation*}
  Hence
  \begin{equation*}
    \begin{split}
      \|W_1+W_2\| &= \|c_1W_{\Sigma_1}+c_2W_{\Sigma_2}\| \\
      &= (c_1+c_2) \left\|\frac{c_1}{c_1+c_2}W_{\Sigma_1} +
        \frac{c_2}{c_1+c_2}W_{\Sigma_2} \right\| \\
      &\leq c_1+c_2.
    \end{split}
  \end{equation*}
  Taking the infimum over $c_1$ and $c_2$ as $W_{\Sigma_1}$ and
  $W_{\Sigma_2}$ range over those elements of
  $\mathcal{T}_1(X,\Lambda, \mathcal H)$ such that $W_1 = c_1
  W_{\Sigma_1}$ and $W_2 =c_2 W_{\Sigma_2}$ with $c_1,c_2 \geq 0$
  yields the triangle inequality.

  For any choice of representation $\rho: C_b(\Lambda) \to
  \mathcal{L(E)}$, and $U$ the identity operator, we get $W_\Sigma =
  1_X$, the function which is identically $1_{\mathcal{L(H)}}$ on $X$.
  More generally, if $D \in \mathcal{L(H)}$ is a contraction operator,
  and $U = \begin{pmatrix} 0&0\\0&D \end{pmatrix}$ for the same choice
  of $\mathcal E$ and $\rho$, $W_\Sigma = D$.

  If $\psi\in\Psi$ and we choose $\mathcal E = \mathcal{L(H)}$, $Z^+ =
  1_{\mathcal{L(H)}}$ and $Z^- = \psi\otimes 1_{\mathcal{L(H)}}$, then
  $S = \psi \otimes 1_{\mathcal{L(H)}}$.  So with $A=D=0$ and $B= C =
  1_{\mathcal{L(H)}}$, we get $W_\Sigma = \psi \otimes
  1_{\mathcal{L(H)}}$.  Then since $\mathcal{T}_1(X,\Lambda, \mathcal
  H)$ is closed under products, we also have for any $n \in \mathbb
  N^{|\Psi|}$, $\psi^n \otimes 1_{\mathcal{L(H)}} \in
  \mathcal{T}_1(X,\Lambda, \mathcal H)$.  This also then gives that
  $\psi^n T$ for any contraction $T\in \mathcal{L(H)}$.  Scaling and
  closure under addition yields that any operator valued polynomial in
  the test functions is in $\mathcal{T}(X,\Lambda, \mathcal H)$.

  The topology with which $X$ is endowed ensures that all test
  functions are continuous.  Hence for all $\lambda\in\Lambda$,
  $\psi^\pm_\lambda$ is also continuous, and thus
  $\psi^+_\lambda\psi^{+*}_\lambda$ is a continuous function bounded
  below by $1$, and so has a continuous inverse.  Consequently, any
  auxiliary test function $\sigma_\lambda$ is continuous.  (In the
  case that $\Lambda$ is an ample preordering, this was automatic,
  since $\sigma_\lambda \in H^\infty_1(X,{\mathcal K}_{\Lambda,\mathbb
    C^n})$ for some $n$, and all functions in this space are
  continuous.)  Since for any $\lambda\in\Lambda$,
  $\|\sigma_\lambda(x)\|<1$ for all $x\in X$, it follows that for any
  unitary colligation $\Sigma$, the associated function $S(x)$ is also
  continuous and has norm less than $1$.  Hence when $\mathcal{L(H)}$
  is given the norm topology, $W_\Sigma \in \mathcal{T}_1(X,\Lambda,
  \mathcal H)$ is continuous, and so $\mathcal{T}(X,\Lambda, \mathcal
  H) \subset C(X,\mathcal{L(H)})$.

  By definition, the test functions separate the points of $X$, and so
  by the Stone-Weierstrass theorem, the space of polynomials in the
  test functions, $\mathcal P_\Psi$, is dense in $\mathcal{T}(X,\Lambda,
  \mathbb C)$ with the supremum norm.  Hence if $\mathcal H$ is finite
  dimensional with orthonormal basis $(e_j)$ and $W\in
  \mathcal{T}(X,\Lambda, \mathcal H)$, then
  $W_{j\ell}:=\ip{We_j}{e_\ell} \in \mathcal{T}(X,\Lambda, \mathbb C)$.
  Let $\epsilon >0$.  For each $1\leq j,\ell \leq \dim\mathcal H$,
  find a polynomial $p_{j\ell}$ such that
  $\|W_{j\ell}-p_{j\ell}\|_\infty < \epsilon/(\dim\mathcal H)^2$.
  Then $\|W - (p_{j\ell})\|_\infty < \epsilon$, showing that
  $\mathcal{L(H)} \otimes \mathcal P_\Psi$ is norm dense in
  $W\in\mathcal{T}(X,\Lambda, \mathcal H)$.  From this argument, we see
  that $\mathcal{L(H)} \otimes \mathcal P_\Psi$ is weakly dense in
  $\mathcal{T}(X,\Lambda, \mathcal H)$ if $\dim\mathcal H$ is not
  finite.

  Now suppose that $W\in \mathcal{T}(X,\Lambda, \mathcal H)$ where the
  dimension of $\mathcal H$ is not necessarily finite.  Fix $\epsilon
  > 0$ and let $C$ be a compact subset of $X$.  Then $W(C)$ is
  compact, and a cover of $W(C)$ by open balls in $\mathcal{L(H)}$ by
  balls of radius less than $\epsilon/12$ has a finite subcover
  $\{U_j\}$.  For each $j$ choose $x_j\in W^{-1}(U_j)$.  Then for all
  $x\in X$, $\max_j\|W(x)-W(x_j)\| < \epsilon/6$.

  For each $j$ choose a finite dimensional subspace $\mathcal H_j
  \subset \mathcal H$ such that $\|W(x_j) - P_{\mathcal H_j} W(x_j)
  |_{\mathcal H_j}\| < \epsilon/6$.  Set ${\mathcal H}' = \bigvee_j
  \mathcal H_j$.  This is finite dimensional and for all $x\in X$,
  \begin{equation*}
    \begin{split}
      &\|W(x) - P_{{\mathcal H}'} W(x) |_{{\mathcal H}'}\| \\
      \leq & \max_j \|W(x)-W(x_j)\| +\max_j \|P_{{\mathcal
        H}'}(W(x)-W(x_j)|_{{\mathcal H}'}\| + \max_j \|W(x_j) -
    P_{{\mathcal H}'} W(x_j) |_{{\mathcal H}'}\| \\
    < & \epsilon/2.
    \end{split}
  \end{equation*}
  As we have seen, we can find $p\in \mathcal{L}(\mathcal{H}') \otimes
  \mathcal P_\Psi$ such that $\|p- P_{{\mathcal H}'} W(x) |_{{\mathcal
      H}'}\| < \epsilon/2$.  Extending $p$ to $\mathcal{L(H)} \otimes
  \mathcal P_\Psi$ by padding with $0$s, we then have that $\|W-p\| <
  \epsilon$, showing that we can approximate elements of
  $\mathcal{T}(X,\Lambda, \mathcal H)$ pointwise, and hence uniformly in
  norm on compact subsets of $X$, by polynomials in $\mathcal{L(H)}
  \otimes \mathcal P_\Psi$.  If we know that $W\in
  \mathcal{T}^A(X,\Lambda, \mathcal H)$, then by weak-$*$ compactness of
  $\overline{X}$, we claim that we can approximate elements of
  $\mathcal{T}(X,\Lambda, \mathcal H)$ uniformly in norm on $X$.

  It suffices to prove the last claim in the case $c=1$; that is, when
  $W = W_\Sigma$ for some colligation $\Sigma$.  Let $\tilde
  U = \begin{pmatrix} \tilde A & \tilde B\\ \tilde C & \tilde D
  \end{pmatrix}$, where $\tilde D = D$, $\tilde A$ is an $(M+2)\times
  (M+2)$ operator matrix with the first super-diagonal having all
  entries equal to $A$ and all other entries $0$, $\tilde B$ is an
  $M+2$ operator column with the first $M+1$ entries equal to
  $\tfrac{1}{\sqrt{M+1}}B$ and the last entry $0$, and $\tilde C$ is
  an $M+2$ operator row with the first entry $0$ and the remaining
  entries equal to $\tfrac{1}{\sqrt{M+1}}C$.  It is easily verified
  that $\tilde U$ is a contraction.  Set $\tilde S$ to the
  $(M+2)\times(M+2)$ diagonal matrix with diagonal entries equal to
  $S$.  Then
  \begin{equation*}
    W_M := \tilde D +\tilde C\tilde S(1-\tilde A \tilde
    S)^{-1} \tilde B \in \mathcal{T}_1(X,\Lambda, \mathcal H)
  \end{equation*}
  and
  \begin{equation}
    \begin{split}
      \label{eq:17}
      W_M &= D + CS\left(\tfrac{M}{M+1}1+\tfrac{M-1}{M+1}AS+\cdots +
    \tfrac{1}{M+1}(AS)^M\right)B \\
    & = D+CS(1-AS)^{-1}\left(1 - \tfrac{1}{m+1} (1-(AS)^{M+2})
      (1-AS)^{-1} \right)B.
    \end{split}
  \end{equation}
  Since by Lemma~\ref{lem:factoring_Z-} $S(x)$ is a strict
  contraction, we see that $W_M$ converges pointwise with $M$ to $W$.
  Arguing as above, we then get $W_M$ converging uniformly on compact
  subsets of $X$ to $W$.
\end{proof}

We write $\overline{\mathcal{T}}(X,\Lambda, \mathcal H)$ for the
completion of $\mathcal{T}(X,\Lambda, \mathcal H)$ in the norm
from~\eqref{eq:12}, and $\overline{\mathcal{T}}^A(X,\Lambda, \mathcal
H)$ for the closure of $\mathcal{T}^A(X,\Lambda, \mathcal H)$ in
$\overline{\mathcal{T}}(X,\Lambda, \mathcal H)$.

\begin{corollary}
  \label{cor:tr_fns_are_op_alg}
  The spaces $\{\overline{\mathcal{T}}(X,\Lambda, \mathcal H \otimes
  M_n(\mathbb C))\}_{n\in\mathbb N}$ and
  $\{\overline{\mathcal{T}}^A(X,\Lambda, \mathcal H \otimes
  M_n(\mathbb C))\}_{n\in\mathbb N}$ define unital operator algebra
  structures for $\overline{\mathcal{T}}(X,\Lambda, \mathcal H)$ and
  $\mathcal{T}^A(X,\Lambda, \mathcal H)$, respectively.
\end{corollary}

\begin{proof}
  Let $W\in \mathcal{T}_1(X,\Lambda, \mathbb C^n \otimes \mathcal H)$
  with unitary colligation $\Sigma = (U,\mathcal H,\rho)$, $U=
  \begin{pmatrix} A&B \\ C&D \end{pmatrix}$, and let , $X\in
  M_{mn}(\mathbb C)$, $Y\in M_{nm}(\mathbb C)$ with $\|X\|,\|Y\| \leq
  1$.  Then
  \begin{equation*}
    XWY = XDY + XC(1-AS)^{-1}BY = \tilde W
  \end{equation*}
  where $\tilde W \in \mathcal{T}_1(X,\Lambda, \mathcal H \otimes
  \mathbb C^m)$ has contractive colligation $\tilde\Sigma = (\tilde
  U,\mathcal H, \rho)$ with $\tilde U = \begin{pmatrix} A&BY \\
    XC&XDY \end{pmatrix} = \begin{pmatrix} 1&0 \\ 0&X
  \end{pmatrix}\begin{pmatrix} A&B \\ C&D \end{pmatrix}\begin{pmatrix}
    1&0 \\ 0&Y \end{pmatrix}$.  Hence $\overline{\mathcal{T}}(X,\Lambda,
  \mathcal H)$ is an abstract operator space.  Since for all $n$,
  $W_1, W_2\in \mathcal{T}_1(X,\Lambda, \mathcal H \otimes \mathbb C^n)$
  implies $W_1W_2 \in \mathcal{T}_1(X,\Lambda, \mathcal H \otimes
  \mathbb C^n)$, it follows that $\overline{\mathcal{T}}(X,\Lambda,
  \mathcal H)$ is an operator algebra.

  The case for $\mathcal{T}^A(X,\Lambda, \mathcal H)$ is proved
  similarly.
\end{proof}

The above provides something of a converse to the main result of
Jury~\cite{MR2945207} in a special case.

We close this subsection with a lemma which will be useful when we
want to construct representations on algebras of transfer functions.

\begin{lemma}
  \label{lem:tensor-collig-sp-by-H}
  Let $W_\Sigma: X\to \mathcal{L(H)}$ be a transfer function obtained
  via a unitary colligation $\Sigma = (U,\mathcal E, \rho)$.  Then
  there is another unitary colligation $\tilde\Sigma = (\tilde U,
  \tilde{\mathcal E} = \mathcal E \otimes \mathcal H, \tilde\rho)$
  such that $W_{\tilde\Sigma} = W_\Sigma$.
\end{lemma}

\begin{proof}
  Recall that by construction, there are orthogonal projections
  $P_\lambda$ with orthogonal ranges such that $\mathcal E =
  \bigoplus_\lambda \ran P_\lambda \otimes \mathbb C^{n_\lambda}$,
  $n_\lambda = 2^{|\lambda|-1}$, and $S(x) = \sum_\lambda
  P_\lambda \otimes \sigma_\lambda(x)$.  We construct the new
  colligation from the old by taking
  \begin{equation*}
    \tilde{\mathcal E} = \bigoplus_\lambda \ran P_\lambda \otimes
    (\mathbb C^{n_\lambda} \otimes \mathcal H),
  \end{equation*}
  and setting
  \begin{equation*}
    \tilde S(x) = \sum_\lambda P_\lambda \otimes (\sigma_\lambda(x)
    \otimes 1_{\mathcal{L(H)}}).
  \end{equation*}
  Fix $e\in \mathcal H$ with $\|e\|=1$.  Define an operator $\tilde U
  = \begin{pmatrix} \tilde A & \tilde B \\ \tilde C & \tilde D
  \end{pmatrix}$ on $\tilde{\mathcal E} \oplus \mathcal H$ as follows.
  For $f\in \mathcal E$, $h,g \in \mathcal H$ decomposed as $g =
  \alpha e + e^\bot$ where $\ip{e}{e^\bot} = 0$.  Then set
  \begin{equation*}
    \begin{split}
      \tilde A (f\otimes \alpha e + f \otimes e^\bot) &:= Af \otimes
      \alpha e + f \otimes e^\bot, \\
      \tilde B h &:= Bh \otimes e, \\
      \tilde C (f\otimes \alpha e + f \otimes e^\bot) &:= \alpha Cf \\
      \tilde D h &:= Dh,
    \end{split}
  \end{equation*}
  extending by linearity where necessary.  One easily checks that
  the adjoints of these operators are given by
  \begin{equation*}
    \begin{split}
      {\tilde A}^* (f\otimes \alpha e + f \otimes e^\bot) &= Af
      \otimes \alpha e + f \otimes e^\bot, \\
      {\tilde B}^* (f\otimes \alpha e + f \otimes e^\bot) &=  \alpha
      B^*f, \\
      {\tilde C}^* h &= C^*h \otimes e \\
      {\tilde D}^* h &= D^*h,
    \end{split}
  \end{equation*}
  again extending by linearity as needed.
  A straightforward calculation gives
  \begin{equation*}
    \begin{split}
      ({\tilde A}^*\tilde A + {\tilde C}^*\tilde C) (f\otimes \alpha e
      + f \otimes e^\bot) &= (A^*A+C^*C)f \otimes \alpha e + f \otimes
      e^\bot = f\otimes \alpha e + f \otimes e^\bot, \\
      ({\tilde B}^*\tilde B + {\tilde D}^*\tilde D)h &= (B^*B+D^*D) h,
    \end{split}
  \end{equation*}
  showing that the operators so defined are bounded.  The other
  equations needed to show that $\tilde U$ is unitary are likewise
  checked.

  We find that $\tilde C \tilde S(x) \tilde B h = \tilde C \tilde S(x)
  (Bh \otimes e) = \tilde C ((S(x)Bh) \otimes e) = CS(x)Bh$.  Also,
  $\tilde C (\tilde A\tilde S(x))^n \tilde B h = C(AS(x))^nBh$.  We
  conclude that $W_{\tilde\Sigma} = W_\Sigma$.
\end{proof}

\subsection{Contractivity and complete contractivity of
  representations of transfer function algebras}
\label{subsec:contr-compl-contr}

\begin{definition}
  \label{def:contr-on-aux-test-fns}
  We write that a representation $\pi:\mathcal{T}^A(X,\Lambda,
  \mathcal H) \to \mathcal{L(G)}$ or $\pi:\mathcal{T}(X,\Lambda,
  \mathcal H) \to \mathcal{L(G)}$ is \textbf{contractive on auxiliary
    test functions} if for each $\lambda \in \Lambda$, an appropriate
  ampliation of $\pi$ (also denoted by $\pi$) has the property that
  $\pi(\sigma_\lambda \otimes 1_{\mathcal{L(H)}}) \leq 1$.  It is said
  to be strictly contractive in case this is a strict inequality.  A
  representation is \textbf{strongly / weakly continuous} if whenever
  a bounded net $(\varphi_\alpha)$ converges pointwise in norm to
  $\varphi$ (in other words, $\sup_\alpha \|\varphi_\alpha\| < \infty$
  and for each $x\in X$, $\|\varphi_\alpha(x) - \varphi(x)\| \to 0$),
  then $\pi(\varphi_\alpha)$ converges strongly / weakly to
  $\pi(\varphi)$.
\end{definition}

Given a bounded unital representation $\pi$ of $\HLH$, we define
$\pi(\psi^\pm_\lambda)$ by applying $\pi$ entrywise.  Then $\pi$ is a
Brehmer representation if and only if $\pi$ is contractive on the test
functions and for any maximal element $\lambda$ of the preordering
$\Lambda$,
\begin{equation*}
  \pi(\psi_\lambda^+)\pi(\psi_\lambda^+)^* -
  \pi(\psi_\lambda^-)\pi(\psi_\lambda^-)^*
  \geq 0.
\end{equation*}
In this case, for each $\lambda$ in the maximal preordering associated
to $\Lambda$, there is a contraction $G: \ran\pi(\psi_\lambda^+)^* \to
\pi(\psi_\lambda^-)^*$ such that $\pi(\psi_\lambda^+)G_\lambda =
\pi(\psi_\lambda^-)$.  The following is then well defined:
\begin{equation*}
  \pi(\sigma_\lambda) = G_\lambda,
\end{equation*}
though properly speaking, this should be viewed as an ampliation of
the representation $\pi$.  As we saw in
Theorem~\ref{thm:ext-aux-test-fns-ample-case}, when $\Lambda$ is an
ample preordering, we can extend $\sigma_\lambda$ to a function in
$H^\infty(X,{\mathcal K}_{\Lambda,\mathbb C^n})$ where $n =
2^{|\lambda|-1}$, and so $\pi$ (or rather $\pi^{(n)}$) is already
defined on $\sigma_\lambda$, and potentially may not be equal to
$G_\lambda$.  Nevertheless, it is the case that once $\pi$ is given on
test functions, it induces a well defined map which is contractive on
auxiliary test functions, and so on the algebra of transfer functions,
as we shall see.

The next theorem is a version of the von~Neumann inequality for the
algebra $\mathcal{T}(X,\Lambda, \mathcal H)$.

\begin{theorem}
  \label{thm:vN-for-transfer-fns}
  Let $\pi:\mathcal{T}^A(X,\Lambda, \mathcal H) \to \mathcal{L(G)}$ be a
  unital representation which is contractive on auxiliary test
  functions, or $\pi:\mathcal{T}(X,\Lambda, \mathcal H) \to
  \mathcal{L(G)}$ be a weakly continuous unital representation which
  is contractive on auxiliary test functions.  For all $W_\Sigma$ in
  $\mathcal{T}^A_1(X,\Lambda, \mathcal H)$ $($respectively,
  $\mathcal{T}_1(X,\Lambda, \mathcal H)\,)$, $\|\pi(W_\Sigma)\| \leq 1$;
  that is $\pi$ is contractive.
\end{theorem}

\begin{proof}
  We begin by observing that in either case, the representation
  $\pi_0: \mathcal{L(H)} \to \mathcal{L(G)}$ obtained by restricting
  $\pi$ to constant functions is a unital representation of the
  $C^*$-algebra $\mathcal{L(H)}$, and so is contractive.  The same is
  true of the ampliations of $\pi_0$, so it is in fact completely
  contractive.

  Let $W = W_\Sigma \in \mathcal{T}_1(X,\Lambda, \mathcal H)$, where
  $\Sigma = (U,\mathcal E, \rho)$, $U = \begin{pmatrix} A&B\\ C&D
  \end{pmatrix}$, is a unitary colligation.  For $r\in (0,1)$, define
  $W_r= W_{\Sigma_r}$, where $\Sigma_r = (U_r,\mathcal H,\rho)$ is a
  contractive colligation with
  \begin{equation*}
    U_r = U \begin{pmatrix} r1&0 \\ 0&1 \end{pmatrix} =
    \begin{pmatrix} rA&B\\ rC&D \end{pmatrix}.
  \end{equation*}
  By Lemma~\ref{lem:contractive-to-unitary-transfer-fn}, $W_r \in
  \mathcal{T}_1(X,\Lambda, \mathcal H)$, and
  \begin{equation*}
    W_r = D +C(rS)(1-A(rS))^{-1}B
  \end{equation*}

  We now follow the line of proof in Lemma~3.1 of~\cite{MR2389623}.
  Since $rAS(x)$ is a strict contraction, $\tfrac{M}{M+1}\sum_1^M
  \tfrac{M-n}{M}(rAS(x))^n$ converges uniformly in norm to
  $(1-rAS(x))^{-1}$, and by the proof of the last statement in
  Theorem~\ref{thm:trfr-fns-span-alg}, for all $M$, $W_{r,M}:= D +
  CS\left(\tfrac{M}{M+1}1+\tfrac{M-1}{M+1}AS+\cdots +
    \tfrac{1}{M+1}(AS)^M\right)B \in \mathcal{T}_1(X,\Lambda, \mathcal
  H)$ and converges pointwise in norm with $M$ to $W_r$.  By
  assumption then, $\pi(W_{r,M})$ converges weakly to $\pi(W_r)$.

  As in the proof of Lemma~3.1 of~\cite{MR2389623}, we see that
  $\pi(W_r) = W_{{\tilde\Sigma}_r} \in \mathcal{T}_1(X,\Lambda, \mathcal
  G)$ with ${\tilde\Sigma}_r = ({\tilde U}_r,\mathcal E \otimes
  \mathcal G, \rho\otimes \pi)$, where ${\tilde U}_r = \begin{pmatrix}
    r\tilde A\otimes 1&\tilde B\otimes 1\\ r\tilde C\otimes 1&\tilde
    D\otimes 1 \end{pmatrix}$.  We have $\tilde D = \pi(D)$ and
  $\tilde A$, $\tilde B$ and $\tilde C$ are obtained by applying $\pi$
  component-wise.  Since $\pi_0 = \pi|_{\mathcal{L(H)}}$ is completely
  contractive, ${\tilde U}_r$ is a contraction.  Hence
  $\|\pi(\varphi_r)\| \leq 1$.  Now $(\varphi_r)_r$ is a bounded net
  converging pointwise in norm to $\varphi$, so by assumption
  $(\pi(\varphi_r))_r$ converges weakly to $\pi(\varphi)$, meaning
  that $\|\pi(\varphi)\| \leq 1$.

  For $\mathcal{T}^A(X,\Lambda, \mathcal H)$, the same argument applies
  when $\pi$ is simply assumed to be contractive on auxiliary test
  functions, since $\mathcal{T}^A(X,\Lambda, \mathcal H)$ is the norm
  closure of polynomials in test functions.
\end{proof}

\begin{corollary}
  \label{cor:weakly-ctns_B_reps_are_cc}
  Let $\pi$ be a representation of $\mathcal{T}^A(X,\Lambda, \mathcal
  H)$, respectively, a weakly continuous representation of
  $\mathcal{T}(X,\Lambda, \mathcal H)$, which is contractive on
  auxiliary test functions.  Then $\pi$ is completely contractive.
\end{corollary}

\begin{proof}
  If $\pi$ is a representation of either $\mathcal{T}^A(X,\Lambda,
  \mathcal H)$ or $\mathcal{T}(X,\Lambda, \mathcal H)$ which is
  contractive on auxiliary test functions, then the same is true for
  $\pi^{(n)}$ for all $n$.  Hence the result follows from the previous
  theorem applied to the auxiliary test functions tensored with $1_n$.
\end{proof}

\subsection{Brehmer representations and spectral sets}
\label{subsec:brehm-repr}

\begin{definition}
  Let $\pi$ be a bounded unital representation of $\HLH$.  Call $\pi$
  a \textbf{Brehmer representation} (associated to the preordering
  $\Lambda$) if for any test function $\psi$, $\|\pi(\psi\otimes
  1_{\mathcal{L(H)}})\| \leq 1$ and for all $\lambda\in \Lambda$,
  \begin{equation}
    \label{eq:1}
    \prod_{\lambda\ni\lambda_i\neq 0} \left(1 -
    \pi(\psi_i\otimes 1_{\mathcal{L(H)}})\pi(\psi_i \otimes
    1_{\mathcal{L(H)}})^*\right)^{\lambda_i} \geq 0.
  \end{equation}

  Note since $\mathcal{L(H)}$ is a $C^*$-algebra, it is automatic that
  $\pi_0 = \pi|_{\mathcal{L(H)}}$ with $\pi_0(T) = \pi(1\otimes T)$ is
  completely contractive.

  A representation $\pi$ of $\HLH$ is a \textbf{strict Brehmer
    representation} if the inequalities in \eqref{eq:1} are strict.
  It is a \textbf{strongly / weakly continuous Brehmer representation}
  if it is a Brehmer representation and which is either strongly or
  weakly continuous in the sense defined in the last subsection.

  We say that $X$ is a \textbf{spectral set} for the representation
  $\pi$ (equivalently, that the \textbf{von~Neumann inequality} holds)
  if $\pi$ is a contractive representation of $\ALH$.  It is a
  \textbf{complete spectral set} if $\pi$ is a completely contractive
  representation of $\ALH$.

  A representation $\tilde\pi$ \textbf{dilates} a representation $\pi$
  (equivalently, $\pi$ \textbf{dilates to} $\tilde\pi$) if $\pi$ is
  the restriction of $\tilde\pi$ to a semi-invariant subspace; that
  is, the difference of two invariant subspaces.  The
  \textbf{$H^\infty$ dilation property} is said to hold for a domain
  $X$ if whenever $\pi$ is a representation of $\HLH$ for which $X$ is
  a spectral set, then $X$ is a complete spectral set for $\pi$.
\end{definition}

While Brehmer representations induce representations which are
contractive on test functions, the converse is also true.

\begin{lemma}
  \label{lem:contr-aux-t-fns-are-B-reps}
  If a representation $\pi$ of $\HLH$ is contractive on auxiliary test
  functions then it is a Brehmer representation.
\end{lemma}

\begin{proof}
  This follows from \eqref{eq:5}.
\end{proof}

Clearly, a strict Brehmer representation is norm continuous, a norm
continuous one is strongly continuous, and a strongly continuous one
is weakly continuous.  The $H^\infty$ dilation property is akin to the
better known \textbf{rational dilation property}, where $\HLH$ is
replaced by the algebra of functions generated by the rational
functions over a compact subset of $\mathbb C^d$ with poles off of the
set.

The connection of the von~Neumann inequality as defined above with the
usual von~Neumann inequality is as follows.  Suppose that $X = \mathbb
D^d$ and $\Psi$ is the set of coordinate functions in $\mathbb C^d$
(so $\psi_j(z) = z_j$ for $j=1,\dots, d$), and assume that $\Lambda =
\{e_j\}_{j=1,\dots ,d}$.  Then Agler's realization theorem for the
polydisk (Theorem~\ref{thm:Aglers-realization} above) implies that any
representation $\pi$ of $\HL$ for which $T_j = \pi(\psi_j)$ is
strictly contractive for all $j$ (so $(T_1,\dots,T_d)$ is a tuple of
commuting strict contractions) is contractive on $\HL$.  Note that in
this case $S(z) = Z^-(z) = \sum_j P_j z_j$, where $P_j$s are
orthogonal projections summing to the identity.  We therefore
naturally take $\pi(S(z)) = \sum_j P_j \otimes T_j$, which then, via
the transfer function representation, allows us to interpret
$\pi(\varphi)$ for $\varphi\in\HLb$ in the natural way.  So in other
words, for a tuple $T$ of commuting operators with $\|T_j\| < 1$ for
all $j$, $\|\varphi(T)\| \leq 1$ for all $\varphi$ in the Schur-Agler
class of the polydisk.

The name for the rational dilation property derives from a theorem of
Arveson~\cite{MR1668582}, which states in the example from the
previous paragraph, a tuple $T$ of commuting contractions has a
commuting unitary dilation $U$ if and only if for all $n\in\mathbb N$,
$T$ induces a completely contractive representation $\pi$ on the
algebra $\mathcal P$ of polynomials over $\mathbb C^d$, the norm
closure of which is the polydisk analogue of the disk algebra.  Write
$\tilde\pi$ for the representation induced by $U$. By the spectral
theorem for normal operators, $\tilde\pi$ is completely contractive.
The converse direction is an application of the Arveson extension
theorem and Stinespring dilation theorem.  Of course there would be no
hope of dilating $T$ to $U$ if it were the case that the
representation induced by $T$ is not contractive, which the example
due to Kaijser and Varopoulos~\cite{MR0355642} demonstrates can happen
when $d\geq 3$ in $H^\infty(\mathbb D^d)$.

Because $\mathbb D^d$ is polynomially convex, the polynomial algebra
suffices when considering rational dilation in this setting.  For more
complex domains $X\subset \mathbb C^d$ such as for example an annulus
in $\mathbb C$, one needs to consider $M_n(\mathbb C)$ valued rational
functions over $\mathbb C^d$ with poles off of $\overline{X}$, and the
commuting tuple of unitary operators is replaced by a commuting tuple
of normal operators with spectrum supported on $\partial X$ (or more
precisely, the distinguished boundary of $X$).

It becomes evident then that one can view Arveson's theorem as
describing when a contractive representation of the analogue of the
disk algebra is completely contractive.  An example due to
Parrott~\cite{MR0268710} shows that when $d\geq 3$, there are
contractive representations which are not completely contractive.
Further examples when $d=3$ are given by Bagchi, Bhattacharyya and
Misra in~\cite{MR1873633}, and they show that these examples are not
even $2$-contractive.  As we shall see, this is no accident --- in
fact any representation which is contractive but not completely
contractive must fail to be $2$-contractive.

When $d=1$ or $2$, contractive representations are automatically
completely contractive by the Sz.-Nagy dilation theorem and And\^o's
theorem, respectively.  Agler showed that over an annulus $\mathbb A$,
it is again the case that contractive representations of the algebra
of functions analytic in a neighborhood of $\mathbb A$ are completely
contractive.  This was later shown to fail for domains of higher
connectivity~\cite{MR2375060,MR2163865,MR2643788}.

It is a consequence of the Arveson extension theorem and the
Stinespring dilation theorem that any completely contractive
representation of either $\ALH$ or $\HLH$ extends to a completely
contractive representation of $C^*(\HLH)$ or $C^*(\ALH)$,
respectively.

We have the following dilation theorem, generalizing Arveson's
dilation result for the polydisk.

\begin{theorem}
  \label{thm:dilation-theorem}
  Let $\pi$ be a representation of $\mathcal{T}^A(X,\Lambda, \mathcal
  H)$, or a weakly continuous representation of $\mathcal{T}(X,\Lambda,
  \mathcal H)$, which is contractive on auxiliary test functions.
  Then $\pi$ dilates to a completely contractive representation
  $\tilde\pi$ of $C^*(\mathcal{T}(X,\Lambda, \mathcal H))$
  $($respectively, $C^*(\mathcal{T}^A(X,\Lambda, \mathcal H))\,)$, with
  the property that the only completely positive map agreeing with
  $\tilde\pi$ on $\mathcal{T}(X,\Lambda, \mathcal H)$ $($respectively,
  $\mathcal{T}^A(X,\Lambda, \mathcal H)\,)$ is $\tilde\pi$ itself.
\end{theorem}

\begin{proof}
  This is a corollary of Corollary~\ref{cor:weakly-ctns_B_reps_are_cc}
  and Theorem~1.1 of~\cite{MR2132691}.
\end{proof}

A representation with the properties of $\tilde\pi$ (ie, that
$\tilde\pi$ extends uniquely as a completely positive map to the
$C^*$-envelope) is called a \textbf{boundary representation} if, in
addition, it is irreducible.  We use an alternative, equivalent
description of boundary representations due Muhly and
Solel~\cite{MR1639657} below.

An analogue of the rational dilation problem ask whether every
contractive representation of $\ALH$ is completely contractive.
Likewise, one might ask if every contractive representation of $\HL$
(or more generally, of $\HLH$) is automatically completely
contractive; that is, whether the $H^\infty$ dilation property holds.
Perhaps surprisingly, even for $H^\infty(\mathbb D)$ this is unknown.
The problem is that in many cases the boundary of $X$ is rather
complicated, since it is the difference between the Stone-\v{C}ech
compactification of $X$ and $X$ in the appropriate topology, and this
can be very complex.  There will be representations corresponding to
point evaluations in the boundary.  In general, these may not be
weak-$*$ continuous, and so there is no obvious characterization of
contractive representations of $\HL$ in terms of its action on test
functions, which is generally what is used in the showing the
contractivity of ampliations of a representation.

As an alternative, one might ask if there are any simply described
subclasses of the contractive representations which are completely
contractive.  For example, we will prove that representations of
$\HLH$ which are Brehmer representations and which are weakly
continuous are completely contractive.  We should note that for
general $\Lambda$, it is easy to find examples where not all
contractive representations are Brehmer representations.

Over $\mathbb D^d$ when $d \geq 3$, Parrott's example implies that
rational dilation fails for $A(\mathbb D^3)$, though as we saw in
Corollary~\ref{cor:tr_fns_are_op_alg}, with the Agler algebra and
Schur-Agler matrix norm structure, this is not the case.  We prove
that in general any representation of $\ALH$ which is contractive on
the auxiliary test functions is completely contractive.  When the
preordering is ample over $d$ test functions, this will imply that any
representation which is $2^{d-1}$-contractive is completely
contractive.  As we will show, there is an improvement which can be
made to this when $d>1$ using the so-called nearly ample preorderings,
and giving that $2^{d-2}$-contractive representations are completely
contractive.  In particular, this will imply that for $d\geq 3$,
$2^{d-2}$-contractive representations of $A(\mathbb D^d)$ are
completely contractive, and that such representations of
$H^\infty(\mathbb D^d)$ which are at least weakly continuous are also
completely contractive.  When $d=3$ then, $2$ contractivity will imply
complete contractivity, and so any example like Parrott's of a
contractive representation of $A(\mathbb D^3)$ which is contractive
but not completely contractive must fail to be $2$-contractive.

\subsection{Some boundary representations for the classical Agler
  algebra}
\label{subsec:some-bound-repr-Agler-alg}

Since in the classical setting the auxiliary test functions are simply
the test functions, it follows from Theorem~\ref{thm:classic-real-thm}
that any representation of $\HLH$ which is contractive is completely
contractive.  At first this may seem to contradict the examples of
Parrott~\cite{MR0268710} and Varopoulos and Kaiser~\cite{MR0355642}
when $X=\mathbb D^3$, which both give commuting tuples of contractions
on $H^\infty(\mathbb D^3)$ which do not dilate to commuting unitary
operators (indeed, the Kaijser-Varopoulos example is not even a
contractive representation of $H^\infty(\mathbb D^3)$).  The reason
that there is no difficulty is that the Schur-Agler norm of $\HLH$
(and more generally, the corresponding matrix norm structure) is not
the same as the supremum norm in this case.

Let us consider more closely the classical Agler algebra over the
tridisk.  We examine the representations generated by commuting
triples of contractions from several particularly interesting
examples: first that of Parrott, then a Kaijser-Varopoulos type
example due to Grinshpan, Kaliuzhnyi-Verbovet\-skyi and Woerdeman
from~\cite{MR3057417}, and finally the Kaijser-Varopoulos example
itself.  We show that these give rise to nontrivial non-scalar
boundary representations for the disk algebra analogue for the
classical Agler algebra.  Of course such representations are expected
since, as has been noted~\cite{MR1049839}, this is not a uniform
algebra, but these are explicit.  According to a result of Muhly and
Solel~\cite{MR1639657}, a \textbf{boundary representation} in the
sense of Arveson is an irreducible completely contractive unital
representation of $\HL$ with the property that any completely
contractive dilation of this representation must contain it as a
direct summand (see also~\cite{MR2132691}).

We begin by considering the Parrott example.

\begin{lemma}
  \label{lem:Parrott-boundary-repn}
  Let $X = \mathbb D^3$, $\Psi = \{z_1,z_2,z_3\}$ a collection of test
  functions on $X$, $\Lambda = \{e_1,e_2,e_3\}$, and $\mathcal K$ the
  corresponding set of admissible kernels.  Let $U,V\in
  \mathcal{L(K)}$ be unitary operators with the property that $UV =
  -VU$ $($for example, we might choose $U = \begin{pmatrix} 1 & 0 \\ 0
    & -1 \end{pmatrix}$ and $V = \begin{pmatrix} 0 & 1 \\ 1 & 0
  \end{pmatrix}\,)$.  Then on $\mathcal K \oplus \mathcal K$,
  \begin{equation*}
    \pi(z_1) := T_1 =
    \begin{pmatrix}
      0 & 1 \\ 0 & 0
    \end{pmatrix}, \qquad
    \pi(z_2) := T_2 =
    \begin{pmatrix}
      0 & U \\ 0 & 0
    \end{pmatrix}, \qquad
    \pi(z_3) := T_3 =
    \begin{pmatrix}
      0 & V \\ 0 & 0
    \end{pmatrix}, \qquad
  \end{equation*}
  defines a $($completely contractive$)$ boundary representation of
  $\HL$.
\end{lemma}

\begin{proof}
  It is obvious that the operators in the statement of the lemma
  commute.  By Theorem~\ref{thm:classic-real-thm}, this gives a
  contractive representation of $\HL$, and so by
  Corollary~\ref{cor:tr_fns_are_op_alg} a completely contractive
  representation.  It is clearly irreducible.  As noted in the
  discussion preceding the statement of the lemma, it suffices to
  prove that any contractive dilation of this representation contains
  it as a direct summand.

  Assume that
  \begin{equation*}
    \tilde\pi(z_1) =
    \begin{pmatrix}
      A_1 & A_2 & A_3 \\ 0 & T_1 & A_4 \\ 0 & 0 & A_5
    \end{pmatrix}, \qquad
    \tilde\pi(z_2) =
    \begin{pmatrix}
      B_1 & B_2 & B_3 \\ 0 & T_2 & B_4 \\ 0 & 0 & B_5
    \end{pmatrix}, \qquad
    \tilde\pi(z_3) =
    \begin{pmatrix}
      C_1 & C_2 & C_3 \\ 0 & T_3 & C_4 \\ 0 & 0 & C_5
    \end{pmatrix}
  \end{equation*}
  are commuting contractions.  We show that $A_2$, $B_2$ $C_2$, $A_4$,
  $B_4$ and $C_4$ are zero.  Since $1$, $U$ and $V$ are unitary, it
  follows that
  \begin{equation*}
    A_2 =
    \begin{pmatrix}
      a & 0
    \end{pmatrix}, \qquad
    B_2 =
    \begin{pmatrix}
      b & 0
    \end{pmatrix}, \qquad
    C_2 =
    \begin{pmatrix}
      c & 0
    \end{pmatrix}
  \end{equation*}
  on $\mathcal K \oplus \mathcal K$.  Commutativity then gives
  \begin{equation*}
    \begin{split}
      A_1B_2 + A_2T_2 &= B_1A_2 + B_2T_1 \\
      A_1C_2 + A_2T_3 &= C_1A_2 + C_2T_1 \\
      B_1C_2 + B_2T_3 &= C_1B_2 + C_2T_2.
    \end{split}
  \end{equation*}
  Right multiplication of the first of these by $T_1$ yields $A_1b =
  B_1a$, and so $A_1B_2 = B_1A_2$.  Hence $A_2T_2 = B_2T_1$, and so
  $aU = b$.  Similar calculations with the other two equations give
  $aV = c$ and $bV = cU$.  Thus $aUV = bV = cU = aVU = -aUV$, and
  since $UV$ is unitary, $a=0$.  We then also have $b=c=0$.  A similar
  calculation shows that $A_4$, $B_4$ and $C_4$ are zero.
\end{proof}

We next turn to the example of Grinshpan, Kaliuzhnyi-Verbovet\-skyi
and Woerdeman from~\cite{MR3057417}, which again as in the Parrott
example is nilpotent, but this time of order~$2$.

\begin{theorem}
  \label{thm:GKVW-boundary-repn}
  Let $X = \mathbb D^3$, $\Psi = \{z_1,z_2,z_3\}$ a collection of test
  functions on $X$, $\Lambda = \{e_1,e_2,e_3\}$, and $\mathcal K$ the
  corresponding set of admissible kernels.  Let $u_1, u_2, u_3 \in
  \mathbb R^2$ be unit vectors with the property that $u_1+u_2+u_3 =
  0$ (without loss of generality, we may assume $u_1 =
  \begin{pmatrix} 0 & 1 \end{pmatrix}$, $u_2 = \begin{pmatrix}
    \sqrt{3}/2 & -1/2 \end{pmatrix}$, $u_3 = \begin{pmatrix}
    -\sqrt{3}/2 & -1/2 \end{pmatrix}$).  Define a representation $\pi
  : \HL \to M_4(\mathbb C)$ by
  \begin{equation*}
    \pi(z_j) := T_j =
    \begin{pmatrix}
      0 & u_j & 0 \\ 0 & 0 & u_j^* \\ 0 & 0 & 0
    \end{pmatrix}, \qquad j = 1,2,3.
  \end{equation*}
  Then this is a $($completely contractive$)$ boundary representation
  of $\HL$.
\end{theorem}

\begin{proof}
  We assume that we have made the explicit choice of $u_j$s mentioned
  in the statement of the theorem.  Consider a commuting contractive
  dilation
  \begin{equation*}
    V_j = 
    \begin{pmatrix}
      a_j & b_j & v_j & c_j & d_j \\
      0 & 0 & u_j & 0 & e_j \\ 0 & 0 & 0 & u_j^* & {v'}^*_j \\ 
      0 & 0 & 0 & 0 & f_j \\ 0 & 0 & 0 & 0 & g_j \\
    \end{pmatrix}, \qquad j = 1,2,3,
  \end{equation*}
  of the $T_j$s.  Because each $u_j$ is a unit vector, $c_j = 0$ and
  $e_j = 0$ for each $j$.  We also have that $u_jv_j^* = u_j {v'}_j^*
  = 0$, so
  \begin{equation*}
    \begin{split}
      v_1 = \alpha_1 \begin{pmatrix} 1 & 0 \end{pmatrix}
      \qquad 
      v_2 &= \alpha_2 \begin{pmatrix}  -1/2 & -\sqrt{3}/2 \end{pmatrix}
      \qquad
      v_3 = \alpha_3 \begin{pmatrix}  -1/2 & \sqrt{3}/2 \end{pmatrix}
      \\
      {v'}_1 = {\alpha'}_1 \begin{pmatrix} 1 & 0 \end{pmatrix}
      \qquad 
      {v'}_2 &= {\alpha'}_2 \begin{pmatrix}  -1/2 & -\sqrt{3}/2
      \end{pmatrix}
      \qquad
      {v'}_3 = {\alpha'}_3 \begin{pmatrix}  -1/2 & \sqrt{3}/2
      \end{pmatrix}.
    \end{split}
  \end{equation*}

  By commutativity,
  \begin{equation*}
    u_j {v'}^*_k = u_k {v'}_j^* \qquad\text{and} \qquad
    v_j u_k^* = v_k u_j^*.
  \end{equation*}
  Using the explicit form of these vectors, it is easy to check that
  the first of these equations gives ${\alpha'}_2 = {\alpha'}_3 =
  -{\alpha'}_1$ and ${\alpha'}_2 = -{\alpha'}_3$, and so ${\alpha'}_j
  = 0$ for all $j$.  Similar calculations with the second equation
  yields $\alpha_j = 0$ for all $j$ as well.  Thus $v_j = {v'}_j = 0$
  for all $j$.

  It also follows from commutativity that $b_j u_k = b_k u_j$, and
  since the $u_k$s are pairwise linearly independent, it follows that
  $b_j = 0$ for all $j$.  Likewise, $f_j = 0$ for all $j$, and so we
  conclude that each $V_j$ contains $T_j$ as a direct summand.

  Finally, we show that the representation is irreducible.  If
  $\mathcal G \subset \mathbb R^4$ is a reducing subspace, then it is
  invariant for $T_j^*T_j$ and $T_jT_j^*$ for each $j$.  From this we
  see that $\mathcal G \neq \mathbb C^4$, any vector in $\mathcal G$
  must be of the form $v_1 = {\begin{pmatrix} c_1 & 0 & 0 & c_2
    \end{pmatrix}}^t$, $v_2 = {\begin{pmatrix} c_1 & c_2 & c_3 & 0
    \end{pmatrix}}^t$, $v_3 = {\begin{pmatrix} 0 & c_1 & c_2 & c_3
    \end{pmatrix}}^t$, where $c_j\in \mathbb C$ for all $j$.
  Multiplying $v_1$ by $T_j$ we get $c_2 = 0$, and by $T_j^*$ we get
  $c_1 = 0$; that is, $\mathcal G = \{0\}$.  Similarly, since the
  $u_j$s span $\mathbb R^2$, we conclude after considering $T_j^*v_2$
  and $T_jv_3$ that $c_2=c_3 =0$ in the first case and $c_1=c_2=0$ in
  the second, and from this that $c_1=0$ in $v_2$ and $c_3 = 0$ in
  $v_3$, finishing the proof.
\end{proof}

Finally, we turn to the Kaijser-Varopoulos example.  As it happens,
the operators there can be dilated to other commuting contractions
which can only be further dilated by means of a direct sum.  The proof
is similar to the above, and we leave it as an exercise for the
interested reader.

\begin{theorem}
  \label{thm:KV-boundary-repn}
  Let $X = \mathbb D^3$, $\Psi = \{z_1,z_2,z_3\}$ a collection of test
  functions on $X$, $\Lambda = \{e_1,e_2,e_3\}$, and $\mathcal K$ the
  corresponding set of admissible kernels.  Then the representation
  $\pi : \HL \to M_6(\mathbb C)$ defined by
  \begin{equation*}
    \begin{split}
      \pi(z_1) &:= T_1 =
      \begin{pmatrix}
        0 &0&0&0&0&0 \\ 1&0&0&0&0&0 \\ 0&0&0&0&0&0 \\ 0&0&0&0&0&0
        \\[2pt] 0 &
        \tfrac{1}{\sqrt{3}} & -\tfrac{1}{\sqrt{3}} &
        -\tfrac{1}{\sqrt{3}} & 0 & 0 \\[5pt] 0 & \tfrac{2}{\sqrt{6}}
        & \tfrac{1}{\sqrt{6}} & \tfrac{1}{\sqrt{6}} &0 &0
      \end{pmatrix},\\
      \pi(z_2) &:= T_2 =
      \begin{pmatrix}
        0 &0&0&0&0&0 \\ 0&0&0&0&0&0 \\ 1&0&0&0&0&0 \\ 0&0&0&0&0&0
        \\[2pt] 0 &
        -\tfrac{1}{\sqrt{3}} & \tfrac{1}{\sqrt{3}} &
        -\tfrac{1}{\sqrt{3}} & 0 & 0 \\[5pt] 0 & \tfrac{1}{\sqrt{6}}
        & \tfrac{2}{\sqrt{6}} & \tfrac{1}{\sqrt{6}} &0 &0
      \end{pmatrix},\\
      \pi(z_3) &:= T_3 =
      \begin{pmatrix}
        0 &0&0&0&0&0 \\ 0&0&0&0&0&0 \\ 0&0&0&0&0&0 \\ 1&0&0&0&0&0
        \\[2pt] 0 &
        -\tfrac{1}{\sqrt{3}} & -\tfrac{1}{\sqrt{3}} &
        \tfrac{1}{\sqrt{3}} & 0 & 0 \\[5pt] 0 & \tfrac{1}{\sqrt{6}}
        & \tfrac{1}{\sqrt{6}} & \tfrac{2}{\sqrt{6}} &0 &0
      \end{pmatrix}.\\
    \end{split}
  \end{equation*}
  is a $($completely contractive$)$ boundary representation of $\HL$.
\end{theorem}

We finally mention that the Crabb-Davie example~\cite{MR0365179} gives
a boundary representation of the Agler algebra consisting of nilpotent
operators of order~3, and presumably examples with any degree of
nilpotency can be constructed in a similar manner.  Using results of
\cite{MR1169882} and \cite{MR3160536}, one can show that there will be
boundary representations of the Agler algebra which are neither
commuting unitaries nor commuting nilpotents.

\section{Realization theorems}
\label{sec:realization}

\subsection{The first realization theorem}
\label{subsec:first-realization}

As usual, we assume all test functions are in $\AL$.

Fix a finite set $F\subset X$.  Define a cone in $M_{|F|}(\mathbb C)$
by
\begin{equation*}
  \mathcal C_F := \left\{\left(\Gamma(x,y) \left( E^+(x)E^+(y)^* -
        E^-(x)E^-(y)^*\right)\right) : \Gamma \in \mathbb
    K^+_X(C_b(\Lambda), \mathbb C) \right\}.
\end{equation*}
This is a cone rather than simply a wedge since $E^+(x)E^+(x)^* -
E^-(x)E^-(x)^* >0$, and so if $\Gamma_1, \Gamma_2 \geq 0$ with
$\left(\Gamma_1(x,y) \left( E^+(x)E^+(y)^* -
    E^-(x)E^-(y)^*\right)\right) = -\left(\Gamma_2(x,y) \left(
    E^+(x)E^+(y)^* - E^-(x)E^-(y)^*\right)\right)$ for all $x,y \in
F$, then for all $x$, $\Gamma_1(x,x) = \Gamma_2(x,x) = 0$, and hence
by positivity, $\Gamma_1(x,y) = \Gamma_2(x,y) = 0$ for all $x,y\in F$.

More generally, there is an operator version of this.  For a fixed
Hilbert space $\mathcal H$, define a cone in $M_{|F|}(\mathcal{L(H)})$
by
\begin{equation*}
  \mathcal C_{F,\mathcal H} := \left\{\left(\Gamma(x,y) \left(
        E^+(x)E^+(y)^* - E^-(x)E^-(y)^*\right)\right) : \Gamma 
    \in \mathbb K^+_X(C_b(\Lambda), \mathcal{L(H)}) \right\}.
\end{equation*}

The proof of the first realization theorem relies on the following
lemma of independent interest.

\begin{lemma}
  \label{lem:cone_closed}
  The cone $\mathcal C_{F,\mathcal H}$ is closed and has non-empty
  interior.  Furthermore, for each $\lambda \in\Lambda$,
  $1_{\mathcal{L(H)}} \otimes
  {\left(\textstyle\prod_{\lambda_i\in\lambda}
      (1 - \psi_i(x)\psi_i(y)^*)^{\lambda_i} \right)}_{x,y\in F} \in
  \mathcal C_{F,\mathcal H}$.
\end{lemma}

\begin{proof}
  Fix $F\subset X$ finite and a Hilbert space $\mathcal H$, and define
  the cones $\mathcal C_F$ and $\mathcal C_{F,\mathcal H}$ as above.
  Following the proof of Lemma~3.4 of~\cite{MR2389623}, we first show
  that $\mathcal C_F$ is closed.

  By assumption, for all $x\in X$, there exists $\epsilon_x > 0$ such
  that $\sup_{\psi\in\Psi} (1-\psi(x)\psi(x)^*) > \epsilon_x$.  Also,
  for $n := \sup_{\lambda\in\Lambda} |\lambda| < \infty$,
  \begin{equation*}
    E^+(x)E^+(x)^* - E^-(x)E^-(x)^* \geq \epsilon_x^n.
  \end{equation*}
  Setting $\epsilon = \min_{x\in F} \epsilon_x^n > 0$, we have then
  that for all $x\in F$, $E^+(x)E^+(x)^* - E^-(x)E^-(x)^* \geq
  \epsilon$.  Therefore, for any $M = \Gamma * (E^+E^{+*}-E^-E^{-*})
  \in \mathcal C_F$ and any $x\in F$,
  \begin{equation*}
    \|\Gamma(x,x)\| \leq \tfrac{1}{\epsilon} \max_{x\in F}\|M(x,x)\|
    \leq  \tfrac{1}{\epsilon} \|M\|.
  \end{equation*}
  Positivity of $\Gamma$ then gives $\|\Gamma(x,y)\| \leq
  \tfrac{1}{\epsilon} \|M\|$ for all $x,y\in F$.  Thus for any Cauchy
  sequence $(M_n) \subset \mathcal C_F$, the corresponding sequence of
  positive operators $(\Gamma_n)$ has $(\Gamma_n(x,y))$ in a norm
  closed ball of $C_b(\Lambda)^*$ and so has a weak-$*$ convergent
  subsequence.  Applying this idea to each pair of points in $F$, we
  eventually end up with a subsequence $\Gamma_{\ell_n}$ such that for
  any $x,y\in F$, $\Gamma_{\ell_n}(x,y)$ converges weak-$*$ to
  $\Gamma(x,y)$.  It is not difficult to see that $\Gamma$ is
  positive, and so $(M_n)$ converges to some $M = \Gamma *
  (E^+E^{+*}-E^-E^{-*}) \in \mathcal C_F$; that is, $\mathcal C_F$ is
  closed.

  Next consider $\mathcal C_{F,\mathcal H}$.  Arguing as above, there
  exists $\epsilon > 0$ such that for any $M = \Gamma *
  (E^+E^{+*}-E^-E^{-*}) \in \mathcal C_{F,\mathcal H}$,
  $\|\Gamma(x,y)\| \leq \tfrac{1}{\epsilon} \|M\|$ for all $x,y\in F$.
  Suppose $(M_n)\subset \mathcal C_{F,\mathcal H}$ with $\sup_n
  \|M_n\| = C < \infty$ converging to $M$.  Note that the corresponding
  sequence $(\Gamma_n)$ is bounded by $C/\epsilon$.  For $h = (h_x)
  \in \mathcal H^{|F|}$ with $\|h\| = 1$, define $M_{h,n}$ by
  $M_{h,n}(x,y) = \ip{M_n(x,y) h_x}{h_y}$ and $\Gamma_{h,n}$ by
  $\Gamma_{h,n}(x,y)(f) = \ip{\Gamma_n(x,y) (f) h_x}{h_y}$.  Then
  $(M_{h,n}) \subset \mathcal C_F$ is a Cauchy sequence, and since
  $\mathcal C_F$ is closed, $\lim_n M_{h,n} = M_h = \Gamma_h *
  (E^+E^{+*}-E^-E^{-*})$, where $\Gamma_h \geq 0$ and $\|\Gamma_h\|
  \leq C/\epsilon$.  Thus $\Gamma$ defined via polarization from
  $\ip{\Gamma (f) h}{h} = \Gamma_h(f)$ is positive and bounded in norm
  by $C/\epsilon$, and $M = \Gamma * (E^+E^{+*}-E^-E^{-*})$.  Hence
  the cone $\mathcal C_{F,\mathcal H}$ is also closed.

  We next show that $\mathcal C_{F,\mathcal H}$ (and as a consequence,
  $\mathcal C_F$) has non-empty interior.  Let $P:X\times X \to
  \mathcal{L(H)}$ be a positive kernel with Kolmogorov decomposition
  $P(x,y) = Q(x)Q(y)^*$.  A straightforward argument as in the proof
  of Lemma~3.5 of~\cite{MR2389623} shows that the kernel
  $\Gamma_{P,\lambda}$ mapping $X\times X$ to
  $\mathcal{L}(C_b(\Lambda),\mathcal{L(H)})$ by
  \begin{equation*}
    \Gamma_{P,\lambda}(x,y)(f) =
    {\left((Q(x)\otimes\psi^+_\lambda(x))
        (Q(y)\otimes\psi^+_\lambda(y))^* -
    (Q(x)\otimes\psi^-_\lambda(x))
    (Q(y)\otimes\psi^-_\lambda(y))^*\right)}^{-1} 
    f(\lambda)
  \end{equation*}
  is positive.  Thus
  \begin{equation*}
    \left(\Gamma_{P,\lambda}(x,y)\left( E^+(x)E^+(y)^* - 
        E^-(x)E^-(y)^*\right)\right) = P(x,y),
  \end{equation*}
  and so $\mathcal C_{F,\mathcal H}$ has nonempty interior since it
  contains all elements of $(\mathcal{L(H)}\otimes M_{|F|}(\mathbb
  C))^+$.

  Finally, the kernel $\Gamma(f) := [1_{\mathcal{L(H)}}] f(\lambda)$
  is obviously positive, and
    \begin{equation*}
      \Gamma*(E^+E^{+*}-E^-E^{-*}) = 1_{\mathcal{L(H)}} \otimes
      (\psi^+\psi^{+*} - \psi^-\psi^{-*})
      = 1_{\mathcal{L(H)}} \otimes \prod_{\lambda_i\in\lambda}
      (1 - \psi_i \psi_i^*)^{\lambda_i},
    \end{equation*}
    so restricting to $F\times F$ we have the last statement.
\end{proof}

We now state and prove our first realization theorem.

\begin{theorem}[Realization theorem, I]
  \label{thm:realization_I}
  Let $\varphi:X\to \mathcal{L(H)}$.  The following are equivalent:
  \begin{enumerate}
  \item[$($SC$\,)$] $\varphi \in \HLHb$;
  \item[$($AD$\,)$] There is a positive kernel $\Gamma \in \mathbb
    K^+_X(C_b(\Lambda), \mathcal{L(H)})$ such that for all $x,y\in X$,
    \begin{equation*}
      1-\varphi(x)\varphi(y)^* = \Gamma(x,y) \left( E^+(x)E^+(y)^* - 
        E^-(x)E^-(y)^* \right).
    \end{equation*}
  \end{enumerate}
  In this situation, $\varphi$ has a transfer function representation.
\end{theorem}

\begin{proof}
  Assume that ($AD$) does not hold.  This is equivalent to
  the statement that for some finite set $F$, $[1_{\mathcal{L(H)}}] -
  \varphi\varphi^* \notin \mathcal C_{F,\mathcal H}$.  A Hahn-Banach
  separation argument gives linear functional $\nu :
  \mathcal{L}(\mathcal{H}^{|F|}) \to \mathbb C$ such that
  $\nu(\mathcal C_{F,\mathcal H}) \geq 0$, $\nu([1_{\mathcal{L(H)}}])
  = 1$ and $\nu([1_{\mathcal{L(H)}}] - \varphi\varphi^*) < 0$.  Note
  that $\nu \geq 0$ since $\mathcal C_{F,\mathcal H} \supseteq
  (\mathcal{L(H)}\otimes M_{|F|}(\mathbb C))^+$, and so in particular,
  $\nu$ is continuous.

  By the Riesz representation theorem, there exists $h = (h(x)) \in
  \mathcal H^{|F|}$ such that $\nu(M) = \ip{Mh}{h}$.  If $F' = \{x\in
  F : h(x) \neq 0\}$, then $\nu'(M) := \ip{Mh|_{F'}}{h|_{F'}}$ defines
  a linear functional separating $[1_{\mathcal{L(H)}}] -
  \varphi|_{F'}\varphi|_{F'}^*$ from $\mathcal C_{F',\mathcal H}$.  So
  without loss of generality we assume that for all $x\in F$, $h(x)
  \neq 0$.  We use this to define Hilbert spaces $\mathcal H_x$ as the
  quotient completion of $\mathcal{L(H)}$ under the inner product
  \begin{equation*}
    \ip{f(x)}{g(x)} := \tfrac{1}{\|h(x)\|^2}\ip{f(x)h(x)}{g(x)h(x)},
    \quad f(x),g(x) \in \mathcal{L(H)}.
  \end{equation*}
  Since $h(x) \neq 0$, $\mathcal H_x$ is isomorphic to $\mathcal H$.

  on $F$.  Write $1_\mathcal F$ for the function which equals
  $1_{\mathcal{L(H)}}$ at every $x\in F$.  If $p\in \mathcal F$, $p^*$
  stands for the element of $\mathcal F$ with $x$th entry $p(x)^*$.
  Also, let $\chi_x(p)$ denote the element of $\mathcal F$ with all
  entries $0$ except the $x$th, which equals $p(x)$.  In this way
  $\mathcal F$ is a unital algebra with addition and multiplication
  (written as $f\cdot g$) defined entry-wise, and unit $1_\mathcal F$.

  We can also view the (quotient completion of) $\mathcal F$ as a
  Hilbert space $\mathcal H_\mathcal F = \bigoplus_{j=1}^{|F|} \mathcal
  H_j$ with inner product
  \begin{equation*}
    \ip{f}{g} = \sum_{x\in F}
    \tfrac{1}{\|h(x)\|^2}\ip{f(x)h(x)}{g(x)h(x)}.
  \end{equation*}

  For each $x,y\in F$,
  \begin{equation*}
    \ip{k(x,y) f(x)}{g(y)} = \ip{g(y)^*f(x)h(x)}{h(y)} =
    \nu((g^*(y)f(x))_{x,y\in F})
  \end{equation*}
  defines a bounded linear operator $k(x,y) \in \mathcal L(\mathcal
  H_x,\mathcal H_y)$.  For each $x\in F$, identifying $\mathcal H_x$
  with $\mathcal H$, we have $k = (k(x,y))_{x,y\in F} \in \mathcal
  L(\mathcal H^{|F|})$ with
  \begin{equation*}
    \ip{kf}{g} = \ip{(g^*f)h}{h}.
  \end{equation*}
  Extend $k$ to a kernel from $X\times X$ to $\mathcal{L(H)}$ by
  setting $k(x,y) = 0$ if either $x$ or $y$ is not in $F$.
  
  Since $\nu \geq 0$ if follows that $k \geq 0$, and so has a
  Kolmogorov decomposition $k(x,y) = k_y^*k_x$, where $k_x: X \to
  \mathcal E$ for some Hilbert space $\mathcal E$.  We therefore can
  view $\mathcal E \otimes \mathcal H_{\mathcal F}$ as a Hilbert space
  with the inner product on elementary tensors given by
  \begin{equation*}
    \ip{k_x\otimes f}{k_y\otimes g} =
    \ip{k(x,y) f(x)}{g(y)}.
  \end{equation*}

For any $f\in \mathcal F$,
  \begin{equation*}
    \begin{split}
      0 & \leq \nu\left({\left(f(y)^*((1_{\mathcal{L(H)}}\otimes
        \psi^+_\lambda(x))(1_{\mathcal{L(H)}}\otimes
        \psi^+_\lambda(y)^*) - (1_{\mathcal{L(H)}}\otimes
        \psi^-_\lambda(x))(1_{\mathcal{L(H)}}\otimes
        \psi^-_\lambda(y)^*))f(x)\right)}_{x,y\in F}\right)\\ 
      & = \ip{{\left(f(y)^*((1_{\mathcal{L(H)}}\otimes
        \psi^+_\lambda(x))(1_{\mathcal{L(H)}}\otimes
        \psi^+_\lambda(y)^*) - (1_{\mathcal{L(H)}}\otimes
        \psi^-_\lambda(x))(1_{\mathcal{L(H)}}\otimes
        \psi^-_\lambda(y)^*))f(x)\right)}_{x,y\in F}h}{h} \\
      & = \sum_{x,y} \ip{((1_{\mathcal{L(H)}}\otimes
        \psi^+_\lambda(x))(1_{\mathcal{L(H)}}\otimes
        \psi^+_\lambda(y)^*) - (1_{\mathcal{L(H)}}\otimes
        \psi^-_\lambda(x)*)(1_{\mathcal{L(H)}}\otimes
        \psi^-_\lambda(y)^*))f(x)h(x)}{f(y)h(y)} \\
      & = \sum_{x,y} \ip{((1_{\mathcal{L(H)}}\otimes
        \psi^+_\lambda(x))k(x,y)(1_{\mathcal{L(H)}}\otimes
        \psi^+_\lambda(y)^*) - (1_{\mathcal{L(H)}}\otimes
        \psi^-_\lambda(x))k(x,y)(1_{\mathcal{L(H)}}\otimes
        \psi^-_\lambda(y)^*)f(x)}{f(y)} \\
      & = \ip{\left(\left((1_{\mathcal{L(H)}}\otimes
        \psi^+_\lambda)(1_{\mathcal{L(H)}}\otimes
        \psi^{+*}_\lambda) - (1_{\mathcal{L(H)}}\otimes
        \psi^-_\lambda)(1_{\mathcal{L(H)}}\otimes
        \psi^{-*}_\lambda)\right)*k\right) f}{f}.
    \end{split}
  \end{equation*}
  Since $k(x,y)$ is $0$ when $x$ or $y$ is not in $F$, this suffices
  to show that $k$ is an admissible kernel.

  The calculation
  \begin{equation}
    \label{eq:9}
    \begin{split}
      0 & \geq  \nu\left((1_{\mathcal{L(H)}}- \varphi(x)
        \varphi(y)^*)_{x,y\in F}\right)\\
      & = \ip{(1_{\mathcal{L(H)}}-\varphi(x)\varphi(y)^*)_{x,y\in
          F} h}{h} \\
      & = \sum_{x,y} \left[\ip{h(x)}{h(y)} - \ip{\varphi(x)\varphi(y)^*
        h(x)}{h(y)} \right] \\
      & = \sum_{x,y} \ip{(k(x,y) - \varphi(x) k(x,y) \varphi(y)^*)
        1_{\mathcal{L(H)}}}{1_{\mathcal{L(H)}}} \\
      & = \ip{\left(([1_{\mathcal{L(H)}}] - \varphi\varphi^*)*k\right)
        1_\mathcal F}{1_\mathcal F}, \\
    \end{split}
  \end{equation}
  then implies that $\varphi \notin \HLHb$.

  So far we have shown that $\varphi\in\HLHb$ implies the Agler
  decomposition holds when restricted to any finite set $F$.  A
  standard application of Kurosh's theorem (see, for
  example,~\cite{MR2389623}) then gives the existence of the Agler
  decomposition on the whole of $X$.

  Now suppose that $\varphi:X \to \mathcal{L(H)}$ and that
  ($AD$) holds; that is there is a positive kernel $\Gamma
  \in \mathbb K^+_X(C_b(\Lambda),\mathcal{L(H)})$ such that for all
  $x,y\in X$,
  \begin{equation*}
    1-\varphi(x)\varphi(y)^* = \Gamma(x,y) \left( E^+(x)E^+(y)^* - 
      E^-(x)E^-(y)^* \right).
  \end{equation*}
  Fix a finite set $F \subset X$ and an admissible kernel
  $k\in\mathcal K_{\Lambda,\mathcal H}$.  Then on $F\times F$,
  \begin{equation*}
    \begin{split}
        &(1-\varphi(x)^*\varphi(y))*(k(x,y)) \\
        =& \left((\Gamma(x,y))* \left( E^+(x)E^+(y)^* - E^-(x)E^-(y)^*
          \right)\right)*(k(x,y))) \\
        =& \left(\gamma(x)(Z^+(x)Z^+(y)^* - Z^-(x)Z^-(y)^*)
          \gamma(y)^*\right)*(k(x,y)) \\
        =& \left(\gamma(x)\sum_{\lambda\in\Lambda} P_\lambda \otimes
          \left(E^+(x)(\lambda)E^+(x)(\lambda)^* -
          E^-(x)(\lambda)E^-(x)(\lambda)^*\right)\gamma(y)^*
      \right)*(k(x,y)) \\
        =& \left(\gamma(x)\sum_{\lambda\in\Lambda} P_\lambda \otimes
          \left(\psi_\lambda^+(x)k(x,y)\psi_\lambda^+(x)^* -
          \psi_\lambda^-(x)k(x,y)\psi_\lambda^-(x)^*\right)\gamma(y)^*
      \right), \\
    \end{split}
  \end{equation*}
  which is positive, since for each $\lambda\in\Lambda$,
  \begin{equation*}
      \left(\psi_\lambda^+(x)k(x,y)\psi_\lambda^+(x)^* -
          \psi_\lambda^-(x)k(x,y)\psi_\lambda^-(x)^*\right) 
      = \left( \psi_\lambda^+\psi_\lambda^{+*} - 
        \psi_\lambda^-\psi_\lambda^{-*} \right)*k
      \geq  0.
  \end{equation*}
  Thus $\varphi \in\HLHb$, and so ($SC$) and ($AD$)
  are equivalent.

\goodbreak

  Assuming ($AD$), we show that $\varphi$ has a transfer
  function representation by employing a standard lurking isometry
  argument.  To begin with, we have a Kolmogorov decomposition $\Gamma
  = \gamma^*\gamma$, and by Proposition~\ref{prop:factorization}, for
  all $\lambda\in\Lambda$, the entries of $Z^\pm$ satisfy
  $Z^\pm_\lambda(x) = \rho(E^\pm_\lambda(x))$ entry-wise (that is, for
  all $\lambda$).  Hence
  \begin{equation}
    \label{eq:14}
    1-\varphi(x)\varphi(y)^* = \gamma(x)Z^+(x)Z^+(y)^*\gamma(y)^* -
    \gamma(x)Z^-(x)Z^-(y)^*\gamma(y)^*,    
  \end{equation}
  and so bringing negative terms over to opposite sides of the
  equation, we have by the usual arguments the existence of a unitary
  $U = \begin{pmatrix} A&B \\ C&D \end{pmatrix} \in \mathcal
  L(\mathcal E \oplus \mathcal H)$ such that
  \begin{equation*}
    \begin{pmatrix}
      Z^+(x)^* \gamma(x)^* \\ \varphi(x)^*
    \end{pmatrix}
    =
    \begin{pmatrix}
      A^* & C^* \\ B^* & D^*
    \end{pmatrix}
    \begin{pmatrix}
      Z^-(x)^* \gamma(x)^* \\ 1
    \end{pmatrix}
    = 
    \begin{pmatrix}
      A^* & C^* \\ B^* & D^*
    \end{pmatrix}
    \begin{pmatrix}
      S(x)^* Z^+(x)^* \gamma(x)^* \\ 1
    \end{pmatrix},
  \end{equation*}
  where $S(x) = Y(x)Z^-(x)$.  Hence $C = \gamma(x)Z^+(x)(1 - S(x)A)$,
  and so
  \begin{equation}
    \label{eq:13}
    \gamma(x) = C(1 - S(x)A)^{-1}Y(x).
  \end{equation}
  Plugging this into the second equation,
  \begin{equation*}
    \begin{split}
      \varphi(x) &= D + \gamma(x)Z^+(x)S(x)B \\
      &= D + C(1 - S(x)A)^{-1}Y(x)Z^+(x)S(x)B \\
      &= D + C(1 - S(x)A)^{-1}P(x)S(x)B \\
      &= D + C(1 - S(x)A)^{-1}S(x)B \\
      &= D + CS(x)(1 - AS(x))^{-1}B; \\
    \end{split}
  \end{equation*}
  that is, $\varphi$ has a transfer function representation.
\end{proof}

\subsection{Realizations for ample preorderings}
\label{subsec:real-ample-preord}

The realization theorem Theorem~\ref{thm:realization_I} has the
drawback that having a transfer function representation is not enough
to ensure that a function is in $\HLHb$.  There are circumstances in
which this difficulty can be circumvented.  For example, if $\Psi$
contains $d$ test functions over a set $X$ and $\Lambda =
\{e_j\}_{j=1}^d$ (which is the classical setting), we have the result
presented in Theorem~\ref{thm:classic-real-thm}.

The reason we get so much more with the classical realization theorems
is that the auxiliary test functions are the same as the test
functions and these are by construction in our algebra.  As it
happens, with ample preorderings, something similar occurs
(Theorem~\ref{thm:ext-aux-test-fns-ample-case}).  One consequence of
the next theorem is that in the setting of ample preorderings, $\HLH$
inherits its norm from the transfer function algebra
$\mathcal{T}(X,\Lambda, \mathcal H)$, and in fact the two are equal,
thus strengthening Lemma~\ref{lem:contr-aux-t-fns-are-B-reps} in this
context.

\goodbreak

\begin{theorem}[Realization theorem, II]
  \label{thm:realization_II}
  Suppose $\Psi = \{\psi_1,\ldots \psi_d\}$ is a collection of test
  functions and $\Lambda$ is an ample preordering.  The following are
  equivalent:
  \begin{enumerate}
  \item[$($SC$\,)$] $\varphi\in \HLHb$;
  \item[$($AD1$\,)$] There exist positive kernels
    $\Gamma_\lambda:X\times X\to \mathcal{L(H)}$ so that for all
    $x,y\in X$
    \begin{equation*}
      1_{\mathcal H} - \varphi(x)\varphi(y)^* = 
    \sum_{\lambda\in\Lambda}\Gamma_\lambda(x,y)
    \prod_{\ell_j\in\lambda} ([1] - \psi_{\ell_j}(x)\psi_{\ell_j}(y)^*)
    \end{equation*}
  \item[$($AD2$\,)$] There exist positive kernels
    ${\tilde\Gamma}_\lambda : X\times X\to \mathcal{L}({\mathbb
      C}^{|\lambda|},\mathcal H)$ so that for all $x,y\in X$
    \begin{equation*}
      1_{\mathcal H} - \varphi(x)\varphi(y)^*= 
      \sum_{\lambda\in\Lambda}{\tilde\Gamma}_\lambda(x,y)
      ([1_{\mathbb C^{|\lambda|}}] -
      \sigma_\lambda(x)\sigma_\lambda(y)^*)
    \end{equation*}
  \item[$($TF$\,)$] There is a colligation $\Sigma = (U,\rho,\mathcal
    E)$ so that $\varphi=W_\Sigma$;
  \item[$($vN-a$\,)$] For every representation of $\HLH$ which is
    strictly contractive on auxiliary test functions or which is
    contractive on auxiliary test functions and either strongly or
    weakly continuous, $\|\pi(\varphi)\| \leq 1$.
  \end{enumerate}
\end{theorem}

\begin{proof}
  The proof that (\textsl{SC})$\Leftrightarrow$(\textsl{AD1}) follows
  directly from Theorem~\ref{thm:realization_I}.  A straightforward
  factorization shows that (\textsl{AD1})$\Rightarrow$(\textsl{AD2}).
  The standard lurking isometry argument as in that theorem then gives
  (\textsl{AD2})$\Rightarrow$(\textsl{TF}).  That the weak form of
  (\textsl{vN-a}) implies the strong form which then implies the
  strict form is also immediate.

  By Theorem~\ref{thm:ext-aux-test-fns-ample-case}, the auxiliary test
  functions are in $H_1^\infty({\mathcal K}_{\Lambda,\mathbb C^n})$
  for appropriate $n$ and these functions generates the same
  collection of admissible kernels.  Using the fact that the operator
  in the colligation for $\varphi$ is unitary, the usual sort of
  calculation shows that if $\varphi$ has a transfer function
  representation, then for $G = C(1-SA)^{-1}$,
  \begin{equation*}
    \left((1-\varphi(x)\varphi(y)^*k(x,y)\right)
    = \left(G(x)\left((1-S(x)S(y)^*)k(x,y)\right)G(y)^*\right) \geq
    0,
  \end{equation*}
  and so (\textsl{TF})$\Rightarrow$(\textsl{SC}).

  If $\pi$ is a representation of $\HLH$ which is strictly contractive
  on auxiliary test functions, and if we interpret
  \begin{equation*}
    \pi(\varphi) = D\otimes 1 + C \otimes \pi(S) (1\otimes 1 - A\otimes
    \pi(S))^{-1} B\otimes 1,\qquad
    \pi(S) = \sum_{\lambda\in\Lambda} P_\lambda \otimes
    \pi(\sigma_\lambda),
  \end{equation*}
  then a nearly identical argument to that of the last paragraph shows
  that $1-\pi(\varphi)\pi(\varphi)^* \geq 0$; that is, $\pi$ is a
  contractive representation of $\HLH$.  Hence (\textsl{TF}) implies
  the strict form of (\textsl{VN-a}).  On the other hand, if $\pi$ is
  only assumed to be weakly continuous, then we argue as in
  \cite{MR2389623}, first scaling $A$ and $C$ to $rA$ and $rC$ for
  $r<1$ and calling the resulting functions $\varphi_r$, then
  approximating $\varphi_r$ by polynomials in $S$ as at the end of the
  proof Theorem~\ref{thm:trfr-fns-span-alg}.  The representation is
  easily seen to be contractive on these polynomials.  Since these can
  be chosen to converge pointwise to $\varphi_r$, the representation
  is contractive on $\varphi_r$ for all $r$.  Taking $r$ to $1$ we
  have pointwise convergence to $\varphi$, and so once again,
  $1-\pi(\varphi)\pi(\varphi)^* \geq 0$.

  Finally, suppose that the strict form of (\textsl{vN-a}) holds.  Fix
  $\varphi\in\HLH$.  We show that for $k\in \mathcal
  K_{\Lambda,\mathcal H}$, $([1_{\mathcal H}]- \varphi\varphi^*)*k
  \geq 0$, and so (\textsl{SC}) holds as well.  For this, it suffices
  to prove that for fixed $k\in \mathcal K_{\Lambda,\mathcal H}$,
  $([1_{\mathcal H}] - \varphi\varphi^*)*k \geq 0$ when we restrict to
  a finite subset $F\subset X$.  So fix a finite set $F\subset X$.  On
  $F$ replace the test functions $\Psi$ by $\Psi_r = \{\psi_r = r\psi
  : \psi\in\Psi\}$, where $r>1$ is sufficiently close to $1$ so that
  $\sup_{\psi_r\in\Psi_r} |\psi_r(x)| < 1$ for all $x\in F$ (this is
  possible since $F$ is finite).  Define in the same way as before,
  $\mathcal K^r_{\Lambda,\mathcal H}$ on $F$ with these test
  functions, as well as $H^\infty(F,\mathcal K^r_{\Lambda,\mathcal H})$.

  Since for $k_r\in \mathcal K^r_{\Lambda,\mathcal H}$ and
  $\lambda\in\Lambda$,
  \begin{equation}
    \label{eq:2}
    \prod_{\ell\in\lambda} ([1_{\mathcal H}]-\psi_\ell\psi_\ell^*)*k_r
    = \tfrac{1}{r^2} \left((r^2-1)[1_{\mathcal H}] + [1_{\mathcal H}]
      - \psi_{r,\ell}\psi_{r,\ell}^*\right)*k_r \geq 0,
  \end{equation}
  so it follows that $\mathcal K^r_{\Lambda,\mathcal H} \subseteq
  \mathcal K_{\Lambda,\mathcal H}$ on $F$.  Hence $H^\infty(F,\mathcal
  K_{\Lambda,\mathcal H}):= H^\infty(X,\mathcal K_{\Lambda,\mathcal
    H})|_F \subseteq H^\infty(F,\mathcal K^r_{\Lambda,\mathcal H})$.
  Hence, we view $H^\infty(F,\mathcal K_{\Lambda,\mathcal H})$ as a
  subalgebra of $H^\infty(F,\mathcal K^r_{\Lambda,\mathcal H})$
  endowed with the $H^\infty(F,\mathcal K^r_{\Lambda,\mathcal
    H})$-norm, which we write as $\|\cdot\|_r$.  For $k_r\in \mathcal
  K^r_{\Lambda,\mathcal H}$, the map $\pi$ taking $f\in
  H^\infty(F,\mathcal K_{\Lambda,\mathcal H})$ to $M_f$ on $H^2(k_r)$
  defines a strictly contractive representation of
  $H^\infty(F,\mathcal K_{\Lambda,\mathcal H})$, since \eqref{eq:2}
  implies that for $\psi\in\Psi$, $\|\psi\|_r \leq 1/r$.  It follows
  then that under the assumption that the strict form of
  (\textsl{vN-a}) holds for $\varphi$,
  \begin{equation}\label{eq:4}
    ([1_{\mathcal H}] - \varphi^*\varphi)*k_r \geq 0, \qquad k_r\in
    \mathcal K^r_{\Lambda,\mathcal H}.
  \end{equation}

  Fix $k\in\mathcal K_{\Lambda,\mathcal H}$.  For any $r>1$ as above,
  the kernel $k_0$ defined by
  \begin{equation*}
    k_0(x,y) =
    \begin{cases}
      k(x,x) & y=x;\\
      0 & \text{otherwise,}
    \end{cases}
  \end{equation*}
  defines a kernel in $\mathcal K^r_{\Lambda,\mathcal H}$.  Fix
  $t\in(0,1)$.  We show that $tk-(1-t)k_0 \in \mathcal
  K^r_{\Lambda,\mathcal H}$ for any $r>1$ and sufficiently close
  to~$1$.  It will follow then that for such $r$ and for all
  $\lambda\in\Lambda$,
  \begin{equation}
    \label{eq:3}
    \begin{split}
      & \prod_{\ell\in\lambda} ([1]-\psi_{r,\ell}\psi_{r,\ell}^*)*[tk
      - (1-t)k_0] \\
      = & \prod_{\ell\in\lambda} \left[t([1]-\psi_\ell^*\psi_\ell)*k +
        (1-t)([1]-\psi_{r,\ell}\psi_{r,\ell}^*)*k_0 - t(r^2-1)
        \psi_\ell\psi_\ell^* * k\right] \geq 0.
    \end{split}
  \end{equation}
  We do this by proving that in this case,
  \begin{equation*}
    (1-t)\prod_{\ell\in\lambda} ([1]-\psi_{r,\ell}\psi_{r,\ell}^*)*k_0
    \geq  t(r^2-1)\prod_{\ell\in\lambda} \psi_\ell\psi_\ell^* * k.
  \end{equation*}

  Assume that $k(x,x) > 0$ for all $x\in F$.  We can do this since if
  $k(x,x) = 0$ for some $x$, then $k(x,y)=k(y,x) = 0$ for all $y$.  So
  without loss of generality we restrict to $F'\subset F$ with the
  property that the diagonal entries of $k$ are strictly positive.
  Since $F$ is finite $(1-t)\prod_{\ell\in\lambda}
  ([1]-\psi_{r,\ell}^*\psi_{r,\ell})*k_0$ is a diagonal matrix with
  diagonal entries of the form $(1-t)\prod_{\ell\in\lambda}
  (1-r^2|\psi_\ell(x)|)k(x,x) \geq (1-t)\prod_{\ell\in\lambda}
  (1-r^2|\psi_\ell(x)|)\epsilon > 0$ for some $\epsilon > 0$.

  Hence it suffices to show that for $1_F$ the usual identity matrix
  over $\mathbb C^{|F|}$,
  \begin{equation*}
    \frac{\epsilon(1-t)}{t} \prod_{\ell\in\lambda}
    \frac{1-r^2|\psi(x)|}{r^2-1} 1_F\otimes 1_{\mathcal H} \geq
    \prod_{\ell\in\lambda} \psi_\ell\psi_\ell*k.
  \end{equation*}
  However, as $r\searrow 1$, $\dfrac{1-r^2|\psi(x)|}{r^2-1} \nearrow
  \infty$.  Thus for $r$ sufficiently close to $1$, \eqref{eq:3} is
  satisfied.

  Since \eqref{eq:4} holds for all sufficiently small $r > 1$ and
  $k_r\in \mathcal K^r_{\Lambda,\mathcal H}$, it follows that for all
  all such $r$ and $k \in \mathcal K^r_{\Lambda,\mathcal H}$ and
  $t\in(0,1)$,
  \begin{equation*}
    ([1_{\mathcal H}] - \varphi\varphi^*)*[tk + (1-t)k_0] \geq 0,
  \end{equation*}
  and so taking $t\nearrow 1$, we have $([1_{\mathcal H}] -
  \varphi\varphi^*)*k \geq 0$ on $F\times F$.  The set $F\subset X$
  was arbitrary, and so we conclude that $\varphi\in H^\infty(X,\mathcal
  K_{\Lambda,\mathcal H})$.
\end{proof}

\begin{corollary}
  \label{cor:HLH_equals_TLH_w_ample_po}
  With $\Lambda$ an ample preordering, $\HLH$ is isometrically
  isomorphic to\break $\mathcal{T}(X,\Lambda, \mathcal H)$ and $\ALH$
  is isometrically isomorphic to $\mathcal{T}^A(X,\Lambda, \mathcal
  H)$.
\end{corollary}

\begin{corollary}
  \label{cor:n-contr-gives-cc-for-ample}
  Let $\Lambda$ be an ample preordering over $d$ test functions.  Then
  any strictly contractive or weakly continuous contractive
  representation of $\HLH$ which is $2^{d-1}$-contractive is
  completely contractive.  Likewise, any $2^{d-1}$-contractive
  representation of $\ALH$ is completely contractive.
\end{corollary}

\begin{proof}
  This follows from the last two theorems since a representation which
  is $2^{d-1}$-contractive is contractive on auxiliary test functions.
\end{proof}

\subsection{Agler-Pick interpolation}
\label{subsec:agler-pick-interp}

It is now standard practice to apply the realization theorem to Pick
type interpolation problems.

\goodbreak

\begin{theorem}[Agler-Pick interpolation]
  \label{thm:agler-pick_interpolation}
  Let $\Lambda$ be an ample preordering, $X_0 \subseteq X$.  For each
  $x\in X_0$, let $a_x, b_x \in \mathcal{L(H)}$.  The following are
  equivalent:
  \begin{enumerate}
  \item There exists $\varphi\in \HLHb$ such that for all $x\in X_0$,
    $b_x = \varphi(x) a_x$;
  \item There exists a positive kernel $\,\Gamma:X_0\times X_0\to
    C_b(\Lambda,\mathcal{L(H)})$ so that for all $x,y\in X_0$
    \begin{equation*}
      \left(((a_x a_y^* - b_x b_y^*)\otimes 1_n)*k_\sigma(x,y)\right)
      \geq 0,
    \end{equation*}
    where $k_\sigma = 1_\mathcal{L(H)} \otimes (1 -
    \sigma_{\lambda^m}(x) \sigma_{\lambda^m}(y)^*)^{-1}$ and $n =
    2^{|\lambda|-1}$.
  \end{enumerate}
\end{theorem}

\begin{proof}
  The proof follows the first part of the proof of the realization
  theorem, giving a transfer function $W$ such that $b_x = W(x) a_x$
  for all $x\in X_0$.  This transfer function is well defined for all
  $x \in X$, and hence $W$ extends to $\varphi\in \HLHb$.
\end{proof}

In the case of the bidisk, the function $\sigma$ will be a $2\times 2$
matrix valued inner function, and though heavily constrained, it will
not be uniquely determined.  For practical purposes it would be very
interesting to know a concrete example of a choice of $\sigma$.

Taking $b_x = \sqrt{\delta}$ for all $x\in X_0=X$ in
Theorem~\ref{thm:agler-pick_interpolation} gives the so-called
\textsl{Toeplitz-corona theorem}.  We need a special case of this,
stated in the following lemma.

\begin{lemma}
  \label{lem:Z+inv_in_Hinf}
  Let $\Lambda$ be an ample preordering.  For $\lambda\in\Lambda$,
  there exist $\omega_\lambda$ with entries in $\HL$ such that for all
  $x\in X$, $\psi^+_\lambda(x) \omega_\lambda(x) = 1$.  Consequently,
  given a unitary colligation $(U,\mathcal E,\rho)$, there exists $Y$
  with entries in $\HL$ such that for all $x$, $Z^+(x)Y(x) = 1$.
\end{lemma}

\begin{proof}
  Recall from the definition in
  subsection~\ref{subsec:auxil-test-funct}, $\psi^+_\lambda : X \to
  \mathbb C^{2^{|\lambda|-1}}$ with entries which are products of test
  functions and hence in $\HLb$ and the first entry equal to $1$.  If
  we replace this $1$ by $0$ and call the resulting function
  ${\tilde\psi}^+_\lambda$, then we see that for any admissible kernel
  $k$,
  \begin{equation*}
    (\psi^+_\lambda\psi^{+*}_\lambda - [1])*k_\sigma =
    ({\tilde\psi}^+_\lambda{\tilde\psi}^{+*}_\lambda) * k_\sigma \geq
    0,
  \end{equation*}
  where $k_\sigma$ is as in the statement of
  Theorem~\ref{thm:agler-pick_interpolation}.

  Now apply Theorem~\ref{thm:agler-pick_interpolation} to get a
  function in $\omega_\lambda \in H^\infty_1(X,\mathcal K_{\Lambda,
    \mathbb C^{2^{|\lambda|-1}}})$.  It might be objected that the
  $\psi^\pm_\lambda$s are not square matrices, but this can be
  rectified by padding with rows of zeros.  The definition of $Y$ in
  terms of the $\omega_\lambda$s then yields the final statement of
  the theorem.
\end{proof}

\begin{corollary}
  \label{cor:gamma_limit_of_HLH_fns}
  Let $\Lambda$ be an ample preordering, $(U,\mathcal E,\rho)$ a
  unitary colligation.  Then for $S$ and $Y$ chosen as in
  Theorem~\ref{thm:ext-aux-test-fns-ample-case} and
  Lemma~\ref{lem:Z+inv_in_Hinf},
  \begin{equation*}
    \gamma(x):= C(1-S(x)A)^{-1}Y(x)
  \end{equation*}
  is the pointwise limit functions in $\HLH$.
\end{corollary}

\begin{proof}
  By the now standard arguments whereby we scale $A$ and $C$ to $rA$
  and $rC$, $0<r<1$, approximate $rC(1-S(x)rA)^{-1}$ by polynomials in
  $S$ and then take limits, the result follows since the entries of
  $S$ and $Y$ are in $\HLH$.
\end{proof}

\subsection{Brehmer representations again}
\label{subsec:brehm-repr-again}

Using the last corollary, we can now include a statement concerning
Brehmer representations to the realization theorem for ample
preorderings.

\goodbreak

\begin{theorem}[Realization theorem, III]
  \label{thm:realization_III}
  Suppose $\Psi = \{\psi_1,\ldots \psi_d\}$ is a collection of test
  functions and $\Lambda$ is an ample preordering.  The following are
  equivalent:
  \begin{enumerate}
  \item[$($SC$\,)$] $\varphi\in \HLHb$;
  \item[$($AD1$\,)$] There exist positive kernels
    $\Gamma_\lambda:X\times X\to \mathcal{L(H)}$ so that for all
    $x,y\in X$
    \begin{equation*}
      1-\varphi(x)\varphi(y)^*= 
    \sum_{\lambda\in\Lambda}\Gamma_\lambda(x,y)
    \prod_{\ell_j\in\lambda} ([1] - \psi_{\ell_j}(x)\psi_{\ell_j}(y)^*)
    \end{equation*}
  \item[$($AD2$\,)$] There exist positive kernels
    ${\tilde\Gamma}_\lambda : X\times X\to \mathcal{L}({\mathbb
      C}^{|\lambda|},\mathcal H)$ so that for all $x,y\in X$
    \begin{equation*}
      1-\varphi(x)\varphi(y)^*= 
      \sum_{\lambda\in\Lambda}{\tilde\Gamma}_\lambda(x,y)
      ([1_{\mathbb C^{|\lambda|}}] -
      \sigma_\lambda(x)\sigma_\lambda(y)^*)
    \end{equation*}
  \item[$($TF$\,)$] There is a unitary colligation $\Sigma =
    (U,\rho,\mathcal E)$ so that $\varphi=W_\Sigma$;
  \item[$($vN-B$\,)$] For every strict / strongly continuous / weakly
    continuous Brehmer representation $\pi$ of\break $\HLH$,
    $\|\pi(\varphi)\| \leq 1$;
  \item[$($vN-a$\,)$] For every representation of $\HLH$ which is
    strictly contractive on auxiliary test functions or which is
    contractive on auxiliary test functions and either strongly or
    weakly continuous, $\|\pi(\varphi)\| \leq 1$.
  \end{enumerate}
\end{theorem}

\begin{proof}
  The equivalence of all parts except for (\textsl{vN-B}) are the
  content of Theorem~\ref{thm:realization_II}.  Since by
  Lemma~\ref{lem:contr-aux-t-fns-are-B-reps}, representations which
  are strictly contractive on auxiliary test functions are strictly
  contractive Brehmer representations, to finish the proof it suffices
  to prove that if (\textsl{AD1}) holds, then any weakly contractive
  Brehmer representation is contractive.

  We can rewrite the statement of (\textsl{AD1}) as being that that
  there exists a positive kernel $\Gamma$ with Kolmogorov
  decomposition $\Gamma = \gamma\gamma^*$ such that for all $x,y$,
  \begin{equation*}
    1-\varphi(x)\varphi(y)^*= \gamma(x)(Z^+(x)Z^{+*}(y) -
    Z^-(x)Z^{-*}(y))\gamma^*.
  \end{equation*}
  A lurking isometry argument then gives that there is a unitary
  operator $U = \begin{pmatrix} A & B \\ C & D \end{pmatrix}$, so that
  \begin{equation*}
    C = \gamma(x) (Z^+(x) - Z^-(x)A).
  \end{equation*}
  According to Lemma~\ref{lem:Z+inv_in_Hinf}, we can choose $S$ with
  entries in $\HL$ to be strictly contractive for all $x$, and so
  $\gamma(x)Z^+(x) = C(1-S(x)A)^{-1}$.  Applying
  Lemma~\ref{lem:Z+inv_in_Hinf}, we can choose $Y$ with entries in
  $\HL$ such that for all $x$, $Z^+(x) Y(x) = 1$, and hence
  \begin{equation*}
    \gamma(x) = C(1-S(x)A)^{-1}Y(x),
  \end{equation*}
  which by Corollary~\ref{cor:gamma_limit_of_HLH_fns} as a limit of
  functions in $\HLH$.  The lurking isometry argument also gives that
  \begin{equation*}
    \varphi(x) = D + \gamma(x) Z^-(x) B.
  \end{equation*}

  Let $\pi$ be a weakly continuous Brehmer representation.  If
  $\gamma$ has entries in $\HLH$, then
  \begin{equation*}
    \pi(\varphi) = 1\otimes D + (\pi(\gamma)\pi(Z^-))\otimes B,
  \end{equation*}
  where we are using the shorthand notation of ``$\pi(\gamma)$'' and
  ``$\pi(Z^-)$'' for the entrywise application of $\pi$ to these
  functions.  A straightforward calculation using the fact that $U$ in
  the colligation is unitary gives
  \begin{equation*}
    1-\pi(\varphi)\pi(\varphi)^* = \pi(\gamma)(\pi(Z^+)\pi(Z^+)^* -
    \pi(Z^-)\pi(Z^-)^*)\pi(\gamma)^* \geq 0.
  \end{equation*}
  More generally, we approximate $\gamma$ by function with entries in
  $\HLH$.  Taking limits, we still find that
  $1-\pi(\varphi)\pi(\varphi)^* \geq 0$; that is, $\pi$ is contractive.
\end{proof}

We close this section with an analogue of Brehmer's theorem.

\begin{proposition}
  \label{prop:dilating_Brehmer_reps}
  Let $\pi$ be a Brehmer representation of $\ALH$ or weakly continuous
  Brehmer representation of $\HLH \subset C_b(X,\mathcal{L(H)})$.
  Then $\pi$ dilates to a $($weakly continuous$)$ Brehmer
  $*$-representation $\tilde\pi$ of $C_b(X,\mathcal{L(H)})$.
\end{proposition}

\begin{proof}
  This follows from Corollary~\ref{cor:HLH_equals_TLH_w_ample_po} and
  Theorem~\ref{thm:dilation-theorem}.
\end{proof}

\subsection{Algebras generated by two test functions}
\label{subsec:algebr-gener-two}

Brehmer's original theorem is a dilation theorem which works over the
polydisk, but requires a special class of representations.  On the
other hand, for the $\mathbb D^2$, And\^{o}'s theorem shows that such
dilation results hold for a broader class of representations.  We
first state a bidisk version of the realization theorem.  The
emphasis here is on the equivalence of the two versions of
von~Neumann's inequality, since by
Lemma~\ref{lem:contr-aux-t-fns-are-B-reps}, the collection of 
representations which are strictly contractive on auxiliary test
functions is the smallest set of representations we consider, while
the weakly continuous Brehmer representations form the largest set.

\begin{theorem}
  \label{thm:2-var-polydisk_real}
  Let $\varphi: \mathbb D^2 \to \mathcal{L(H)}$.  The following are
  equivalent:
  \begin{enumerate}
  \item[$($SC$\,)$] $([1]-\varphi\varphi^*)*k_s \geq 0$, $k_s(z,w) =
    (1-z_1w_1)^{-1}(1-z_2w_2)^{-1}$, or equivalently $\varphi \in
    H^\infty_1(\mathbb D^2,\mathcal{L(H)})$;
  \item[$($AD1$\,)$] For every admissible kernel $k\in \{k\geq 0:
    (1-Z_jZ_j^*)*k \geq 0,\ j=1,2\}$, we have $(1-\varphi\varphi^*)*k
    \geq 0$;
  \item[$($AD2$\,)$] There exist positive kernels $\Gamma_1$, $\Gamma_2$
    such that $[1]-\varphi\varphi^* = \Gamma_1*([1]-Z_1Z_1^*)+
    \Gamma_2*([1]-Z_2Z_2^*)$, where $Z_j(z) = z_j$;
  \item[$($vN1$\,)$] $\varphi \in H^\infty(\mathbb D^2,\mathcal{L(H)})$
    and for every weakly continuous Brehmer representation $\pi$
    $($so\break $1-\pi(z_1)\pi(z_1)^*-\pi(z_2)\pi(z_2)^*+
    \pi(z_1)\pi(z_2)\pi(z_1)^*\pi(z_2)^* \geq 0)$, we have
    $\|\pi(\varphi)\| \leq 1$;
  \item[$($vN2$\,)$] $\varphi \in H^\infty(\mathbb D^2,\mathcal{L(H)})$
    and for every strictly contractive representation $\pi$ $($so
    $\|\pi(z_i)\| < 1)$, we have $\|\pi(\varphi)\| \leq 1$.
  \end{enumerate}
\end{theorem}

We next show how to generalize this to algebras over general domains
generated by a pair of test functions.

Let us assume that $\Psi = \{\psi_1,\psi_2\}$ is a collection of test
functions on a set $X$ and $\Lambda$ be the standard ample preordering
with maximal element $(1,1)$, while $\Lambda_0 = \{(1,0),(0,1)\}$, the
nearly ample preordering used for the standard realization theorem.
Write $\mathcal K_0$ for the set of admissible kernels associated to
$\Lambda_0$; so $k\in\mathcal K_0$ means that $k\geq 0$ and
$([1]-\psi_j\psi_j^*)*k \geq 0$, $j=1,2$.

By assumption, for each $x\in X$, $\max\{|\psi_1(x)|,|\psi_2(x)|\} <
1$ and the elements of $\Psi$ separate the points of $X$.  Hence by
Lemma~\ref{lem:id_HLH_w_pdisk_subalg}, there is an injective mapping
$\xi$ of $X$ onto a subset $\Omega$ of $\mathbb D^2$ given by
$\xi:x\mapsto z = (z_1,z_2) = (\psi_1(x),\psi_2(x))$.

\begin{theorem}
  \label{thm:2-var-realization}
  Let $\Psi = \{\psi_1,\psi_2\}$ be a collection of test functions on
  a set $X$.  Let $\Lambda$ be the ample preordering with maximal
  element $(1,1)$, $\Lambda_0$ the preordering $\{(1,0),(0,1)\}$.  For
  $\varphi: \mathbb D^2 \to \mathcal{L(H)}$, the following are
  equivalent:
  \begin{enumerate}
  \item[$($SC1$\,)$] $\varphi \in H^\infty_1(X, \mathcal
    K_{\Lambda,\mathcal{L(H)}})$, the closed unit ball of
    $H^\infty(X,\mathcal K_{\Lambda,\mathcal{L(H)}})$;
  \item[$($SC2$\,)$] $\varphi \in H^\infty_1(X, \mathcal
    K_{\Lambda_0,\mathcal{L(H)}})$, the closed unit ball of
    $H^\infty(X, \mathcal K_{\Lambda_0,\mathcal{L(H)}})$;
  \item[$($vN1$\,)$] $\varphi \in H^\infty(X, \mathcal
    K_{\Lambda,\mathcal{L(H)}})$ and for every weakly continuous
    Brehmer representation $\pi$, we have $\|\pi(\varphi)\| \leq 1$;
  \item[$($vN2$\,)$] $\varphi \in H^\infty(X, \mathcal
    K_{\Lambda_0,\mathcal{L(H)}})$ and for every strictly contractive
    representation $\pi$ $($so $\|\pi(\psi_i)\| < 1)$, we have
    $\|\pi(\varphi)\| \leq 1$.
  \end{enumerate}
  In particular, $\Lambda$ and $\Lambda_0$ are equivalent
  preorderings; that is, $H^\infty(X, \mathcal
  K_{\Lambda_0,\mathcal{L(H)}}) = H^\infty(X, \mathcal
  K_{\Lambda,\mathcal{L(H)}})$.
\end{theorem}

\begin{proof}
  The implication ($SC2$) implies ($SC1$) is trivial, while ($vN1$) is
  equivalent to ($SC1$) and ($vN2$) is equivalent to ($SC2$) by
  Theorem~\ref{thm:realization_II}.  We therefore only need to show
  that ($SC1$) implies ($SC2$).

  Assume $\varphi \in H^\infty_1(X, \mathcal K_{\Lambda,
    \mathcal{L(H)}})$.  Using the embedding of $X$ in the polydisk
  given in Lemma~\ref{lem:id_HLH_w_pdisk_subalg}, we let
  $\tilde\varphi = \varphi\circ\xi^{-1}$.  Fix $\lambda = (1,1)$.  We
  suppress the dependence on $\lambda$ in the following, writing
  $\psi^+ = \begin{pmatrix} 1 & \psi_1\psi_2 \end{pmatrix}$, $\psi^- =
  \begin{pmatrix} \psi_1 & \psi_2 \end{pmatrix}$, and $\sigma_0 =
  \psi^{+\,*} |\psi^+|^{-2} \psi^-$.  When composed with $\xi^{-1}$
  these become ${\tilde\psi}^+ = \begin{pmatrix} 1 & z_1z_2
  \end{pmatrix}$, ${\tilde\psi}^- = \begin{pmatrix} z_1 & z_2
  \end{pmatrix}$, and ${\tilde\sigma}_0 = {\tilde\psi}^{+\,*}
  |{\tilde\psi}^+|^{-2} {\tilde\psi}^-$ on $\Omega = \xi(X) \subset
  \mathbb D^2$.

  By Theorem~\ref{thm:ext-aux-test-fns-ample-case}, ${\tilde\sigma}_0$
  extends to an auxiliary test function $\tilde\sigma \in
  H^\infty(\mathbb D^2,M_2(\mathbb C))$ corresponding to $\lambda$.
  Set $\sigma = (\tilde\sigma | \Omega)\circ \xi$.  Easy calculations
  show that $\psi^+ \sigma = \psi^-$ and $([1_2] - \sigma\sigma^*) *
  (k_s\otimes 1_2) \geq 0$, where $k_s(x,y) =
  (1-\psi_1(x)\psi_1(y))^{-1} (1-\psi_2(x)\psi_2(y))^{-1}$.  In other
  words, $\sigma$ is an auxiliary test function corresponding to
  $\lambda = (1,1)$, this time over $X$.

  Also by Theorem~\ref{thm:ext-aux-test-fns-ample-case}, $\varphi \in
  H^\infty_1(X, \mathcal K_{\Lambda,\mathcal{L(H)}})$ if and only if
  $(1_2\otimes ([1_{\mathcal{L(H)}}] - \varphi\varphi^*)) * k_\sigma
  \otimes 1_{\mathcal{L(H)}} \geq 0$, where $k_\sigma(x,y) = (1_2 -
  \sigma(x)\sigma(y))^{-1}$.  Translating over to $\Omega$, this
  becomes
  \begin{equation*}
    (1_2\otimes ([1_{\mathcal{L(H)}}] -
    {\tilde\varphi}{\tilde\varphi}^*)) * k_{\tilde\sigma}
  \otimes 1_{\mathcal{L(H)}} \geq 0,
  \end{equation*}
  on $\Omega\times\Omega$, where $k_{\tilde\sigma}$ is defined
  analogously to $k_\sigma$.  By
  Theorem~\ref{thm:agler-pick_interpolation}, $\tilde\varphi$ extends
  to a function in $H^\infty(\mathbb D^2, \mathcal{L(H)})$, and so by
  Agler's realization theorem in the case of the bidisk
  (Theorem~\ref{thm:Aglers-realization}), $\tilde\varphi \in
  H^\infty_1(\mathbb D^2, \mathcal K_{\Lambda_0,\mathcal{L(H)}})$.
  Restricting to $\Omega$ and mapping back to $X$, it follows that
  $\varphi \in H^\infty_1(X, \mathcal K_{\Lambda_0,\mathcal{L(H)}})$.
\end{proof}

\subsection{Realizations with nearly ample preorderings}
\label{subsec:real-nearly-ample-preord}

It is now possible to extend the results of the previous section to
more than two test functions.  Following the template set there, we
first do this over the polydisk, thus obtaining a generalization of
the results in~\cite{MR2502431}, and then to general algebras obtained
with a finite collection of test functions.  We begin with a
$d$-variable version of Theorem~\ref{thm:2-var-polydisk_real}.

Throughout this section we assume that we have the standard ample
preordering $\Lambda^a = \{1\}$ over the collection of test functions
is $\Psi = \{\psi_1,\dots,\psi_d\}$, where here $1$ stands for the
$d$-tuple with all values $1$, and a standard nearly ample preordering
$\Lambda^{na} = \{\lambda_1, \lambda_2\}$, where $\lambda_i$ is a
$d$-tuple with the $j_i$th entry equal to $0$ and all others equal to
$1$, and $\lambda_1 \neq \lambda_2$.  In the first case, the
collection of kernels is particularly simple.  By
Lemma~\ref{lem:adm_kernels_for_ample_po} they are all subordinate to
the so-called Szeg\H{o} kernel, $k_s$.  In the nearly ample case the
set is more complex, since then
\begin{equation*}
  k\in \mathcal K_{\Lambda^{na}} :=
  \{k\geq 0: \prod_{j\neq j_1} (1-\psi_j \psi_j^*)*k \geq 0 \text{ and }
  \prod_{j\neq j_1} (1-Z_j Z_j^*)*k \geq 0,\ j=1,2\}.
\end{equation*}
Recall from Theorem~\ref{thm:ample-near-ample-equiv-polydsk} that over
the polydisk with test functions equal to the coordinate functions,
the algebras we get from these two collections of kernels are the
same, with equal norms.  Also, since by
Lemma~\ref{lem:preorderings_are_equivalent} $\{1, \lambda_1,
\lambda_2\}$ is also an ample preordering equivalent to $\Lambda^a$,
by Theorem~\ref{thm:ext-aux-test-fns-ample-case}, for any collection
of $d$ test functions over a set $X$, we have that the auxiliary test
functions $\sigma_1 \in H^\infty(X, {\mathcal K}_{\Lambda,\mathbb
  C^{2^{d-1}}})$ and $\sigma_{\lambda_1}, \sigma_{\lambda_2} \in
H^\infty(X, {\mathcal K}_{\Lambda,\mathbb C^{2^{d-2}}})$.  Thus we
have the following generalization of the main theorem
of~\cite{MR2502431}.

\begin{theorem}
  \label{thm:d-var-polydisk_real}
  Let $\varphi: \mathbb D^d \to \mathcal{L(H)}$.  The following are
  equivalent:
  \begin{enumerate}
  \item[$($SC1$\,)$] $([1]-\varphi\varphi^*)*k_s \geq 0$, or
    equivalently, $\varphi \in H^\infty_1(\mathbb
    D^d,\mathcal{L(H)})$;
  \item[$($SC2$\,)$] For every admissible kernel $k\in \mathcal
    K_{\Lambda^{na}}$, $([1]-\varphi\varphi^*)*k \geq 0$;
  \item[$($AD$\,)$] There exists a positive kernels $\Gamma$,
    $\Gamma_1$, $\Gamma_2$ such that
    \begin{equation*}
      [1]-\varphi\varphi^* = \Gamma*\prod_{j=1}^d ([1]-Z_jZ_j^*)
      = \Gamma_1*\prod_{j\neq j_1}
        ([1]-Z_jZ_j^*) + \Gamma_2*\prod_{j\neq j_2} ([1]-Z_jZ_j^*),
    \end{equation*}
    where $Z_j(z) = z_j$;
  \item[$($TF$\,)$] There are unitary colligations $\Sigma_a$ and
    $\Sigma_{na}$ in the ample and nearly ample setting respectively,
    such that $\varphi$ has transfer function representations
    \begin{equation*}
        \varphi = W_{\Sigma_a} = W_{\Sigma_{na}};
    \end{equation*}
  \item[$($vN-a$\,)$] $\varphi \in H^\infty(\mathbb D^d,\mathcal H) =
    H^\infty(\mathbb D^d, {\mathcal K}_{\Lambda^a,\mathcal H}) =
    H^\infty(\mathbb D^d, {\mathcal K}_{\Lambda^{na},\mathcal H})$ and
    for every representation $\pi$ which is strictly contractive on
    the auxiliary test function $\sigma_1$ $($respectively, the
    auxiliary test functions $\sigma_{\lambda_1},
    \sigma_{\lambda_2})$, or contractive on these and either strongly
    or weakly continuous, we have $\|\pi(\varphi)\| \leq 1$;
  \item[$($vN-B$\,)$] $\varphi \in H^\infty(\mathbb D^d,\mathcal H) =
    H^\infty(\mathbb D^d, {\mathcal K}_{\Lambda^a,\mathcal H}) =
    H^\infty(\mathbb D^d, {\mathcal K}_{\Lambda^{na},\mathcal H})$ and
    for every representation $\pi$ which is a strict / strongly
    continuous/ weakly continuous Brehmer representation with respect
    to $\Lambda^a$ $($respectively, $\Lambda^{na})$, we have
    $\|\pi(\varphi)\| \leq 1$.
  \end{enumerate}
\end{theorem}

There are also statements regarding transfer function representations
which we have not included here.

\begin{proof}[Proof of Theorem~\ref{thm:d-var-polydisk_real}]
This follows immediately from
Theorem~\ref{thm:ample-near-ample-equiv-polydsk} and an application of
the realization theorem to the two equivalent preorderings $\Lambda^a$
and $\Lambda^{na}$.
\end{proof}

We now state a $d$-variable version of
Theorem~\ref{thm:2-var-realization}.  As usual, $k_s$ stands for the
Szeg\H{o} kernel $\prod_1^d ([1]-\psi_j\psi_j^*)^{-1}$.

\goodbreak

\begin{theorem}[Realization theorem, IV]
  \label{thm:realization_IV}
  Suppose $\Psi = \{\psi_1,\ldots \psi_d\}$ is a collection of test
  functions over a set $X$, $\Lambda^{na}$ is a standard nearly ample
  preordering under the standard ample preordering $\Lambda^{a} =
  \{1\}$.  The following are equivalent:
  \begin{enumerate}
  \item[$($SC1$\,)$] $([1]-\varphi\varphi^*)*k_s \geq 0$, or
    equivalently, $\varphi \in H_1^\infty({\mathcal
      K}_{\Lambda^a,\mathcal H})$;
  \item[$($SC2$\,)$] For every admissible kernel $k\in \mathcal
    K_{\Lambda^{na}}$, $([1]-\varphi\varphi^*)*k \geq 0$, or
    equivalently, $\varphi \in H_1^\infty({\mathcal
      K}_{\Lambda^{na},\mathcal H})$;
  \item[$($AD$\,)$] There exists a positive kernels $\Gamma$,
    $\Gamma_1$, $\Gamma_2$ such that
    \begin{equation*}
      \begin{split}
        [1]-\varphi\varphi^* &= \Gamma*\prod_{j=1}^d ([1]-Z_jZ_j^*) \\
        & = \Gamma_1*\prod_{j\neq j_1}
        ([1]-Z_jZ_j^*) + \Gamma_2*\prod_{j\neq j_2} ([1]-Z_jZ_j^*),
      \end{split}
    \end{equation*}
    where $Z_j(x) = \psi_j(x)$;
  \item[$($TF$\,)$] There are unitary colligations $\Sigma_a$ and
    $\Sigma_{na}$ in the ample and nearly ample setting respectively,
    such that $\varphi$ has transfer function representations
    \begin{equation*}
        \varphi = W_{\Sigma_a} = W_{\Sigma_{na}};
    \end{equation*}
  \item[$($vN-a$\,)$] $\varphi \in H^\infty(X, {\mathcal
      K}_{\Lambda^a,\mathcal H}) = H^\infty(X, {\mathcal
      K}_{\Lambda^{na},\mathcal H})$ and for every representation
    $\pi$ which is strictly contractive on the auxiliary test function
    $\sigma_1$ $($respectively, the auxiliary test functions
    $\sigma_{\lambda_1}, \sigma_{\lambda_2})$, or contractive on these
    and either strongly or weakly continuous, we have
    $\|\pi(\varphi)\| \leq 1$;
  \item[$($vN-B$\,)$] $\varphi \in H^\infty(X, {\mathcal
      K}_{\Lambda^a,\mathcal H}) = H^\infty(X, {\mathcal
      K}_{\Lambda^{na},\mathcal H})$ and for every representation
    $\pi$ which is a strict / strongly continuous / weakly continuous
    Brehmer representation with respect to $\Lambda^a$\break
    $($respect\-ively, $\Lambda^{na})$, we have $\|\pi(\varphi)\| \leq
    1$.
  \end{enumerate}
\end{theorem}

In particular, the theorem implies that $\Lambda^a$ and $\Lambda^{na}$
are always equivalent preorderings.

\begin{proof}[Proof of Theorem~\ref{thm:realization_IV}]
  The idea is very much like that in the proof of
  Theorem~\ref{thm:2-var-realization}.  As we did there, we use the
  embedding $\xi$ of $X$ in the polydisk given in
  Lemma~\ref{lem:id_HLH_w_pdisk_subalg} and
  Theorem~\ref{thm:agler-pick_interpolation} to get a function
  $\tilde\varphi$ in $H^\infty_1(\mathbb D^d,\mathcal{L(H)})$, which
  when restricted to the image of $X$ under $\xi$ pulls back to
  $\varphi$.  Applying the polydisk realization theorem
  (Theorem~\ref{thm:d-var-polydisk_real}) to $\tilde\varphi$, we
  obtain the equivalence of the various statements in the theorem over
  the polydisk, and then pulling back to $X$ the result follows.
\end{proof}

The Hilbert space $\mathcal H$ is arbitrary in the last theorem, so we
get the following corollary, generalizing Brehmer's theorem and a
result in~\cite{MR2502431}, via its obvious specialization to
$H^\infty(\mathbb D^d,\mathcal{L(H)})$.  Compare with
Theorem~\ref{thm:dilation-theorem}.

\begin{corollary}
  \label{cor:generalized-Brehmers-theorem}
  Suppose $\Psi = \{\psi_1,\ldots \psi_d\}$ is a collection of test
  functions, $\Lambda^{na}$ is a standard nearly ample preordering
  with maximal elements $\lambda^m_1$ and $\lambda^m_2$.  Let $\pi:
  A({\mathcal K}_{\Lambda^{na},\mathcal{L(H)}}) \to \mathcal{L(K)}$ be
  a Brehmer representation.  Then $\pi$ is completely contractive and
  so dilates to a completely contractive representation $\tilde\pi$ of
  the $C^*$-envelope of $A({\mathcal K}_{\Lambda^{na},\mathcal{L(H)}})$ 
  with the property that it is the only completely positive agreeing
  with $\tilde\pi| A({\mathcal K}_{\Lambda^{na},\mathcal{L(H)}})$.
\end{corollary}

\goodbreak

\section{Some applications}
\label{sec:some-applications}

We give some more or less immediate applications of the material
presented.  For example, the following, which is the main result
of~\cite{MR2502431}, is the last corollary applied to the polydisk.

\begin{corollary}
  \label{cor:generalized-Brehmers-theorem-for-polydisk}
  Let $\Psi = \{z_1,\ldots z_d\}$ be the coordinate functions on
  $\mathbb D^d$, $\Lambda^{na}$ a standard nearly ample preordering
  with maximal elements $\lambda^m_1$ and $\lambda^m_2$.  Let $\pi:
  A(\mathbb D^d,\mathcal{L(H)}) \to \mathcal{L(K)}$ be
  a Brehmer representation.  Then the contractions $\pi(z_j)$ dilate
  to a commuting unitary operators.
\end{corollary}

Another interesting corollary of the realization theorem is a sort of
weak form of the rational dilation property.

\begin{corollary}
  \label{cor:n-contr-implies-cc}
  Let $\Psi = \{\psi_1,\ldots \psi_d\}$, $d\geq 2$, is a collection of
  test functions, $\Lambda$ a standard ample preordering.  Then any
  $2^{d-2}$-contractive representation of $\ALH$ or weakly continuous
  $2^{d-2}$-contractive representation of $\HLH$ is completely
  contractive.
\end{corollary}

\begin{proof}
  This follows from Lemma~\ref{lem:contr-aux-t-fns-are-B-reps}, the
  last corollary and the fact that the auxiliary test functions with a
  standard nearly ample preordering are in $H^\infty(X, {\mathcal
    K}_{\Lambda,\mathbb C^{2^n}})$, $n\leq d-2$.
\end{proof}

On the polydisk, we then get the following.

\begin{corollary}
  \label{cor:n-contr-implies-cc-improved}
  For the polydisk $\mathbb D^d$, $d\geq 2$, any $2^{d-2}$-contractive
  representation of $A(\mathbb D^d,\mathcal{L(H)})$ or weakly
  continuous $2^{d-2}$-contractive representation of $H^\infty(\mathbb
  D^d,\mathcal{L(H)})$ is complete\-ly contractive.
\end{corollary}

This implies that $2$-contractive representations of the tridisk
algebra are completely contractive.  In particular, examples like that
due to Parrott of contractive representations of this algebra which
are not completely contractive can only fail to be so by failing to be
$2$-contractive.

\begin{corollary}
  \label{cor:3disk-not-cc=not-2-contr}
  Any representation of $A(\mathbb D^3,\mathcal{L(H)})$ or weakly
  continuous representation of\break $H^\infty(\mathbb
  D^3,\mathcal{L(H)})$ which is $2$-contractive is complete\-ly
  contractive.  Equivalently, any such representation which is not
  completely contractive must fail to be $2$-contractive.
\end{corollary}

Here are a couple of other examples involving two test functions.  Let
$X$ be an annulus $\mathbb A$ with out boundary the unit circle
$\mathbb T$ and inner boundary $r\mathbb T$ for some $0< r < 1$.
Choose for test functions the set $\Psi = \{\psi_1(z) = z, \psi_2(z) =
r/z\}$.  By what we have shown (see also~\cite{MR2595740}),
contractive representations of this algebra are automatically
completely contractive, and so the rational dilation property holds.
The rational dilation problem for the annulus was originally solved by
Agler in~\cite{MR794373} (see~\cite{2013arXiv1305.4272D} for an
alternate proof).  It can be shown that although $\AL$ and $A(\mathbb
A)$ have different norms, they are the same algebra, and in fact as
operator algebras they are completely boundedly
equivalent~\cite{MR862189} (see also~\cite{MR2449098}).  One might
naively expect that this would give yet another approach to solving
this problem, but unfortunately it does not.  Indeed, the same
phenomenon occurs for multiply connected planar
domains~\cite{MR862189}.

To perhaps better illustrate what might happen, consider instead the
disk $\mathbb D$ with test functions $\Psi = \{\psi_1(z) = z^2,
\psi_2(z) = z^3\}$.  This is an example of a constrained algebra,
since $\AL$ consists of functions with first derivative equal to $0$
at the origin.  This algebra differs from the subalgebra of the disk
algebra of functions with derivative $0$ at the origin (that is,
$\mathbb C + z^2 A(\mathbb D)$).  For the latter, one can find
examples of contractive representations which are not completely
contractive (ie, rational dilation fails)~\cite{2013arXiv1305.4272D},
while for $\AL$, by what we have shown, it holds.  Indeed, for $\AL$,
a representation which maps the two test functions to contractions
(satisfying the obvious constraint that $\pi(\psi_1)^3 =
\pi(\psi_2)^2$) is completely contractive by
Theorem~\ref{thm:2-var-realization} and
Corollary~\ref{cor:n-contr-implies-cc-improved}, while there are
examples of such representations of the constrained subalgebra of the
disk algebra which are not even contractive (much as in the
Kaijser-Varopoulos example)~\cite{2013arXiv1305.4272D}.

\section{Conclusion}
\label{sec:conclusion}

When we have more than two test functions over some set $X$, there
will be preorderings with their associated algebras for which we
cannot say much beyond what is in our initial realization theorem,
Theorem~\ref{thm:realization_I}.  In particular, we do not know if the
auxiliary test functions can be extended to matrix valued functions in
our algebra, as we have in either the classical case or the cases of
ample and nearly ample preorderings.  We also wonder if there are
other types of preorderings other than the ample and nearly ample ones
which are equivalent.

It would be nice to know more concretely what the auxiliary test
functions are, particularly over polydisks.  The knowledge of this
could provide a key tool in resolving a number of questions regarding
the connection between Schur-Agler class in the classical sense and
$H^\infty$ over these domains, and hence resolving some of the
mysteries surrounding these algebras.  One immediate question is
whether for $d>3$ there are $(2^{d-2}-1)$-contractive representations
which are not completely contractive (that is, are the bounds in
Corollary~\ref{cor:n-contr-implies-cc-improved} sharp?).  Another is
whether contractive representations of the polydisk algebras
$A(\mathbb D^d)$ with $d \geq 3$ are necessarily completely bounded.

It would also be useful to know a norming set of boundary
representations for the Agler algebra in the classical setting.  Over
the tridisk any representation sending the coordinate functions to a
commuting tuples of unitaries are included (if we ignore
irreducibility), but as we also saw in
\S\ref{subsec:some-bound-repr-Agler-alg}, other sorts of
representations are also there.  In the concrete examples given, these
all send the coordinate functions to nilpotent partial isometries.  As
was mentioned, there will be examples of boundary representations
which come from neither tuples of commuting unitaries, nor commuting
tuples of nilpotent operators, but these in between cases are somewhat
mysterious.  Perhaps, up to unitary conjugation, they always send the
coordinate functions to tuples of commuting partial isometries?
However, from Kaijser-Varopoulos example, one sees that this alone is
not enough.  Note that the boundary representations coming from
commuting unitaries are $1$ dimensional, and our examples of nilpotent
boundary representations are all finite dimensional.  Is it the case
that there is an upper bound to the dimension of all boundary
representations?  Though this seems unlikely, there is no obvious
proof.  There will be Schur-Agler class functions which peak on the
boundary representations.  Are they related to the polynomials from
which these examples are initially drawn?  In any case, for the
nilpotent representations, these will presumably be polynomials of the
same degree as the order of nilpotency.

What happens with the unit ball in $\mathbb C^d$, $d>1$?  It is well
known that the unit ball of the Drury-Arveson algebra does not
coincide with the unit ball of $H^\infty$ of this
space~\cite{MR1866874}.  While one must be careful applying the
results here in this setting since the test function is vector valued,
it is still intriguing to speculate what algebras one might obtain
with powers of the Drury-Arveson kernel.

Finally, the resemblance of results from real algebra to those
presented here is striking.  Are there some deeper connections?  For
example, could one use the techniques here to find, at least in some
circumstances, a proof of such results as Schm\"udgen's theorem?

\end{document}